\documentclass[a4 paper, 10 pt]{amsart}
\usepackage[utf8]{inputenc}

\usepackage{geometry}
 \usepackage[english]{babel}
 \usepackage{latexsym}
 \usepackage{amssymb}
 \usepackage{amsmath}
 \usepackage{amsfonts}
 \usepackage{graphics}
 \usepackage{graphicx}
 \usepackage{pictex}
 \usepackage{dcpic}
 \usepackage{tikz-cd}
 \usepackage{mathrsfs}
 \usepackage[mathscr]{euscript}
 \usepackage{braket} 
 \usepackage{amsthm}
 \usepackage{stmaryrd}
 \usepackage[most]{tcolorbox}

 \usepackage{url}
 \usepackage{marvosym}
 \usepackage{enumitem}
 
 \usepackage[T1]{fontenc}
 \usepackage{calligra}

\newcommand{\longmono}{\mbox{\;$\lhook\joinrel\longrightarrow$\;}}

\newcommand{\longepi}{\mbox{\;$\relbar\joinrel\twoheadrightarrow$\;}}

\usepackage{bm} 

 \newtheorem{teor}{Theorem}[section]
 \newtheorem{prop}[teor]{Proposition}
 
 \newtheorem{lem}[teor]{Lemma}
 \newtheorem{assumption}[teor]{Assumption}

 \newtheorem*{teoA}{Theorem 1}
 \newtheorem*{teoB}{Theorem 2}

 \theoremstyle{definition}
 \newtheorem{defi}[teor]{Definition}
 \newtheorem{rmk}[teor]{Remark}

 \AfterEndEnvironment{teor}{\noindent\ignorespaces}
 \AfterEndEnvironment{defi}{\noindent\ignorespaces}

\usepackage[OT2,T1]{fontenc}
\DeclareSymbolFont{cyrletters}{OT2}{wncyr}{m}{n}
\DeclareMathSymbol{\Sha}{\mathalpha}{cyrletters}{"58}

\DeclareMathOperator{\End}{End}

\DeclareMathOperator{\Hom}{Hom}
\DeclareMathOperator{\Gal}{Gal}
\DeclareMathOperator{\GL}{GL}
\DeclareMathOperator{\SL}{SL}
\DeclareMathOperator{\Ima}{im}

\DeclareMathOperator{\Pic}{Pic}
\DeclareMathOperator{\ord}{ord}
\DeclareMathOperator{\Cor}{Cor}
\DeclareMathOperator{\Res}{Res}

\DeclareMathOperator{\Sel}{Sel}
\DeclareMathOperator{\Iw}{Iw}
\DeclareMathOperator{\cris}{cris}
\DeclareMathOperator{\loc}{loc}
\DeclareMathOperator{\cl}{cl}
\DeclareMathOperator{\dR}{dR}
\DeclareMathOperator{\et}{\text{\'et}}

\DeclareMathOperator{\frob}{Frob}
\DeclareMathOperator{\Red}{red}
\DeclareMathOperator{\Sch}{Sch}
\DeclareMathOperator{\Sets}{\mathscr{S}\textit{et}}

\DeclareMathOperator{\Fil}{Fil}
\DeclareMathOperator{\AJ}{AJ}
\DeclareMathOperator{\BK}{BK}
\DeclareMathOperator{\Sym}{Sym}
\DeclareMathOperator{\Spec}{Spec}
\DeclareMathOperator{\Spf}{Spf}

\DeclareMathOperator{\LL}{\text{\calligra{L}}}
\DeclareMathOperator{\cyc}{cyc}
\DeclareMathOperator{\CH}{CH}

\DeclareMathOperator{\id}{id}
\DeclareMathOperator{\crys}{crys}
\DeclareMathOperator{\W}{\mathcal{W}}
\DeclareMathOperator{\V}{\mathcal{V}}
\DeclareMathOperator{\ab}{ab}

\usepackage{xcolor}
\usepackage{graphicx}
\newcommand{\bmu}{\mbox{$\raisebox{-0.59ex}
		{$l$}\hspace{-0.18em}\mu\hspace{-0.88em}\raisebox{-0.98ex}{\scalebox{2}
			{$\color{white}.$}}\hspace{-0.416em}\raisebox{+0.88ex}
		{$\color{white}.$}\hspace{0.46em}$}{}}

\usepackage{setspace}

\begin{document}

	\title{Generalized Heegner cycles and $p$-adic $L$-functions\\in a quaternionic setting}
	\author{Paola Magrone}
	\address{Dipartimento di Matematica, Universit\`a di Genova, Via Dodecaneso 35, 16146 Genova, Italy}
\email{magrone@dima.unige.it}

\begin{abstract} 
In a recent paper, Castella and Hsieh proved 
results for Selmer groups associated with Galois representations attached to newforms twisted by Hecke characters of an imaginary quadratic field. These results are obtained under the so-called Heegner hypothesis that the imaginary quadratic field satisfies with respect to the level of the modular form. In particular, Castella and Hsieh prove the rank 0 case of the Bloch–Kato conjecture for $L$-functions of modular forms in their setting. The key point of the work of Castella and Hsieh is a remarkable link between generalized Heegner cycles and $p$-adic $L$-functions. In this paper, several of the results of Castella–Hsieh are extended to a quaternionic setting, that is, the setting that arises when one works under a “relaxed” Heegner hypothesis. More explicitly, we prove vanishing and one-dimensionality results for Selmer groups. Crucial ingredients in our strategy are Brooks's results on generalized Heegner cycles over Shimura curves.
\end{abstract}
\maketitle
\section{Introduction}

Starting from a modular form $f \in S_k^{\text{new}}(\Gamma_0(N))$ with $k \geq 4$, we consider the two-dimensional $p$-adic Galois representation
$
V_f : G_{\mathbb{Q}} \longrightarrow \GL_2(F)
$
associated with $f$, where $G_{\mathbb{Q}}$ is the absolute Galois group $\Gal(\overline{\mathbb{Q}}/\mathbb{Q})$ of $\mathbb{Q}$ and $F$ is a $p$-adic field containing the Fourier coefficient of $f$. We are interested in studying the representation 
\[V_{f,\chi} := V_f(k/2) \otimes \chi,\]
defined as the twist of the self-dual Tate twist $V_f(k/2)$ of $V_f$ by a Galois character $\chi$.

In \cite{BDP}, Bertolini, Darmon and Prasanna introduced a distinguished collection of algebraic cycles, called \emph{generalized Heegner cycles}, coming from graphs of isogenies between elliptic curves, lying in the product of the Kuga--Sato variety with a power of a fixed elliptic curve. 
Later, Castella and Hsieh constructed in \cite{CH} Euler systems for generalized Heegner cycles; they proved, among other results, a theorem that establishes, under suitable hypotheses, the vanishing of the Selmer group $\text{Sel}(K,V_{f,\chi})$ associated with the representation $V_{f,\chi}$ as a $G_K:= \Gal(\overline{\mathbb{Q}}/K)$-representation, where $K$ is an imaginary quadratic field satisfying the so-called \emph{Heegner hypothesis} relative to $N$:
\begin{itemize}
    \item all the primes dividing $N$ split in $K$.
\end{itemize}
This proves the Bloch--Kato conjecture in this case:
\begin{displaymath}
\dim_F \text{Sel}(K,V_{f,\chi}) = \text{ord}_{s=k/2} L(f,\chi,s) = 0.
\end{displaymath}
Building upon results from \cite{BDP}, the proof by Castella--Hsieh is based on a link between generalized Heegner cycles and a certain $p$-adic $L$-function attached to $f$, and on a generalization of Kolyvagin's method. We emphasize that the Heegner hypothesis is essential in \cite{CH}.

\subsubsection*{The quaternionic setting: relaxing the Heegner hypothesis}

What happens if we want to weaken the Heegner hypothesis? More explicitly, we would like to generalize the work of Castella and Hsieh to the case of an imaginary quadratic field $K$ that does not satisfy the classical Heegner hypothesis, but instead satisfies the following \emph{generalized Heegner hypothesis} relative to the level $N$ of the modular form $f$:
\begin{itemize}
    \item no prime factor of $N$ ramifies in $K$, if a prime $\ell$ is inert in $K$ then $\ell^2$ does not divide $N$ and the number of prime factors of $N$ that are inert in $K$ is \emph{even}.
\end{itemize}
In this setting, we cannot work with Kuga--Sato varieties over classical modular curves, as we are not able to construct Heegner cycles on these varieties without the Heegner hypothesis. The right substitutes for modular curves in this context are Shimura curves, so it is natural to work with Kuga--Sato varieties fibered over Shimura curves.

In \cite{Brooks}, Brooks introduced a collection of generalized Heegner cycles on a Kuga--Sato variety over a Shimura curve $Sh$, coming from graphs of isogenies between abelian surfaces. The curve $Sh$ has the form of a quotient of the complex upper half plane under the action of a group that is determined by an order in an indefinite quaternion algebra over $\mathbb{Q}$ and parametrizes abelian surfaces.
Brooks proved results that generalize (some of) those in \cite{BDP} to this quaternionic setting.

Building on the work of Brooks, our goal is to generalize to a quaternionic context the key result of \cite{CH} relating their $p$-adic $L$-function to the system of generalized Heegner classes. As said before, this is a crucial point for the proof of the vanishing of the Selmer group. We construct a system of generalized Heegner cycles on the Kuga--Sato variety over the Shimura curve $Sh$ and a $p$-adic $L$-function defined as a $p$-adic measure given as a sum of values of a variation of $f$, as a modular form over our Shimura curve, at certain CM abelian surfaces. With these ingredients at hand, we will prove results on the Selmer group $\Sel(K, V_{f,\chi})$, generalizing several of Castella--Hsieh's results.

It is worth remarking that we expect the results of this paper to play a key role in the proof of a generalization of Castella's specialization results (\cite{Cas}) for Howard's big Heegner points in Hida families (\cite{How}) to the quaternionic big Heegner points introduced by Longo and Vigni (\cite{LV}). We plan to address this question in a future project.

\subsubsection*{Main results}

First of all, we fix some notation. Let $f \in S^{\text{new}}_{k}(\Gamma_0(N))$ be a newform of weight $k=2r+2 \geq 4$ and level $N$. Fix an odd prime $p \nmid N$ and a field embedding
$i_p : \overline{\mathbb{Q}} \hookrightarrow \mathbb{C}_p$, where $\mathbb{C}_p$ is the completion of the chosen algebraic closure of $\mathbb{Q}_p$ . Let
$F$ be a finite extension of $\mathbb{Q}_p$ containing the image of the Fourier coefficients of $f$ under $i_p$ and let $K$ be an imaginary
quadratic field of discriminant $D_K$ and ring of integers $\mathcal O_K$ in which $p$ splits as $p\mathcal O_K = \mathfrak{p}\overline{\mathfrak{p}}$ splits, with $\mathfrak{p}$ determined by $i_p$. Let $\chi : \Gal(K_{c_0p^{\infty}}/K) \rightarrow \mathcal{O}^{\times}_F$ be a locally algebraic anticyclotomic character of infinity type $(j, -j)$ and conductor $c_0p^s\mathcal{O}_K$ (see section \ref{char}). Denote by  $V_{f,\chi} := V_f(k/2) \otimes \chi$ the twist of $V_{f}(k/2)$ by $\chi$ seen as a representation of $\Gal(\overline{\mathbb{Q}}/K)$, by $L(f,\chi,s)$ the associated Rankin $L$-series and by $\Sel(K,V_{f,\chi})$ the Block--Kato Selmer group associated with $V_{f,\chi}$ and $K$. Assume that:
\begin{itemize}[noitemsep]
	\item[1.] $p \nmid 2N\phi(N^+)$ (where $\phi$ is Euler's function);
	\item[2.] $c_0$ is prime to $N$, i.e., the conductor of $\chi$ is prime to $N$;
	\item[3.] either $D_K > 3$ is odd or $8 \mid D_K$;
	\item[4.] $p = \mathfrak{p} \overline{\mathfrak{p}}$ splits in $K$.
\end{itemize}
Moreover, assume that $K$ satisfies the generalized Heegner hypothesis relative to $N$ as described above, and factor $N$ as $N = N^+N^-$, where $N^+$ is a product of primes that split in $K$ and $N^-$ is a (necessarily square-free) product of an \emph{even} number of primes that are inert in $K$.

Our first theorem on Selmer groups, which corresponds to Theorem \ref{selmer-1}, is a vanishing result.

\begin{teoA}
\label{A}
	If $f$ is $p$-ordinary and $L(f,\chi,k/2)\neq 0$, then
	\begin{displaymath}
	\dim_F\Sel(K,V_{f,\chi}) = 0.
	\end{displaymath}
\end{teoA}
Denote by $\epsilon(V_{f,\chi})$ the sign of the functional equation of $L(f,\chi,s)$. Our second theorem on Selmer groups, which corresponds to Theorem \ref{selmer-2}, is a one-dimensionality result.
\begin{teoB}
\label{B}
	If $\epsilon(V_{f,\chi})=-1$ and $z_{\chi} \neq 0$, then
	\begin{displaymath}
	\Sel(K,V_{f,\chi}) = F\cdot z_{\chi}.
	\end{displaymath}
\end{teoB}
In the statement above, $z_{\chi}$ is a suitable cohomology class in $H^1(K,V_{f,\chi})$ that comes from an Euler system of generalized Heegner classes.

Let us conclude this introduction by briefly sketching the structure of the paper.
In \S 2 we
introduce the Shimura curve $Sh$ we will work with and also the notions of modular forms and $p$-adic modular forms over Shimura curves.

In \S 3 we review Serre--Tate theory and study deformations of abelian surfaces, which we will use to get power series expansions at ordinary CM points for modular forms over Shimura curves (for which $q$-expansions are not available).

In \S 4 we define an analytic anticyclotomic $p$-adic $L$-function $\LL_{f, \psi}$ as a measure on the Galois group $\Gal(K_{p^{\infty}}/K)$ with values in the ring of Witt vectors $\W=W(\overline{\mathbb{F}}_p)$, where $K_{p^{\infty}}= \cup_{n} K_{p^n}$ for $K_{p^n}$ the ring class field of conductor $p^{n}$ of $K$, $\overline{\mathbb{F}}_p$ is an algebraic closure of the field $\mathbb{F}_p$ with $p$ elements and $\psi$ is an anticyclotomic Hecke character of infinity type $(k/2,-k/2)$ and conductor $c_0\mathcal{O}_K$ with $(c_0, pN^+) = 1$. We close the chapter with an interpolation formula for our $p$-adic $L$-function, which we will use later to obtain the reciprocity law of \S 7, relating the value of $\LL_{f, \psi}$ at $\phi$ to the central critical value $L(f,\chi,k/2)$, where $\phi$ is an anticyclotomic Hecke character of infinity type $(n,-n)$ with $n \geq 0$ and $p$-power conductor such that $\chi = \psi \phi$.

In \S 5, following Brooks, we introduce a family of generalized Heegner cycles on the generalized Kuga--Sato variety over our Shimura curve $Sh$. More precisely, these cycles live in a Chow group of the generalized Kuga--Sato variety $\mathcal{X}_r = \mathcal{A}^r \times A^r$, where $\mathcal{A}$ is the universal object of the fine moduli problem associated with $Sh$ and $A$ is a fixed abelian surface with CM by $K$. Then we apply a $p$-adic Abel--Jacobi map to obtain cohomology classes from generalized Heegner cycles. In this way, we construct a system of generalized Heegner classes associated with $f$ and $\chi$, and indexed by fractional ideals of $K$, for which we prove compatibility properties.

In \S 6 we establish a relation between values of $\LL_{f, \psi}$ at Galois characters $\phi$ of infinity type $(-k/2-j,k/2+j)$ and Bloch--Kato logarithms of generalized Heegner cycles associated with $\chi$ of infinity type $(j,-j)$, with $ -k/2 < j < k/2$, where $\chi = \psi^{-1} \phi^{-1}$. This relation, for which we refer to Theorem \ref{GZ}, has the form
\[
\LL_{f,\psi}(\phi) = \text{(something)} \cdot \braket{\log_{\mathfrak{p}}(z_{\chi}), *},
\]
where ``something'' is an explicit non-zero coefficient that is comparatively less important than the main terms in the formula. The key ingredient to establish this Gross--Zagier type formula is the work of Brooks: we link our $p$-adic $L$-function to the differential operator $\theta = t \frac{d}{dt}$ on the Serre--Tate coordinates and then we use Brooks's results to obtain a formula suitably relating $\theta$ to our generalized Heegner cycles.

Finally, in \S 7 we use the previous formula and the interpolation property to establish, under a $p$-ordinarity assumption on $f$, a reciprocity law relating the analytic $p$-adic $L$-function $\LL_{f,\psi}$ to an algebraic $p$-adic $L$-function obtained as a sort of image of an Iwasawa cohomology class $\boldsymbol{z}_f \in H^1\bigl(K_{p^{\infty}},V_f(k/2)\bigr)$, obtained as an inverse limit of generalized Heegner classes, under a big logarithm map.
This reciprocity law and the costruction of an anticyclotomic Euler system associated with generalized Heegner classes, combined with an extension of Kolyvagin's method for anticyclotomic Euler systems developed in \cite{CH}, lead to the proof of Theorem A. The proof of Theorem B rests instead on the extension of Kolyvagin's method applied to another anticyclotomic Euler system associated with generalized Heegner classes.

\subsubsection*{Acknowledgements}
The contents of this paper are part of my PhD thesis at Universit\`a degli Studi di Genova. I am very grateful to my advisor, Stefano Vigni, for suggesting this problem, for his support and encouragement and for making every
meeting a pleasant and friendly time. I especially thank him for his patience, his enthusiasm and, above all, for the faith he always had in me.
I would also like to thank Francesc Castella, Ming-Lun Hsieh and Matteo Longo for enlightening conversations and correspondence.

\subsubsection*{Notation}

If $F$ is a number field or a local field, whose ring of integers will be denoted by $\mathcal O_F$, we fix an algebraic closure $\overline{F}$ of $F$ and write $G_F$ for the absolute Galois group $\Gal(\overline{F}/F)$ of $F$.

We denote by $\mathbb{A}_F$ the adele ring of a number field $F$ and by $\hat{F}$ the ring of finite adeles of $F$.

For any prime number $p$, we fix an immersion $i_p : \overline{\mathbb{Q}} \hookrightarrow \mathbb{C}_p$, where $\mathbb{C}_p$ is the completion of the chosen algebraic closure of $\mathbb{Q}_p$.

For an imaginary quadratic field $K$ and an integer $n\geq1$, we denote by $K_n$ the ring class field of $K$ of conductor $n$; in particular, $K_1$ is the Hilbert class field of $K$. We denote also by $K^{\text{ab}}$ the maximal abelian extension of $K$.

For an integer $n\geq1$ we write $\zeta_{n}$ for the primitive $n$-th root of unity $e^{\frac{2\pi i}{n}}\in\mathbb C^\times$ and we denote by $\chi_{\cyc}$ the $p$-adic cyclotomic character.

As usual, the notation for pull-back will be a high asterisk $^{*}$ while the notation for push-forward will be a down asterisk $_{*}$. Finally, unadorned tensor products are always taken over $\mathbb{Z}$.

\section{Shimura curves, CM points and modular forms} \label{shimura-chapter}

In this section we introduce the Shimura curves, attached to quaternion algebras over $\mathbb Q$, we will work with, we construct certain CM points on Shimura curves and we define modular forms and $p$-adic modular forms on Shimura curves. 

\subsection{Shimura curves}
Let $K$ be an imaginary quadratic field of discriminant $D_K$ and consider a positive integer $N$, which will be the level of our modular form, such that $(N,D_K)=1$. Suppose that $K$ satisfies the {\bf generalized Heegner hypothesis} relative to $N$:
\begin{itemize}
\item no prime factor of $N$ ramifies in $K$, if a prime $\ell$ is inert in $K$ then $\ell^2$ does not divide $N$ and the number of prime factors of $N$ that are inert in $K$ is \emph{even}.
\end{itemize}
Factor $N$ as a product $N=N^+ N^-$ where $N^+$ is a product of primes that split in $K$ and $N^-$ is a (necessarily square-free)
product of (an even number of) primes that are inert in $K$.

Let $B$ be the indefinite rational quaternion algebra over $\mathbb{Q}$ of discriminant $D=N^-$ and fix a prime $p\nmid N$ that splits in $K$ and $B$.
%
%
%
Fix isomorphisms $\Phi_{\ell} : B_{\ell} \cong M_2(\mathbb{Q}_{\ell})$ for each prime $\ell \nmid D$ and
denote by $\mathcal{O}_B$ a maximal order of $B$ such that each $\Phi_{\ell}$ induces an isomorphism $\mathcal{O}_B \otimes \mathbb{Z}_{\ell} \cong M_2(\mathbb{Z}_{\ell})$. 
Fix also an isomorphism $\Phi_{\infty}: B \otimes \mathbb{R} \cong M_2(\mathbb{R})$.  
Consider the map
\[
\pi_{N^+} : \hat{\mathcal{O}}_{B}^{\times} \twoheadrightarrow\prod_{\ell \mid N^+} (\mathcal{O}_B \otimes \mathbb{Z}_{\ell})^{\times} \cong \prod_{\ell \mid N^+} \GL_2(\mathbb{Z}_{\ell}) \twoheadrightarrow \GL_2(\mathbb{Z}/N^+\mathbb{Z}).
\]
Denote by $\hat{\Gamma}_{1,N^+}$ the open compact subgroup of $\hat{\mathcal{O}}_{B}^{\times}$ composed of the elements $b \in \hat{\mathcal{O}}_{B}^{\times}$ such that 
$\pi_{N^+}(b) \in \left\lbrace \left(\begin{smallmatrix}
	* & * \\ 0 & 1
\end{smallmatrix} \right) \in GL_2(\mathbb{Z}/N^+\mathbb{Z})
\right\rbrace.$
Consider
the space of double cosets
\[
X_{N^+} := B^{\times}
 \backslash \mathcal{H}^{\pm} \times \hat{B}^{\times}  / \hat{\Gamma}_{1,N^+},
\]
where $\mathcal{H}^{\pm}= \mathbb{C} - \mathbb{R}$ is the disjoint union of the upper and lower complex half planes.
Here $\hat{\Gamma}_{1,N^+}$ acts naturally on the right on $\hat{B}^{\times}$ by right multiplication, while $B^{\times}$ acts on the left on $\hat{B}^{\times}$ through the diagonal embedding $B \hookrightarrow\hat{B}$ and on $\mathcal{H}^{\pm}$ under the fixed isomorphism $\Phi_{\infty}$ by the usual linear fractional transformations
\[
\left( \begin{smallmatrix}
a & b \\ c & d
\end{smallmatrix} \right) \cdot \tau := \frac{a\tau +b}{c\tau + d}.
\]
This is the Shimura curve associated with the Shimura datum $G(\mathbb{Q}) = B^{\times}$, $X= \mathcal{H}^{\pm}$ and $K=\hat{\Gamma}_{1,N^+}$. For more details, see \cite{Mi90}.

Because $B$ is indefinite, there is a bijection
\[
X_{N^+} \cong \mathcal{H}/\Gamma_{1,N^+},
\]
where $\mathcal{H}$ is the classical upper half plane and $\Gamma_{1,N^+}$ is the subgroup of matrices in
$\Phi_{\infty}((\hat{\Gamma}_{1,N^+} \cap B)^{\times})$ of determinant $1$ (\cite[\S 1.3]{HMT}). This bijection endows $X_{N^+}$ with a Riemann surface structure and gives, as a consequence, an analytic description of $X_{N^+}$.

The coset space $X_{N^+}$ admits a model over $\mathbb{Q}$, which is the fine moduli scheme classifying abelian surfaces with quaternionic multiplication by $\mathcal{O}_B$ and certain level structures.

\begin{defi}
	Let $S$ be a $\mathbb{Z}[1/D]$-scheme. An \textbf{abelian surface with quaternionic multiplication by $\mathcal{O}_B$} (abelian surface with QM, for short) \textbf{over $S$} is a pair $(A,i)$ where 
	\begin{enumerate}[noitemsep]
		\item $A$ is an abelian scheme $A/S$ of relative dimension 2;
		\item $i$ is an optimal inclusion $i: \mathcal{O}_B \hookrightarrow \End_S(A)$ giving an action of $\mathcal{O}_B$ on $A$.
	\end{enumerate}
	A morphism of abelian surfaces with QM is a morphism of abelian surfaces that respects the action of $\mathcal{O}_B$.
\end{defi} 
Abelian surfaces with quaternionic multiplication are often called \emph{false elliptic curves}.
\begin{defi}
Let $N^+ >0$ be an integer prime to $D$.
A \textbf{level $V_1(N^+)$-structure}, or an arithmetic level $N^+$ structure, on a QM abelian surface $(A, i)$ is an inclusion
\[
\bmu_{N^+} \times \bmu_{N^+} \longmono A[N^+]
\]
of group schemes over $S$, commuting with the action of $\mathcal{O}_B$, where $\bmu_{N^+}$ denotes the group scheme of $N^+$th roots of unity. The action of $\mathcal{O}_B$ on the left hand side is via the isomorphism $\mathcal{O}_B \otimes \mathbb{Z}/N^+\mathbb{Z} \cong M_2(\mathbb{Z}/N^+\mathbb{Z})$ induced by the chosen isomorphisms $\Phi_{\ell} : B_{\ell} \cong M_2(\mathbb{Q}_{\ell})$, through which one has $\mathcal{O}_B \otimes \mathbb{Z}_{\ell} \cong M_2(\mathbb{Z}_{\ell})$ for each $\ell \mid N^+$.
\end{defi}
A morphism of QM abelian surfaces with $V_1(N^+)$-level structure is a morphism of QM abelian surfaces that respects the level structures.

If $A$ is an abelian surface over an algebraically closed field $k$, a $V_1(N^+)$-level structure can be thought
of as an orbit of full level $N^+$ structures, i.e., isomorphisms
$
\mathcal{O}_B \otimes \mathbb{Z}/N^+\mathbb{Z} \cong A[N^+]
$ 
commuting with the action of $\mathcal{O}_B$, under the natural action of the subgroup $\bigl\{\left(\begin{smallmatrix}
* & * \\ 0 & 1
\end{smallmatrix} \right) \in GL_2(\mathbb{Z}/N^+\mathbb{Z})
\bigr\}$ of $GL_2(\mathbb{Z}/N^+\mathbb{Z})$. See \cite[\S 2.2]{Brooks} for details.

The moduli problem of QM abelian surfaces with $V_1(N^+)$-level structure is representable, as asserted by

\begin{teor} \label{moduli-thm}
For $N^+ > 3$, the moduli problem that assigns to a $Z[1/DN^+]$-scheme $S$ the set of isomorphism classes of QM-abelian surfaces over $S$ with $V_1(N^+)$-level
structure is representable by a smooth proper $\mathbb{Z}[1/DN^+]$-scheme $X$.
\end{teor}
For details, see \cite[\S 2.2. and \S 2.3]{Brooks}, \cite[\S 2]{Buz} or \cite[\S 2 and \S 3]{Kas}.
The complex points of $X$ are naturally identified with the compact Riemann surface $X_{N^+} \cong \mathcal{H}/\Gamma_{1,N^+}$.
The $\mathbb{Z}[1/DN^+]$-scheme $X$ from Theorem \ref{moduli-thm} is called the \textbf{Shimura curve of level $V_1(N^+)$} associated to the indefinite quaternion algebra $B$ and we will denote it by $X_{N^+}$, using the same notation for the scheme, the Riemann surface and the double coset space.

We denote by $\pi: \mathcal{A} \rightarrow X_{N^+}$ the universal object of the moduli problem, i.e. the \emph{universal QM abelian surface} over $X_{N^+}$.
For each geometric point $x : \Spec(L) \rightarrow \mathcal{A}$, the fiber $\mathcal{A}_x := \mathcal{A} \times_x \Spec(L)$ is an abelian surface with QM by $\mathcal{O}_B$ and $V_1(N^+)$-level structure defined over $L$, representing the isomorphism class that corresponds to the point $x$.


\subsection{Hecke operators} \label{Heckeop}

The Shimura curve $X_{N^+}$ comes equipped with a ring of Hecke correspondences, which can be introduced by using the adelic description of $X_{N^+}$. See, for example, \cite[\S 1.5]{HMT}; here the construction is for the Shimura curve relative to level structures ``of type $\Gamma_0$'', but it can be done also in our case.
In terms of abelian surfaces the Hecke operators acts in the following way.
For a prime $\ell$, a QM abelian surface $A$ over a field $k$ of characteristic prime to $\ell$ has $\ell + 1$ cyclic $\mathcal{O}_B$-submodules annihilated by $\ell$. Denote them by $C_0, \dots , C_{\ell}$ and consider the isogenies $\psi_i : A \twoheadrightarrow A/C_i$ of QM abelian surfaces. If $\nu_A$ is a $V_1(N^+)$-level structure on $A$ and $\ell \nmid N^+$, then $\psi_i$ induces a $V_1(N^+)$-level structure $ \nu_A \circ \psi_i$ on $A/C_i$. If $\ell \nmid N^+D$, the ``good'' Hecke operator $T_{\ell}$ can be described by
\[
T_{\ell} (A,\iota_A,\nu_A)= \sum_{i=0}^{\ell} (A/C_i,\iota_i,\nu_i).
\]
For more details, see \S \ref{K-S} and \cite[\S 3.6]{Brooks}.

\subsection{Igusa tower}

We are interested in working with $p$-adic modular forms over our Shimura curve, which are defined analogously to Katz's generalized $p$-adic modular forms. Therefore we want to work on a cover of the ordinary locus of the Shimura curve.

Fix a prime $p \nmid N^+D$. Since ${X_{N^+}}$ is a scheme over $\mathbb{Z}[1/N^+D]$, it can be viewed as a scheme over $\mathbb{Z}_{(p)}$.  
For simplicity, denote by $Sh$ the curve
${X_{N^+}}_{/\mathbb{Z}_{(p)}}$. Since $Sh$ is a fine moduli scheme for QM abelian surfaces over $\mathbb{Z}_{(p)}$-schemes with level structures, there is a universal abelian surface $\mathcal{A} \rightarrow Sh$, which is the one associated with $X_{N^+}$ but tensored with $\mathbb{Z}_{(p)}$ over $\mathbb{Z}[1/DN^+]$.

Recall that a QM abelian surface $A$ over a field $k$ of characteristic $p$ is said to be \emph{ordinary} if $A[p](\overline{k}) \cong (\mathbb{Z}/p\mathbb{Z})^{2}$, and \emph{supersingular} otherwise. Indeed, a QM abelian surface in characteristic $p$, is either ordinary or supersingular; equivalently, it is isogenous either to a product of ordinary elliptic curves or to a product of supersingular elliptic curves, respectively.
Consider the \emph{ordinary locus} $Sh^{\ord}$ of $Sh$, i.e., the locus on which the Hasse invariant does not vanish, that is the scheme obtained by removing the supersingular points of $Sh$ in the fiber at $p$, which are those points which correspond in the moduli interpretation to abelian surfaces which have supersingular reduction modulo $p$. See \cite{Kas} for details about the ordinary locus and the Hasse invariant.

Let $\mathcal{A}^{\ord} \rightarrow Sh^{\ord}$ be the universal ordinary QM abelian surface over $Sh^{\ord}$, that is the fiber product $\mathcal{A}^{\ord}= \mathcal{A} \times_{Sh} Sh^{\ord}$.
Consider the functor $ I_n : \Sch_{/{Sh^{\ord}}} \rightarrow \Sets$ that takes an $Sh^{\ord}$-scheme $S$ to the set of closed immersions $\bmu_{p^n} \times \bmu_{p^n} \hookrightarrow \mathcal{A}^{\ord}[p^n]$ of finite flat group schemes over $S$ respecting the $\mathcal{O}_B$-action. This functor is representable by a scheme ${I_n}_{/Sh^{\ord}}$. Then ${I_n}_{/\mathbb{Z}_{(p)}}$ classifies quadruples $(A,i,\nu_{N^+},\nu_{p^n})$, where $A$ is an abelian surface, $\iota$ a quaternionic multiplication, $\nu_{N^+}$ a $V_1(N^+)$-level structure and $\nu_{p^n}$ an $\mathcal{O}_B$-immersion $\bmu_{p^n} \times \bmu_{p^n} \hookrightarrow A[p^n]$.
There is a tower
\[
\dots \longrightarrow I_{n+1} \longrightarrow I_n \longrightarrow I_{n-1} \longrightarrow \cdots.
\]
Consider the formal scheme $I_{/\mathbb{Z}_{(p)}}:= \varprojlim_n {I_n}_{/\mathbb{Z}_{(p)}}$.
This formal scheme parametrizes compatible sequences of isomorphism classes of quadruples $(A,i,\nu_{N^+},\nu_{p^n})$, where $A$ is an ordinary abelian surface, $\iota$ a quaternionic multiplication and $\nu_{N^+}$, $\nu_{p^n}$ respectively a $V_1(N^+)$ and $V_1(p^n)$ level structures.
But a sequence of compatible $V_1(p^n)$-level structures is the same as a $V_1(p^{\infty})$-level structure, that is an immersion $\nu_{p^{\infty}} : \bmu_{p^{\infty}} \times \bmu_{p^{\infty}} \hookrightarrow A[p^{\infty}]$.
Therefore this tower 
parametrizes isomorphism classes of quadruples $(A,i,\nu_{N^+},\nu_{p^{\infty}})$.
There is a bijection
$
I(\mathbb{C}) \cong \varprojlim_n X_{N^+p^n}(\mathbb{C}),
$
between complex points of $I$ and compatible sequences $\{ x_n \}_n$ of complex points $x_n = (A,i,\nu_{N^+},\nu_{p^n}) \in X_{N^+p^n}(\mathbb{C})$.

See \cite{Hi04} for details about Igusa schemes in the case of modular curves (\S 6.2.12) and more in general for Shimura varieties (Ch. 8). See also \cite{Hi09}.

\subsection{CM points on Shimura curves}

In this section we will construct a collection of CM points in our Shimura curves, i.e., points that correspond to abelian surfaces with complex multiplication, indexed by fractional ideals of orders in an imaginary quadratic field.
We denote again by $Sh$ the curve $X_{N^+}$ seen as a scheme over $\mathbb{Z}_{(p)}$ and by $\mathcal{A} \rightarrow Sh$ the universal abelian surface over $Sh$.

\subsubsection{Abelian surfaces with QM and CM over $\mathbb{C}$}

\begin{teor}
	Let $(A,i)$ be an abelian surface with QM by $\mathcal{O}_B$ over $\mathbb{C}$. Then either
	\begin{enumerate}[noitemsep]
		\item $A$ is simple and $\End^0(A) := \End(A) \otimes \mathbb{Q} = B$, or
		\item $A$ is not simple, $A \sim E^2$ is isogenous to the product of an elliptic curve $E$ with CM by an imaginary quadratic field $K$ which embeds in $B$ and $\End^0(A) \cong M_2(K)$.
	\end{enumerate}
\end{teor}
In particular, we are interested in the second case of the previous theorem. Abelian surfaces with QM that satisfy that second condition are said to have \textbf{complex multiplication} (CM for short) by $K$.
Suppose that $(A,i)$ is an abelian surface over $\mathbb{C}$ with QM by $\mathcal{O}_B$ and $V_1(N^+p^n)$-level structure. Then the ring
\[
\End_{\mathcal{O}_B}(A) := \bigl\lbrace f \in \End(A) \mid f \circ i(b) = i(b) \circ f\ \text{for all}\ b \in \mathcal{O}_B \bigr\rbrace 
\]
is either $\mathbb{Z}$ or an order in an imaginary quadratic field $K$. If $K$ is an imaginary quadratic field and $\End_{\mathcal{O}_B}(A) = \mathcal{O}_c$, where $\mathcal{O}_c$ is the order of conductor $c$ in $\mathcal{O}_K$, then $A$ is said to have 
complex multiplication
by $\mathcal{O}_c$  
and the point $P=[(A,i)] \in X_{N^+p^n}(\mathbb{C})$ is said to be a \textbf{CM point of conductor $c$}.

\subsubsection{Products of CM elliptic curves} \label{prodell}

Start with an elliptic curve $E$ over $\mathbb{C}$ with complex multiplication by $\mathcal{O}_K$, take $E:= \mathbb{C} / \mathcal{O}_K$. Consider on $E$ a $\Gamma_1(M)^{\text{arit}}$-level structure given by a morphism
\[
\mu_{M} : \bmu_{M} \longmono E[M],
\]
where $M>3$ is an integer prime to $D$.
Consider now the self product $A := E \times E$ that is an abelian surface over $\mathbb{C}$; then its endomorphism ring is $\End(A) \cong M_2(\End(E)) \cong M_2(\mathcal{O}_K)$.
Since $K$ splits $B$, because of the generalized Heegner hypothesis, we can embed $K$ in $B$ and choose a basis $\left\lbrace b_1, b_2 \right\rbrace $ of $B$ over $K$ with $b_1, b_2 \in \mathcal{O}_B$. Then we have an an immersion $B \hookrightarrow M_2(K) = M_2(\End^0(E)) = \End^0(A)$ such that $\mathcal{O}_B \hookrightarrow M_2(\mathcal{O}_K) = M_2(\End(E)) = \End(A)$. See \cite[\S 2]{Mi79}. Hence $ \iota : \mathcal{O}_B \hookrightarrow \End(A)$ is a quaternionic multiplication for $A$.

Consider the isomorphism $i_K : B \otimes_{\mathbb{Q}} K \cong M_2(K)$ induced by $\iota$ and put \[e:= i_K^{-1}\left( \left( \begin{smallmatrix}
1 & 0 \\ 0 & 0
\end{smallmatrix}
\right) \right) \in B \otimes_{\mathbb{Q}} K,\] which is an idempotent such that $e^*=e$. Then the decomposition of $A$ is induced by $A = eA \oplus (1-e)A = \left( \begin{smallmatrix}
1 & 0 \\ 0 & 0
\end{smallmatrix}
\right)
A \oplus \left( \begin{smallmatrix}
0 & 0 \\ 0 & 1
\end{smallmatrix}
\right)
A \cong E \times E$ (multiplication by $\alpha := \left( \begin{smallmatrix}
0 & 1 \\ 1 & 0
\end{smallmatrix}
\right)$ gives an isomorphism $ eA \ \displaystyle_{\cong}^{\cdot \alpha} \ (1-e)A$).
So $A[M] = eA[M] \oplus (1-e)A[M]$. The choice of a level structure on $eA[M] = E[M]$ induces a $V_1(M)$-level structure on $A[M]$, because of the request of compatibility of the level structure with respect to the action of $\mathcal{O}_B$ (and consequently of $B$). See also the last lines of \cite[\S 2.2]{Brooks}. Thus, the fixed level structure $\mu_{M}$ on $E$ induces a $V_1(M)$-level structure 
\[
\nu_{M} : \bmu_{M} \times \bmu_{M} \longmono A[M]
\]
on $A$.
Hence, starting from a $\Gamma_1(N^+p^{n})^{\text{arit}}$-level structure on $E$, we obtain a quadruple $(A,\iota,\nu_{N^+}, \nu_{p^{n}})$ which determines a CM point in $X_{N^+p^n}(\mathbb{C})$.
Starting from a $\Gamma_1(N^+p^{\infty})^{\text{arit}}$-level structure on $E$, we obtain a quadruple $(A,\iota,\nu_{N^+}, \nu_{p^{\infty}})$ which can be seen as a compatible sequence of CM points in the Shimura tower $X_{N^+p^n}$.

Start now with the elliptic curve $E_c = \mathbb{C} / \mathcal{O}_c$ over $\mathbb{C}$ with complex multiplication by an order $\mathcal{O}_c$ of $K$, with $(c,N)=1$.
The isogeny 
\[ 
\begin{split}
E = \mathbb{C} / \mathcal{O}_K &\longepi E =\mathbb{C} / \mathcal{O}_c\\
\overline{z} &\longmapsto \overline{c z},
\end{split}
\]
induces an isogeny
\[
\phi_c : A = E \times E \longepi A_c:= E_c \times E_c
\]
of complex abelian surfaces. 
Take on $A_c$ the quaternionic multiplication $\iota_c : \mathcal{O}_B \hookrightarrow \End(A_c)$ determined by compatibility with $\phi_c$:
\[
\iota_c(b) (\phi_c(a)) = \phi_c(\iota(b) a),
\]
for any $b \in B$, $a \in A$.
As before, a $\Gamma_1(M)^{\text{arit}}$-level structure $\mu_{c,M}$ on $E_c$, with $M$ prime to $D$, induces a $V_1(M)$-level structure $\nu_{c,M}$ on $A_c$.
Therefore, starting from a $\Gamma_1(N^+p^{n})^{\text{arit}}$-level structure on $E_c$, we obtain a quadruple $(A_c,\iota_c,\nu_{N^+}, \nu_{p^{n}})$ which determines a CM point of conductor $c$ in $X_{N^+p^n}(\mathbb{C})$.
Starting from a $\Gamma_1(N^+p^{\infty})^{\text{arit}}$-level structure on $E$, we obtain a quadruple $(A,i,\nu_{c,N^+}, \nu_{c,p^{\infty}})$ which can be seen as a compatible sequence of CM points in the Shimura tower $X_{N^+p^n}$.

Note that the isogeny $\phi_c$ doesn't necessarily respect the chosen level structures if $p$ divides $c$.

\subsubsection{The action of $\Pic(\mathcal{O}_c)$}

Denote by $\Pic(\mathcal{O}_c)$ the Picard group of the order $\mathcal{O}_c$ of conductor $c$ of an imaginary quadratic field $K$, that is
\[
\Pic(\mathcal{O}_c) = K^{\times} \backslash \hat{K}^{\times} / \hat{\mathcal{O}_c}^{\times} = I_c(\mathcal{O}_c) / P_c(\mathcal{O}_c),
\]
where 
$I_c(\mathcal{O}_c)$ is the group of fractional ideals of $\mathcal{O}_c$ coprime to $c$ and $P_c(\mathcal{O}_c)$ is the subgroup of $I_c(\mathcal{O}_c)$ of principal fractional ideals.

Consider a quadruple $(A,\iota,\nu_{N^+},\nu_{p^{\infty}})$ where $(A,\iota)$ is a QM abelian surface with CM by $\mathcal{O}_c$. There is an action of $\Pic(\mathcal{O}_c)$ on the isomorphism classes of these quadruples, defined by
\[
\mathfrak{a} \star (A,\iota,\nu_{N^+},\nu_{p^{\infty}}) := (A_{\mathfrak{a}},\iota_{\mathfrak{a}},\nu_{\mathfrak{a},N^+},\nu_{\mathfrak{a},{p^{\infty}}}),
\]
where the representative $\mathfrak{a}$ is chosen to be integral and prime to $N^+pc$.
Here $A_{\mathfrak{a}} := A/A[\mathfrak{a}]$, where $A[\mathfrak{a}]$ is the subgroup of the elements of $A$ that are killed by all the endomorphisms in $\mathfrak{a}$.
The quaternionic multiplication $\iota_{\mathfrak{a}}$ and the level structures $\nu_{\mathfrak{a},N^+}, \nu_{\mathfrak{a},p^{\infty}}$ are induced by the ones of $A$. Denote by $\varphi_{\mathfrak{a}}$ the quotient isogeny $A \twoheadrightarrow A/ A[\mathfrak{a}]$, that is an isogeny of degree $N(\mathfrak{a})^2 = (\# \mathcal{O}_c/\mathfrak{a})^2$ and so prime to $N^+p$; define
\[
\iota_{\mathfrak{a}} : \mathcal{O}_B 
\longmono \End(A_{\mathfrak{a}})
, \quad
b 
\longmapsto \Big(\varphi_{\mathfrak{a}}(x) \mapsto \varphi_{\mathfrak{a}}\bigl(\iota(b)(x)\bigr)\Big)
\]
and
\[
\nu_{\mathfrak{a},N^+} : \bmu_{N^+} \times \bmu_{N^+} 
\stackrel{\nu_{N^+}}\longmono A[N^+] \stackrel{\varphi_{\mathfrak{a}}}{\cong} A_{\mathfrak{a}}[N^+], \quad
\nu_{\mathfrak{a},p^{\infty}} : \bmu_{p^{\infty}} \times \bmu_{p^{\infty}} 
\stackrel{\nu_{p^{\infty}}}\longmono A[p^{\infty}] \stackrel{\varphi_{\mathfrak{a}}}{\cong} A_{\mathfrak{a}}[p^{\infty}].
\]
See \cite[\S 2.5]{Brooks}.

If $\sigma_{\mathfrak{a}}\in \Gal(K_c/K)$ corresponds to $\mathfrak{a} \in \Pic(\mathcal{O}_c)$ through the classical Artin reciprocity map, then by Shimura's reciprocity law there is an equality
\[
(A,\iota,\nu_{N^+},\nu_{p^{\infty}})^{\sigma_{\mathfrak{a}}} = 
\mathfrak{a} \star (A,\iota,\nu_{N^+},\nu_{p^{\infty}}).
\]
See, again, \cite[\S 2.5]{Brooks}.

\subsubsection{Construction of CM points} \label{CMpoints}

We want to introduce CM points indexed by ideals of orders of $K$.

Take an element $[\mathfrak{a}] \in \Pic(\mathcal{O}_c)$ and choose the representative $\mathfrak{a}$ to be integral and prime to $N^+pc$.
Consider the elliptic curve $E_{\mathfrak{a}} := \mathbb{C}/\mathfrak{a}^{-1}$ with the $\Gamma_1(N^+p^{\infty})^{\text{arith}}$-level structure defined in \cite[\S 2.3 and \S 2.4]{CH}. Put $A_{\mathfrak{a}}:= E_{\mathfrak{a}} \times E_{\mathfrak{a}}$, which, by the theory of complex multiplication, is an abelian surface defined over the ring class field $K_c$ of $K$ of conductor $c$. The abelian surface $A_{\mathfrak a}$ has QM by $\mathcal{O}_B$ and we can consider the quadruple 
\[ x(\mathfrak{a}) :=(A_{\mathfrak{a}},i_{\mathfrak{a}},\nu_{\mathfrak{a},N^+},\nu_{\mathfrak{a},p^{\infty}}), \]
where $\nu_{\mathfrak{a},N^+},\nu_{\mathfrak{a},p^{\infty}}$ are the level structures induced by the ones of $E_{\mathfrak{a}}$ (as in \S \ref{prodell}). 
We write $x(c) =(A_{c},i_c,\nu_{c,N^+},\nu_{c,p^{\infty}})$ when $\mathfrak{a}=\mathcal{O}_c$.

We have already used the notation above in the previous section for the action of the Picard group  on a quadruple: the reason is that there is an equality
\[
\mathfrak{a} \star (A_{c},i_c,\nu_{c,N^+},\nu_{c,p^{\infty}}) = (A_{\mathfrak{a}},i_{\mathfrak{a}},\nu_{\mathfrak{a},N^+},\nu_{\mathfrak{a},p^{\infty}}). 
\]
Indeed, $\mathfrak{a} \star A_c = A_c/A_c[\mathfrak{a}] \cong E_c/E_c[\mathfrak{a}] \times E_c/E_c[\mathfrak{a}]$ because $\mathfrak{a} \subseteq \mathcal{O}_c \hookrightarrow M_2(\mathcal{O}_c)$ acts diagonally and the level structures are induced by the isogeny $\varphi_{\mathfrak{a}} : A_c \twoheadrightarrow A_c/ A_c[\mathfrak{a}]$ that is the product of the isogeny $E_c \twoheadrightarrow E_c/ E_c[\mathfrak{a}]$.

\subsection{Modular forms on Shimura curves}

We recall here the definitions and some properties of modular forms and $p$-adic modular forms on Shimura curves. The references are \cite{Kas}, \cite{Brooks}, \cite{EdVP}, \cite{Hi04}.

\subsubsection{Geometric modular forms on Shimura curves}\label{geom-mod-forms}

We will need integrality conditions only at $p$, so we define modular forms over algebras $R$ over the localization $\mathbb{Z}_{(p)}$ of $\mathbb{Z}$ at the prime ideal generated by $p$.
Let $(\pi:A \rightarrow \Spec(R), \iota)$ be a QM abelian surface over a $\mathbb{Z}_{(p)}$-algebra $R$. Then $\pi_* \Omega_{A/R}$, where $\Omega_{A/R}$ is the bundle of relative differentials, inherits an action of $\mathcal{O}_B$ which tensored with the scalar action of $\mathbb{Z}_p$ gives an action of $M_2(\mathbb{Z}_p)$ on $\pi_* \Omega_{A/R}$. Write $\underline{\omega}_{A/R}$ for $e\pi_* \Omega_{A/R}$. If $\mathcal{A} \rightarrow Sh$ is the universal QM abelian surface, then $\mathcal{A} \otimes R \rightarrow Sh \otimes R$ is the universal object for $Sh \otimes R$. In  the  particular  case $ \pi : \mathcal{A} \otimes R \rightarrow Sh \otimes R$ of the universal QM abelian surface over a $\mathbb{Z}_p$-algebra $R$,  we just write $\underline{\omega}_R$ for $e\pi_* \Omega_{\mathcal{A} \otimes R/Sh \otimes R}$.

In analogy with the case of elliptic modular forms (see, for example, \cite[\S 1]{BDP}, in particular equation $(1.1.1)$), we give a geometric definition \`a la Katz for modular forms on $Sh$ over a $\mathbb{Z}_{(p)}$-algebra $R$.
For a nice exposition of Katz modular forms in the case of modular curves, see \cite[Chapter 1]{Go}. The geometric definition for modular forms on Shimura curves is due to Kassaei, see \cite[\S 4.1]{Kas}. We closely follow \cite{Brooks}.

\begin{defi}
	A \textbf{modular form with respect to $B$ of weight $k \in \mathbb
	Z$ and level $V_1(N^+)$ over $R$} is a global section of $\underline{\omega}^{\otimes_k}_R$, i.e., an element of $H^0(Sh \otimes R,\underline{\omega}^{\otimes_k}_R)$.
	We denote by $M_k(Sh,R)$ the space of modular forms with respect to $B$, of weight $k \in \mathbb
	Z$ and level $V_1(N^+)$ over $R$.
\end{defi}

Alternatively, one can define modular forms in the following ways.

\begin{defi} 
	Let $R_0$ be  an $R$-algebra. A {\bf test object} is a quadruple $(A/R_0, \iota,\nu,\omega)$ consisting of a QM abelian surface $A$ over $R_0$, a $V_1(N^+)$-level structure $\nu$ on $A$, and a non-vanishing global section of $\underline{\omega}_{A/R_0}$. 
	
	Two test objects $(A/R_0, \iota, \nu,\omega)$ and $(A'/R_0,\iota',\nu',\omega')$ over $R_0$ are {\bf isomorphic} if there is an isomorphism $(A/R_0,\iota,\nu) \rightarrow (A'/R_0,\iota',\nu')$, of QM abelian surfaces with $V_1(N^+)$-level structure, pulling $\omega'$ back to $\omega$. 
	
	A {\bf modular form of weight $k$ and level $V_1(N^+)$ over $R$} is a rule $F$ that assigns to every isomorphism class of test objects $(A/R_0,\iota, \nu,\omega)$ over an $R$-algebra $R_0$ a value $F(A/R_0,\iota, \nu,\omega) \in R_0$ such that 
	\begin{itemize}[leftmargin=*]
		\item (compatibility with base change) if $\varphi:R_0 \rightarrow R'_0$ is a map of $R$-algebras, inducing $\varphi : A \rightarrow A \otimes_{\varphi} R_0'$, then
		$
		F\big((A/R_0, \iota,\nu) \otimes_{\varphi} R'_0 ,\omega\big) = \varphi\big(F(A/R_0,\iota, \nu,\varphi^*(\omega))\big);
		$
		\item (weight condition) for any $\lambda\in R_0$, one has
		$
		F(A/R_0,\iota,\nu,\lambda\omega) =\lambda^{-k}F(A/R_0,\iota, \nu,\omega).
		$ 
	\end{itemize}
\end{defi}

\begin{defi}
	A {\bf modular form of weight $k$ and level $V_1(N^+)$ over $R$} is a rule $G$ that, for any $R$-algebra $R_0$, assigns to any isomorphism class of QM abelian surfaces over $R_0$ with $V_1(N^+)$-level structure $(A/R_0,\iota, \nu)$, a translation-invariant section of $\underline{\omega}^{\otimes_k}_{A/R_0}$, subject to the following base-change axiom: if $\varphi: R_0 \rightarrow R'_0$ is a map of $R$-algebras one has
	\[
	G((A/R_0,\iota, \nu) \otimes_{\varphi} R'_0) = \varphi^* G(A/R_0, \iota, \nu).
	\]
\end{defi}

Given a modular form as in the third definition, we get a modular form as in the first definition by taking the section assigned to the universal QM abelian surface with level structure $\mathcal{A}\otimes R \rightarrow Sh \otimes R$. 
This is an equivalence because $\mathcal{A} \otimes R$ is universal. The last two definitions are related by \[G(A,\iota, \nu) = F(A,\iota,\nu,\omega) \omega^{\otimes_k},\] where $\omega$ is any translation-invariant global section. This expression is independent of the choice of $\omega$.

\subsubsection{$p$-adic modular forms on Shimura curves} \label{pmodform}

Let $R$ be a $p$-adic ring (for $p$-adic ring we mean a complete and separated,
with respect to the p-adic topology, $\mathbb{Z}_p$-algebra). Define the space $V_p(N^+, R)$ of \textbf{$p$-adic modular forms} of level $V_1(N^+)$ by
\[
V_p(N^+, R) 
:= \varprojlim_m H^0 (\varprojlim_n I_n \otimes R/p^mR, \mathcal{O}_{\varprojlim_n I_n \otimes R/p^mR})
\cong \varprojlim_m \varinjlim_n H^0(I_n \otimes R/p^mR, \mathcal{O}_{I_n \otimes R/p^mR}),
\]
where $\mathcal{O}$ is the structure sheaf. When $n =0$
 one can take in the limit the coordinate ring of the affine scheme obtained from $Sh \otimes R/p^mR$ by deleting the supersingular points, that is $H^0((Sh \otimes R/p^mR)^{\ord}, \mathcal{O}_{(Sh \otimes R/p^mR)^{\ord}})$. If $m=0$, we take $H^0(({I_n}_{/R}, \mathcal{O}_{{I_n}_{/R}})$.
Thus elements in $V_p (N^+, R)$ are formal functions on the tower $I_n$, i.e., $f \in V_p (N^+, R)$ is a rule that assigns to
each quadruple $(A,\iota,\nu_{N^+}, \nu_{p}^{\infty})$, where $(A,\iota,\nu_{N^+}, \nu_{p}^{\infty})$ is a QM abelian surface over an $R$-algebra $R_0$ with $V_1(N^+p^{\infty})$-level structure, a value $f(A,\iota,\nu_{N^+}, \nu_{p}^{\infty}) \in R_0$, which depends only on the isomorphism class and that is compatible with base changes. We say that a $p$-adic modular form $f$ is of weight $k \in \mathbb{Z}_p$, if for every $u \in \mathbb{Z}^{\times}_p$, we have
\[
f (A,\iota,\nu_{N^+}, \nu_{p}^{\infty}) = u^{-k} f (A,\iota,\nu_{N^+}, \nu_{p}^{\infty}u),
\]
where $(A,\iota,\nu_{N^+}, \nu_{p}^{\infty})$ is a QM abelian surface over an $R$-algebra with $V_1(N^+p^{\infty})$-level structure.

If $f$ is a modular form with respect to $B$ of weight $k$ and level $V_1(N^+)$ over $R$ as in \ref{geom-mod-forms}, then we can see it as a $p$-adic modular form $\hat{f}$ as follows. 
The $V_1(N^+p^{\infty})$-level structure on $A/R_0$ determines a point $P \in eT_pA_0^t(k)$, where $A_0$ is the reduction $\text{mod}\ p$ of $A$.
A point $P \in eT_pA_0^t(k)$ determines a differential $\omega_P \in \underline{\omega}_{A/R_0}$.
Indeed, there is an isomorphism
\[
T_p(A_0^t) \cong \Hom_{\mathbb{Z}_p}(\hat{A},\hat{\mathbb{G}}_m).
\]
So, taking the homorphism $\alpha_P$ corresponding to the point $P$, one can consider the pull-back $\omega_P := \alpha_P^*(dT/T) \in \underline{\omega}_{\hat{A}/R_0} = \underline{\omega}_{A/R_0}$, of the standard differential $dT/T$ of $\hat{\mathbb{G}}_m$.
See \cite[\S 3.3]{Ka} or the proof of \cite[Lemma 4.2]{Brooks} for details.
One can define 
\[
\hat{f}(A,\iota,\nu_{N^+}, \nu_{p}^{\infty}) := f(A,\iota,\nu_{N^+}, \omega_P).
\]
It follows from the definition that if $f$ is a geometric modular form of weight $k$ and level $V_1(N^+)$, then $\hat{f}$ is a $p$-adic modular form of weight $k$ and level $V_1(N^+)$.

\subsubsection{Jacquet--Langlands correspondence}

The Jacquet--Langlands correspondence establishes a Hecke-equivariant bijection between automorphic
forms on $\GL_2$ and automorphic
forms on multiplicative groups of quaternion algebras. In our setting, this can be stated as a correspondence between classical modular forms and quaternionic modular forms.
\begin{teor}[Jacquet--Langlands]
There is a canonical (up to
scaling) isomorphism
\[
S_k(\Gamma_1(N), \mathbb{C})^{D\emph{-new}} \xrightarrow[\emph{JL}]{\cong} M_k(Sh,\mathbb{C}),
\]
where $\Gamma_1(N)$ is the standard congruence group 
$
\Gamma_1(N) := \left\lbrace A \in \SL_2(\mathbb{Z}) \mid A \equiv \left( \begin{smallmatrix}
1 & * \\ 0 & 1
\end{smallmatrix}
\right) \mod N \right\rbrace,
$
and $S_k(\Gamma_1(N), \mathbb{C})^{D\emph{-new}}$ is the space of classical cuspidal eigenforms with respect to $\Gamma_1(N)$, of weight $k$ and that are new at $D$.
This bijection is compatible with the Hecke-action and the Atkin-Lehner involutions on each
of the spaces.
\end{teor}
In particular, to each eigenform $f \in S_k (\Gamma_1(N),\mathbb{C}) ^{D\text{-new}}$ 
corresponds a unique (up to scaling) quaternionic form $f_B = \text{JL}(f)\in M_k(Sh,\mathbb{C})$ having the same Hecke eigenvalues as $f$ for the good Hecke operators $T_{\ell}$ for $(\ell,D)=1$ and the Atkin--Lehner involutions. 
Anyway, one can normalize $f_B$.

More precisely, if we start from an eigenform $f \in S_k (\Gamma_1(N),\mathbb{C}) ^{D\text{-new}}$ with Nebentypus $\varepsilon_f$ with respect to the action  of $\Gamma_0(N)$, where $\Gamma_0(N) := \left\lbrace A \in \SL_2(\mathbb{Z}) \mid A \equiv \left( \begin{smallmatrix}
* & * \\ 0 & *
\end{smallmatrix}
\right) \mod N \right\rbrace$, the Jacquet--Langlands correspondence asserts the existence of a holomorphic
function $f_B$ on the upper half plane, determined only up to a scalar multiple, such that $f_B$ is a modular form for the discrete subgroup $\Gamma_{1,N^+}$ of $\GL_2(\mathbb{R})$, of weight $k$, with the same eigenvalues as $f$ for the good Hecke operators and with Nebentypus $\varepsilon_f$ for the action of $\Gamma_{0,N^+}$, where $\hat{\Gamma}_{0,N+}$ is the open compact subgroup of $\hat{\mathcal{O}}_{B}^{\times}$ composed of the elements $b \in \hat{\mathcal{O}}_{B}^{\times}$ such that 
	$\pi_{N^+}(b) \in \left\lbrace \left(\begin{smallmatrix}
	* & * \\ 0 & *
	\end{smallmatrix} \right) \in GL_2(\mathbb{Z}/N^+\mathbb{Z})
	\right\rbrace$, and $\Gamma_{0,N^+}^+$ is the subgroup of matrices in
	$\Phi_{\infty}((\hat{\Gamma}_{0,N^+} \cap B)^{\times})$ of determinant $1$.
In particular, if we start from a classical modular form for $\Gamma_0(N)$ we obtain a quaternionic modular form with trivial Nebentypus with respect to the action of $\Gamma_{0,N^+}$.
	
The function $f_B$ gives rise canonically, as in \cite[\S 2.7]{Brooks}, to a modular form in the sense of the geometric definition seen before, i.e., to a section of $\underline{\omega}_{\mathbb{C}} = e\pi_*\Omega_{\mathcal{A}\otimes\mathbb{C}/ Sh \otimes \mathbb{C}}$. If $f \in S_k (\Gamma_1(N),F) ^{D\text{-new}}$, i.e. its Fourier coefficients lie in the ring of integers $\mathcal{O}_F$ of a number field $F$, then the choice of $f_B$ is ambiguous up to multiplication by a unit in $\mathcal{O}_F[1/N]$. See again \cite[\S 2.7]{Brooks}.

\section{Deformation theory and $t$-expansions for modular forms} 

In order to associate with modular forms over Shimura curves power series expansions at CM points, we are interested in deformation theory. In particular, Serre--Tate deformation theory provides us with a way to do this. Thus, in this section we will study the deformation theory of QM abelian surfaces, which is closely related to the deformation theory of elliptic curves, as is well explained in \cite{Buz}. Then we will define power series expansions for modular forms on Shimura curves.


\subsection{Serre--Tate deformation theory}

Following \cite{Ka}, we introduce the Serre--Tate deformation theory for ordinary abelian varieties, which provides a way to attach power series expansions to modular forms on Shimura curves, replacing classical $q$-expansions for ``elliptic'' modular forms that are not available in our case.

Fix an algebraically closed field $k$ of characteristic
$p > 0$ (for our goals, we can take $k = \overline{\mathbb{F}}_p$) and consider an \emph{ordinary} abelian variety $A$ over $k$. Recall that an abelian variety $A$ over $k$ is said to be ordinary if $A[p](k) \cong (\mathbb{Z}/p\mathbb{Z})^{\dim(A)}$. Let $A^t$ be the dual abelian
variety, which is isogenous to $A$ and hence ordinary too. Consider the Tate modules 
$
T_pA := \varprojlim_n A[p^n](k),\ T_pA^t := \varprojlim_n A^t[p^n](k)
$
of $A$ and $A^t$. Because of the ordinarity assumption on $A$, $T_pA$ and $T_pA^t$ are free $\mathbb{Z}_p$-modules of rank $g: = \dim A = \dim A^t$.

\begin{defi}
	If $R$ is an artinian local ring with maximal ideal $\mathfrak{m}_R$ and residue field $k$, a \textbf{deformation} of $A$ to $R$ is
	an abelian scheme $\mathcal{A}$ over $R$ 
	together with an identification $\mathcal{A} \times_R k \cong A$.
\end{defi}

Following a construction due to Serre and Tate, we attach to such a deformation a $\mathbb{Z}_p$-bilinear form
\[
q(\mathcal{A}/R;-,-) : T_pA \times T_pA^t \longrightarrow \widehat{\mathbb{G}}_m(R) =  1 + \mathfrak{m}_R,
\]
where $\widehat{\mathbb{G}}_m := \Spf\bigl(k[T,S]/(TS-1)\bigr)$ is the completion of the multiplicative group scheme $\mathbb{G}_m:=\Spec\bigl(k[T,S]/(TS-1)\bigr)$ over $k$. 
This bilinear map is constructed from the Weil pairings
\[
e_{p^n} : A[p^n] \times A^t[p^n] \longrightarrow \bmu_{p^n}
\]
of $k$-group schemes, as defined in \cite{Oda}. These pairings come from Cartier duality for the $p$-divisible groups $A[p^{\infty}]$ and $A^t[p^{\infty}]$ (duality of abelian schemes is compatible with Cartier duality).
Here $\bmu_{p^n}$ is $\Spec\bigl(k[T]/(T^{p^n}-1)\bigr)$, 
the $k$-group scheme of $p^n$-th roots of unity, which can be seen inside $\mathbb{G}_m$ through the map $k[T,S]/(TS-1) \rightarrow k[T]/(T^{p^n}-1)$ defined by $T \mapsto T$ and $S \mapsto T^{p^n-1}$. For each $k$-algebra $R$, $\bmu_{p^n}(R)$ corresponds to the $p^n$-torsion of $\mathbb{G}_m(R)$. 
For the convenience of the reader, we sketch here the construction of the bilinear map $q(\mathcal{A}/R;-,-)$, because we will use it later. Choose an integer $n\geq0$ such that $\mathfrak{m}_R^n=0$. Since $p \in \mathfrak{m}_R$, $\widehat{\mathcal{A}}(R) := \ker \bigl( \mathcal{A}(R) \rightarrow \mathcal{A}(k) = A(k) \bigr)$ is killed by $p^n$.  
Let $P\in A(k)$; for any lift $\tilde{P} \in \mathcal{A}(R)$ of $P$, since $\widehat{\mathcal{A}}(R)$ is killed by $p^n$, we have that $p^n\tilde{P}$ is independent of the choice of the lift $\tilde{P}$. The existence of a lift $\tilde{P} \in \mathcal{A}(R)$ of $P\in A(k)$ is guaranteed by the smoothness of $\mathcal{A}/R$ (\cite[Corollary 2.13]{Liu}).
Therefore we obtain a map $A(k) \stackrel{``p^n"}{\longrightarrow} \mathcal{A}(R)$. If we take $P \in A[p^n](k)$, then $`` p^n"P \in \widehat{\mathcal{A}}(R)$, so we get
\[
`` p^n" : A[p^n] \longrightarrow \widehat{\mathcal{A}}(R).
\]
Because of the compatibility of the maps $`` p^n"$ when $n\gg0$, we obtain a homomorphism
\[
`` p^n" : T_pA \longepi A[p^n](k) \stackrel{`` p^n"}{\longrightarrow} \widehat{\mathcal{A}}(R)
\]
that is independent of $n$. 

Now, restricting the Weil pairings
\[
e_{p^n} : \widehat{A}[p^n] \times A^t[p^n] \longrightarrow \bmu_{p^n}
\]
for every $n \geq 1$, we obtain a perfect pairing, and then an isomorphism
\[
\widehat{A}[p^n] \stackrel{\cong}{\longrightarrow} \Hom_{\mathbb{Z}_p}\bigl(A^t[p^n],\bmu_{p^n}\bigr)
\]
of $k$-group-schemes. Because of the compatibility of the pairings with respect to $n$, passing to the limit, we deduce an isomorphism
\[
\widehat{A}(k) \stackrel{\cong}{\longrightarrow} \Hom_{\mathbb{Z}_p}\bigl(T_pA^t,\widehat{\mathbb{G}}_{m}\bigr),
\]
of formal groups over $k$.

Since $R$ is artinian, the $p$-divisible group $\mathcal{A}[p^{\infty}]$ has a canonical structure of an extension, as given by
\[
0 \longrightarrow \widehat{\mathcal{A}} \longrightarrow \mathcal{A}[p^{\infty}] \longrightarrow T_pA \times (\mathbb{Q}_p/\mathbb{Z}_p) \longrightarrow 0
\]
of the constant $p$-divisible group $T_pA(k) \times (\mathbb{Q}_p/\mathbb{Z}_p)$ over $R$ by $\widehat{\mathcal{A}}$, which is the unique toroidal formal group over $R$ lifting $\widehat{A}$.
Then the preceding two isomorphisms extend uniquely to isomorphisms of $R$-group schemes
\[
\widehat{\mathcal{A}}[p^n](R) \stackrel{\cong}{\longrightarrow} \Hom_{\mathbb{Z}_p}\bigl(A^t[p^n],\bmu_{p^n}\bigr)
\]
and
\[
\widehat{\mathcal{A}}(R) \stackrel{\cong}{\longrightarrow} \Hom_{\mathbb{Z}_p}\bigl(T_pA^t,\widehat{\mathbb{G}}_{m}\bigr)
\]
(see the proof of \cite[Theorem 2.1]{Ka}), giving pairings
\[
e_{p^n,\mathcal{A}} : \widehat{\mathcal{A}}[p^n](R) \times A^t[p^n] \longrightarrow \bmu_{p^n},
\]
and
\[
e_{\mathcal{A}} : \widehat{\mathcal{A}}(R) \times T_pA^t \longrightarrow \widehat{\mathbb{G}}_{m}.
\]
Finally, the map $q(\mathcal{A}/R;-,-)$ is defined by
\[
q\bigl(\mathcal{A}/R;P,Q^t\bigr) := e_{\mathcal{A}}\bigl(``p^n"P,Q^t\bigr),
\]
for $P \in T_pA$ and $Q^t \in T_pA^t$.

\begin{teor}[Serre--Tate]
	With notation as above, the construction
	\[
	\mathcal{A}/R \longmapsto q(\mathcal{A}/R;-,-) \in \Hom_{\mathbb{Z}_p}\bigl(T_pA\otimes T_p A^t, \widehat{\mathbb{G}}_m(R)\bigr)
	\]
	establishes a bijection 
	\[ 
	\left\lbrace
	\normalsize
	\begin{aligned}
	\emph{isomorphism}\ &\emph{classes of}\ \\
	\emph{deformations of}&\ A/k \ \emph{to}\ R
	\end{aligned} 
	\right\rbrace 
	\stackrel{\cong}{\longrightarrow}
	\Hom_{\mathbb{Z}_p}\big(T_pA( k ) \otimes T_p A^t(k), \widehat{\mathbb{G}}_m(R)\big).
	\]
Furthermore, this correspondence is functorial in $R$, i.e., if $\mathcal{F}$ is the deformation functor from the category $\mathscr{C}$ of artinian local rings with residue field $k$ to the category of sets given by
\[
\mathscr{F} : R \longmapsto \mathscr{F}(R):=\bigl\{ \text{isomorphism classes of deformations of}\ A/k\ \text{to}\ R\bigr\},
\]
then there is an isomorphism of functors
\[ 
\mathscr{F} \sim \Hom_{\mathbb{Z}_p}\bigl(T_pA \otimes T_p A^t, \widehat{\mathbb{G}}_m\bigr).
\]
\end{teor}

\begin{proof} This is \cite[Theorem 2.1, 1) and 2)]{Ka}. \end{proof}
 
The proof of the theorem rests on the fact that there is an equivalence
\[
\left\lbrace
\normalsize
\begin{aligned}
\displaystyle
\text{isomorphism}\ &\text{classes of}\ \\
\text{deformations}\ \mathcal{A}&/R \ \text{of}\ A/k
\end{aligned}
\right\rbrace
\stackrel{\cong}{\longrightarrow}
\Large \left\lbrace
\normalsize
\begin{aligned}
\displaystyle
\text{isomorphism}\ &\text{classes of}\ \\
\text{deformations}\ \mathcal{A}[p^{\infty}]&/R \ \text{of}\ A[p^{\infty}]/k
\end{aligned}
\Large\right\rbrace,
\]
so deforming an ordinary abelian variety $A/k$ is the same as deforming its $p$-divisible group $A[p^{\infty}]$.

Taking inverse limits, we can replace the category of artinian local rings with the category of complete noetherian local rings in the preceding discussion. We can do this because of the compatibility of these correspondences with inverse limits and of the fact that every complete noetherian local ring is the inverse limit of artinian local rings (if $R$ is a complete noetherian local ring with maximal ideal $\mathfrak{m}$, then $R \cong \varprojlim_n R/\mathfrak{m}^n$). However, the procedure for computing the pairings $q_{\mathcal{A}/R}$ only makes sense for artinian local rings.

Passing to complete noetherian local ring is useful because the deformation functor is not representable by an artinian local ring in $\mathscr{C}$ but is pro-representable by a complete noetherian local ring.
Namely, the deformation functor $\mathscr{F}$ is pro-represented by a complete local noetherian ring $\mathcal{R}^u$ that is non-canonically isomorphic to the power series ring $\W[[T_{ij}, 1 \leq i,j \leq g]]$, where $\W := W(k)$ is the ring of Witt vectors over $k$ . Therefore, the functor $\mathscr{F}$ can be seen as a formal scheme $\Spf(\mathcal{R}^u)$. 
Denote by $\widehat{\mathcal{A}}^u/\Spf(\mathcal{R}^u)$ the universal formal deformation of $A/k$, i.e., the formal element of $\mathscr{F}$ corresponding to the identity in $\Hom_{\widehat{\mathscr{C}}}(\mathcal{R}^u, \mathcal{R}^u)$.

Given elements $ P \in T_pA$, $P^t \in T_pA^t$, there is a map
\[ 
\begin{split}
\mathscr{F} &\longrightarrow \widehat{\mathbb{G}}_m\\[1mm]
\mathcal{A}/R &\longmapsto q(\mathcal{A}/R; P, P^t).
\end{split}
\]
If we pick $\mathbb{Z}_p$-bases $\{P_1,\dots,P_g\}$ and $\{P_1^t,\dots,P^t_g\}$ of $T_pA(k)$ and $T_pA^t(k)$, respectively, then we have $g^2$ functions
\[
\begin{split}
t_{ij} : \mathscr{F} &\longrightarrow \widehat{\mathbb{G}}_m\\[1mm]
\mathcal{A}/R &\longmapsto q(\mathcal{A}/R; P_i, P_j^t)
\end{split}
\]
called {\bf Serre--Tate coordinates} and $g^2$ elements $t_{ij}(\widehat{\mathcal{A}}^u/\mathcal{R}^u) \in \mathcal{R}^u$.
Writing $T_{ij} := t_{ij} -1$, there is a ring isomorphism
\[
\mathcal{R}^u \cong \W[[\left\lbrace T_{ij}\right\rbrace ]].
\]
We conclude with the following

\begin{prop} \label{liftmorph}
	Let $ f : A \rightarrow B$ be a morphism of ordinary abelian varieties over $k$, let $ f^t : B^t \rightarrow A^t$ be the dual morphism of $f$ and let $\mathcal{A}/R$ and $\mathcal{B}/R$ be deformations of $A/k$ and $B/k$ to $R$. Then $f$ lifts to a morphism $ F : \mathcal{A} \rightarrow \mathcal{B}$ of deformations if and only if 
\[
q\bigl(\mathcal{A}/R;P,f^t(Q^t)\bigr) = q\bigl(\mathcal{B}/R;f(P),Q^t\bigr)
\]
for every $P \in T_pA(k)$ and $Q^t \in T_pB^t(k)$. Furthermore, if a lifting exists, then it is unique.
\end{prop}

\begin{proof} This is \cite[Theorem 2.1, 4)]{Ka}. \end{proof}

\subsection{Serre--Tate coordinates for Shimura curves}

Take now an \emph{ordinary} QM abelian surface $A$  over $k$ with a $V_1(N^+)$-level structure. We want to deform our abelian surface not only as an abelian surface but also with its structures. Thus, we consider the subfunctor
$ \mathcal{M}$ of $\mathscr{F}= \Spf(\mathcal{R}^u)$ which sends an artinian local ring $R$ with residue field $k$ to the set of deformations of $A$ to $R$, where by deformation of $A$ to $R$ we mean a deformation $\mathcal{A}$ of $A$ to $R$ together with an embedding
$\mathcal{O}_B \hookrightarrow \End_{R}(\mathcal{A})$ deforming the given embedding $\mathcal{O}_B \hookrightarrow \End_k(A)$ and a $V_1(N^+)$-level structure on $\mathcal{A}$ deforming the given $V_1(N^+)$-level structure on $A$.
The $V_1(N^+)$-level structure automatically lifts uniquely, as $A[N^+]$ is \'etale over $R$, so we can ignore it in our discussion.

Consider the idempotent $e$ that acts as $\left( \begin{smallmatrix}
1 & 0 \\ 0 & 0
\end{smallmatrix} \right) \in M_2(\mathbb{Z}_p)$ on $T_pA$ ($i_K$ and $\Phi_p$ can be chosen to be compatible, by the choice of $\mathfrak{p}$ over $p=\mathfrak{p}\overline{\mathfrak{p}}$ split in $K$). We can find a
$\mathbb{Z}_p$-basis $\left\lbrace P_1,P_2\right\rbrace $ of $T_pA$ such that $eP_1 = P_1$ and $eP_2 = 0$, indeed $T_pA = eT_pA\oplus (1-e)T_pA$. Then $P_1^t \in (eT_pA)^t$. 

\begin{prop} \label{defthm}
The subfunctor $\mathcal{M}$ of $\mathscr{F}$ is pro-representable by a ring $\mathcal{R}^f$ that is a quotient of $\mathcal{R}^u$. In fact, $\mathcal{R}^f$ is the quotient of $\mathcal{R}^u$ by the closed ideal generated by the relations
\[
q\bigl(\widehat{\mathcal{A}}^u/\mathcal{R}^u;bP,Q^t\bigr) =q\bigl(\widehat{\mathcal{A}}^u/\mathcal{R}^u; P,b^{*}Q^t\bigr)
\]
for any $b\in B, P \in T_pA, Q^t\in T_pA^t$. Furthermore, there is an isomorphism
\[
\mathcal{R}^f \cong \W[[T_{11}]],
\]
where $T_{11}=t_{11} -1$ and $t_{11}$ corresponds to $q\bigl(\widehat{\mathcal{A}}^u/\mathcal{R}^u;P_1,P_1^t\bigr)$.
\end{prop}

\begin{proof}
This is a consequence of Proposition \ref{liftmorph}. For details, see \cite[Proposition 4.5]{Brooks} and \cite[Proposition 3.3]{Mo}.
\end{proof}

Thus, deformations of the QM abelian surface $A/k$ depend only on the $e$-component $eT_pA$ of $T_pA$.

Since the deformation functor $\mathcal{M}$ is the deformation functor associated with $Sh^{\ord}_{\mid_{\W}}$, i.e., the ordinary part of $Sh_{\mid_{\W}}$, and the point $x \in Sh^{\ord}(k)$ corresponding to the fixed ordinary QM abelian surface $A/k$ with $V_1(N^+)$-level structure, it follows that $\mathcal{M}$ is the formal completion $\widehat{Sh}^{\ord}_x$ of $Sh^{\ord}_{\mid_{\W}}$ at $x$ and so it is the formal spectrum $\Spf(\widehat{\mathcal{O}}_{Sh^{\ord},x})$, where $\mathcal{O}_{Sh^{\ord},x}$ is the local ring of $Sh^{\ord}$ at $x$.

\subsection{Deformations of QM abelian surfaces} \label{defQM}

In the case of QM abelian surfaces, the coordinate ring of the deformation functor has only one coordinate obtained by choosing a point $P \in T_pA$ such that $eP=P$, as we have seen in the previous section. Also in the case of elliptic curves there is only one coordinate obtained by choosing a point $P \in T_pE$. Actually, there is a strict link between deformations of QM abelian surfaces and deformations of elliptic curves.
 
Take an ordinary QM abelian surface $A / k$ where $k$ is again $\overline{\mathbb{F}}_p$. Then, as already pointed out, its deformation theory is equivalent to the deformation theory of the $p$-divisible group $A[p^{\infty}]$ (see the proof of \cite[Theorem 2.1]{Ka}).
The $p$-divisible group $A[p^{\infty}]$ attached to $A$ inherits an action of $\mathcal{O}_B$ and hence of $\mathcal{O}_B \otimes \mathbb{Z}_p$, which is identified with $M_2(\mathbb{Z}_p)$ via the fixed isomorphism $\Phi_p$.
If we set $e = \left( \begin{smallmatrix}
	1 & 0 \\ 0 & 0
\end{smallmatrix}
\right) \in M_2(\mathbb{Z}_p)$ (the idempotent $e$ acts as $\left( \begin{smallmatrix}
	1 & 0 \\ 0 & 0
\end{smallmatrix}
\right) \in M_2(\mathbb{Z}_p)$ on $A[p^{\infty}]$), then $A[p^{\infty}]$ splits as $eA[p^{\infty}] \oplus (1-e)A[p^{\infty}]$. Moreover, $eA[p^{\infty}]$ and $(1-e)A[p^{\infty}]$ are isomorphic via multiplication by $\left( \begin{smallmatrix}
0 & 1 \\ 1 & 0
\end{smallmatrix}
\right)$.
Since $A$ is ordinary, there is an isomorphism $A[p^{\infty}] \cong E[p^{\infty}]^2$ for $E/k$ an ordinary elliptic curve with $E[p^{\infty}] \cong eA[p^{\infty}]$ (see \cite[Corollary 4.6]{Buz}). Following \cite{Buz}, we want to recover the deformation theory of $A$ from the deformation theory of an elliptic curve.
Deforming $A[p^{\infty}]$ with its $\mathcal{O}_B$-action is the same as deforming $A[p^{\infty}]$ with its $M_2(\mathbb{Z}_p) \cong \mathcal{O}_B \otimes \mathbb{Z}_p$-action. According to Theorem \ref{defthm}, this is equivalent to deforming $eA[p^{\infty}]$, 
therefore the deformation theory of $A/k$ (or $A[p^{\infty}]$) is equivalent to the deformation theory of $E/k$ (or $E[p^{\infty}]$).

We want to relate the bilinear map $q_{\mathcal{A}}$, associated with a deformation $\mathcal{A}/R$ of a QM abelian surface $A/k$, to the map $q_\mathcal{E}$ associated with the deformation $\mathcal{E}/R$, corresponding to $\mathcal{A}/R$, of an elliptic curve $E/k$, when there is an isomorphism of $p$-divisible groups
\[
\alpha : eA[p^{\infty}] \stackrel{\cong}{\longrightarrow}E[p^{\infty}]
\]
over $k$. So we start by comparing the Weil pairings. Since the Weil pairing comes from Cartier duality for $p$-divisible groups, there is a commutative diagram for the Weil pairings
\[
\begin{tikzcd}
eA[p^n] \times (eA[p^n])^t \arrow[r, "e_{p^n,A}"] \arrow[d, "\alpha_n \times (\alpha_n^{t})^{-1}"]&  \bmu_{p^n} \arrow[d, "="] \\
E[p^n] \times E^t[p^n] \arrow[r, "e_{p^n,E}"] & \bmu_{p^n} \arrow[u]
\end{tikzcd}
\]
where $\alpha_n$ is the $n$-component of $\alpha$ and the first line in the diagram is the restriction of the Weil pairing associated with $A$ to $eA[p^n] \times (eA[p^n])^t$ (Cartier duality is compatible with duality of abelian schemes, so $(eA[p^n])^t \hookrightarrow (A[p^n])^t \cong A^t[p^n]$ and $(eA[p^n])^t \cong (E[p^n])^t \cong E^t[p^n]$). This means that for each $P\in eA[p^n](k)$ and $Q^t \in E^t[p^n](k)$, we have
\[
e_{p^n,A}\bigl(P,\alpha_n^t(Q^t)\bigr) = e_{p^n,E}\bigl(\alpha_n(P),Q^t\bigr).
\]
The same is true when we take inverse limits.

Considering the completions at the origin and restricting the pairings, we obtain
\[
\begin{tikzcd}
e\widehat{A}[p^n] \times (eA[p^n])^t \arrow[d, "\cong"] \arrow[r, "e_{p^n,A}"] &  \bmu_{p^n} \arrow[dd, "="] \\
\widehat{eA}[p^n] \times (eA[p^n])^t \arrow[d, "\alpha_n \times (\alpha_n^{t})^{-1}"]& &\\
\widehat{E}[p^n] \times E^t[p^n] \arrow[r, "e_{p^n,E}"] & \bmu_{p^n} \arrow[uu]
\end{tikzcd}
\]
because the functor $G \mapsto \widehat{G}$ sending a $p$-divisible group to its completion at the origin is exact and the connected-\'etale sequence is functorial. Then passing to the limits yields pairings between Tate modules and the commutative diagram
\[
\begin{tikzcd}
e\widehat{A}(k) \times (eT_pA)^t \arrow[d, "\alpha \times (\alpha^t)^{-1}"] \arrow[r, "E_{A}"] &  \widehat{\mathbb{G}}_m \arrow[d, "="] \\
\widehat{E}(k) \times T_pE^t \arrow[r, "E_{E}"] & \widehat{\mathbb{G}}_m. \arrow[u]
\end{tikzcd}
\]
When we extend these pairings to $e\widehat{\mathcal{A}}$ and $\widehat{\mathcal{E}}$, everything works well because of the functoriality of the structure of extensions of $p$-divisible groups and the fact that we are deforming also the action of $\mathcal{O}_B$ and so the action of $e$, so that $e\mathcal{A}[p^{\infty}] \cong \mathcal{E}[p^{\infty}]$.

Observe that everything works fine for the $`` p^n"$ maps as well, and there is a commutative diagram
\[
\begin{tikzcd}
T_pE \arrow[r, two heads] \arrow[d, "\cong"] &E[p^n] \arrow[d, "\cong"] \arrow[r, "{`` p^n"}"] &  \widehat{\mathcal{E}}(R) \arrow[d, "\cong"] \\
eT_pA \arrow[r, two heads]  &eA[p^n] \arrow[r, "{`` p^n"}"] & e\widehat{\mathcal{A}}(R),
\end{tikzcd}
\]
again because we are deforming also the action of $\mathcal{O}_B$ and so the action of $e$.

In conclusion, computing the bilinear map on $eT_pA \times (eT_pA)^t$ is the same as computing it on $T_pE \times T_pE^t$, that is
\[
q(\mathcal{A};P,Q^t) = q\bigl(E,\alpha(P),(\alpha^t)^{-1}(Q^t)\bigr),
\]
for all $P \in eT_pA$ and $Q^t \in (eT_pA)^t \subseteq T_pA^t$.

\subsection{Deformation at points of $I$}

If $A/k$ is an ordinary QM abelian surface with $A[p] \cong E[p]^2$ as $\mathcal{O}_B$-group schemes, where here $\mathcal{O}_B$ acts via the natural action of $\mathcal{O}_B \otimes \mathbb{F}_p \cong M_2(\mathbb{F}_p)$ on $E[p]^2$, then there is an induced isomorphism between the set of $V_1(p)$- level structures on $A$ and the set of $\Gamma_1^{\text{arith}}(p)$-level structures
on $E$. 
In the same way, when $A[p^{\infty}] \cong E[p^{\infty}]^2$ there is a bijection between the set of $V_1(p^{\infty})$- level structures on $A$ and the set of $\Gamma_1^{\text{arith}}(p^{\infty})$-level structures
on $E$. Thus, the deformation theory of a $k$-point in $I_n(k)$, or in the Igusa tower, is equivalent to the deformation theory of the associated elliptic curve viewed as a $k$-point of the scheme parameterizing elliptic curves with $\Gamma_1^{\text{arith}}(N^+p^n)$- or $\Gamma_1^{\text{arith}}(N^+p^{\infty})$-level structures. 
In light of what we have seen in the previous section, we can use this equivalence to compute Serre--Tate coordinates.

\subsection{$t$-expansions for modular forms}

Let us start from an $\overline{\mathbb{F}}_p$-point $ x $ in the Igusa tower, i.e., the isomorphism class of a quadruple $(A/{\overline{\mathbb{F}}_p}, \iota, \nu_{N^+},\nu_{p^{\infty}})$.
Then the $V_1(p^{\infty})$-level structure on $A_{| \overline{\mathbb{F}}_p}$ determines a point $P^{t} \in (eT_pA)^t$ (cf. \cite[\S 3.1]{CH}). Take $P \in eT_pA$ corresponding to $P^t$ via the principal polarization. We fix the Serre--Tate coordinate $t_x$ around $x$ to be
\[
t_x := q(-;P,P^t).
\]
Denote by $\bigl({\bm{\mathcal{A}}}/\W[[T]], \bm{\iota}, \bm{\nu}_{N^+},\bm{\nu}_{p^{\infty}}\bigr)$ the universal deformation of $x$ and note that we can evaluate every $p$-adic modular form $f\in V_p(N^+,\W)$ at $\bigl({\bm{\mathcal{A}}}/\W[[T]], \bm{\iota}, \bm{\nu}_{N^+},\bm{\nu}_{p^{\infty}}\bigr)$. We call
\[
f(t_x) := f\bigl({\bm{\mathcal{A}}}/\W[[T]], \bm{\iota}, \bm{\nu}_{N^+},\bm{\nu}_{p^{\infty}}\bigr) \in \W[[T]],
\]
where $T:=t_x-1$, the \textbf{$t$-expansion of $f$ at $x$}.

\subsection{On Serre--Tate coordinates at CM points}

In this section we want to obtain a result analogous to \cite[Lemma 3.2]{CH} in our setting.

If $\mathfrak{a}$ is a prime to $cpN$ fractional ideal of $\mathcal{O}_c$ with $p \nmid c$, then $A_{\mathfrak{a}}$ has a model defined over $\mathcal{V} := \W \cap K^{\ab}$.
Here
$x(\mathfrak{a}) =(A_{\mathfrak{a}},\iota_{\mathfrak{a}},\nu_{\mathfrak{a},N^+}, \nu_{\mathfrak{a},p^{\infty}}) \in I(\mathcal{V})$ is as defined in \S \ref{CMpoints}.
Denote by $t$ the Serre--Tate coordinate around $\overline{x}(\mathfrak{a}) := x(\mathfrak{a}) \otimes_{\mathcal{V}} \overline{\mathbb{F}}_p$.
For $u= 1, \dots, p^n-1$ with $(u,p)=1$, set
\[
x(\mathfrak{a}) \star \alpha(u/p^n) := x(cp^n)^{\text{rec}_K(a^{-1}u_{\mathfrak{p}}p_{\mathfrak{p}}^{-n})} \cdot u \in I(\mathcal{V}),
\]
where $\text{rec}_K : K^{\times} \backslash \hat{K}^{\times} \rightarrow \Gal(K^{\ab}/K)$ is the geometrically normalized reciprocity law map, $a \in \hat{K}^{(cp)\times}$ is such that $\mathfrak{a} = a \hat{\mathcal{O}}_c \cap K$ and the subscript $b_{\mathfrak{p}}$ for $b \in \mathbb{Z}_p^{\times}$ denotes its image in $\hat{K}^{\times}$ under the inclusions $\mathbb{Z}_p^{\times} \subseteq K_{\mathfrak{p}}^{\times} \subseteq \hat{K}^{\times}$.

\begin{lem} \label{3.2}
With notation as above, one has that $\big(x(\mathfrak{a}) \star \alpha(u/p^n) \big)\otimes_{\mathcal{V}} \overline{\mathbb{F}}_p= \overline{x}(\mathfrak{a})$ and $t\big(x(\mathfrak{a}) \star \alpha(u/p^n)\big) = \zeta_{p^n}^{-uN(\mathfrak{a})^{-1}\sqrt{-D_K}^{-1}}$.
\end{lem}

\begin{proof}
The $p$-divisible module $eA_x[p^{\infty}]$, with $A_x$ the QM abelian surface corresponding to $x = x(\mathfrak{a}) \star \alpha(u/p^n)$, is exactly the $p$-divisible module associated with the point $x_{\mathfrak{a}} \star \mathbf{n}(up^{-n})$ considered in \cite[Lemma 3.2]{CH} (cf. \cite[\S 4.5]{CH}). Hence, $x$ in the deformation space of $\overline{x}(\mathfrak{a})$ corresponds to $x_{\mathfrak{a}} \star \mathbf{n}(up^{-n})$ in the deformation space of $x_{\mathfrak{a}}\otimes_{\mathcal{V}} \overline{\mathbb{F}}_p$, where $x_{\mathfrak{a}}$ is the CM point defined in \cite[\S 2.4]{CH}.

Set $\overline{A}_{\mathfrak{a}}:= A_{\mathfrak{a}} \otimes_{\mathcal{V}} \overline{\mathbb{F}}_p$ and $\overline{E}_{\mathfrak{a}}:= E_{\mathfrak{a}} \otimes_{\mathcal{V}} \overline{\mathbb{F}}_p$, with $E_{\mathfrak{a}}$ the elliptic curve corresponding to the CM point $x_{\mathfrak{a}}$. Since the point $P^t \in (eT_p\overline{A}_{\mathfrak{a}})^t$ that is determined by the $V_1(p^{\infty})$-level structure is the same as the point that is determined by the $\Gamma_1^{\text{arith}}(p^{\infty})$-level structure on $T_p\overline{E}_{\mathfrak{a}}^t$, the claim follows from the computations of \S 2.4 and \cite[Lemma 3.2]{CH}.
\end{proof}

\section{Anticyclotomic $p$-adic $L$-functions}

In this section we will define our $p$-adic $L$-function as a measure on $\Gal(K_{p^{\infty}}/K)$ with values in $\W$, which is again the ring of Witt vectors $W(\overline{\mathbb{F}}_p)$, i.e., the ring of integer of the completion of the maximal unramified extension $\mathbb{Q}_p^{\text{ur}}$ of $\mathbb{Q}_p$.

\subsection{Measures on $\mathbb{Z}_p$} \label{meas}

Recall that a \emph{$p$-adic measure} on $\mathbb{Z}_p$ with values in $\W$ is a $\W$-linear function $\mu : \mathcal{C}(\mathbb{Z}_p, \W) \rightarrow \W$ such that there exists a constant $B \geq 0$ with 
    $
	\bigl|\mu(\varphi)\bigr|_p \leq B |\varphi|_p
	$
	for each $\varphi \in \mathcal{C}(\mathbb{Z}_p, \W)$, where $|\varphi|_p := \sup_{x\in\mathbb{Z}_p}\bigl|\varphi(x)\bigr|_p$.
Here with $\mathcal{C}(\mathbb{Z}_p, \W)$ we denote the space of continuous functions from $\mathbb{Z}_p$ to $\W$.
We will write $\int_{\mathbb{Z}_p} \varphi d\mu := \mu(\varphi)$ for the value of a measure $\mu$ on a continuous function $\varphi$. For more details, the reader is referred to \cite[Chapter 3]{Hi93}.

We denote by $M(\mathbb{Z}_p,\W)$ the space of $p$-adic measures on $\mathbb{Z}_p$ with values in $\W$. When equipped with the norm
$
|\mu|_p:=\sup_{\mid\varphi\mid_p =1}\bigl|\mu(\varphi)\bigr|_p,
$
the space $M(\mathbb{Z}_p,\W)$ is a $p$-adic Banach $\W$-module. 
Recall that there is an isomorphism
\[
	\begin{split}
	M(\mathbb{Z}_p,\W) &\stackrel{\cong}{\longrightarrow} \W[[T]]\\
	\mu &\longmapsto \Phi_{\mu}
	\end{split}
\]
where
\[
\Phi_{\mu}(t) := \sum_{n=0}^{\infty} \left(\int_{\mathbb{Z}_p}\binom{x}{n} d \mu \right) T^n \in \W[[T]],\quad\text{with}\ T:=t-1,
\]
and that
\begin{displaymath}
\int_{\mathbb{Z}_p}z^x d \mu = \sum_{n=0}^{\infty} \int_{\mathbb{Z}_p} \binom{x}{n} (z-1)^n d\mu= \Phi_{\mu}(z)\ \text{for}\ z \in \W\ \text{with}\ \mid z-1\mid_p <1.
\end{displaymath}
The space of $p$-adic measures $M(\mathbb{Z}_p,\W)$ is naturally a $\mathcal{C}(\mathbb{Z}_p,\W)$-module in the following way: for $\phi \in \mathcal{C}(\mathbb{Z}_p,\W)$ and $\mu \in M(\mathbb{Z}_p,\W)$, we set
$
\int_{\mathbb{Z}_p}\varphi d \phi\cdot\mu := \int_{\mathbb{Z}_p}\varphi\phi d\mu
$
for any $\varphi \in \mathcal{C}(\mathbb{Z}_p,\W)$. Furthermore, for $\phi \in \mathcal{C}(\mathbb{Z}_p,\W)$ and $\mu \in M(\mathbb{Z}_p,\W)$, we write 
\[
[\phi]\Phi_{\mu}(t) := \Phi_{\phi\mu}(t) = \int_{\mathbb{Z}_p}\phi(x)t^x d \mu \in \W[[t-1]].
\]
Note that for $m \geq 0$
\begin{equation}\label{x^m}
    [x^m]\Phi_{\mu}(t) = \Phi_{x^m\mu} = \left( t\frac{d}{dt} \right) ^m \Phi_{\mu}(t).
\end{equation}
If we consider a locally constant function $\phi \in \mathcal{C}(\mathbb{Z}_p,\W)$ that factors through $\mathbb{Z}_p/p^n\mathbb{Z}_p$, then
\begin{equation}\label{twistphi}
[\phi]\Phi_{\mu} (t) = \Phi_{\phi\mu} (t) = p^{-n} \sum_{b \in \mathbb{Z}/p^n\mathbb{Z}} \phi(b) \sum_{\zeta \in \boldsymbol{\mu}_{p^n}} \zeta^{-b} \Phi_{\mu}(\zeta t) \in \W[[t-1]]
\end{equation}
for $\mu \in M(\mathbb{Z}_p,\W)$ (see \cite[\S3.5]{Hi93}). Observe that the notation $[\phi]\Phi_{\mu}$ coincides with that in \cite[(8.1)]{Bra} and \cite[\S 3.5]{Hi93}, and corresponds to $ \Phi_{\mu} \otimes \phi$ in \cite[\S 3.1]{CH}.
Furthermore, there is an equality
\begin{equation} \label{x^mtheta}
\int_{\mathbb{Z}_p} \phi(x)x^m d\mu = \left( t\frac{d}{dt} \right) ^m {([\phi]\Phi_{\mu})|}_{t=1}.
\end{equation}

\subsection{Measures on $\Gal(K_{c_0p^{\infty}}/K)$}

Let ${\mathfrak{a}_1,\dots,\mathfrak{a}_H}$ be a complete set of representatives for $\Pic(\mathcal O_{c_0})$.
As in \cite[\S 8.2]{Bra}, there is an explicit coset decomposition
$
\Gal(K_{c_0p^{\infty}}/K) = \Pic\mathcal O_{c_0p^\infty} =  \bigsqcup_{j=1}^H \mathfrak{a}_j^{-1} \mathbb{Z}_p^{\times},
$
that allows us to construct a $\W$-valued measure $\mu$ on $\Pic\mathcal O_{c_0p^\infty}$ by constructing $H$ distinct $\W$-valued measures $\mu_{\mathfrak{a}_j}$ on $\mathbb{Z}_p^{\times}$, so that for every continuous function $\varphi : \Pic\mathcal O_{c_0p^\infty} \rightarrow \W$
we have
\[
\int_{\Pic\mathcal O_{c_0p^\infty}} \varphi d\mu_f = \sum_{\mathfrak{a} \in \Pic\mathcal O_{c_0}} \int_{\mathbb{Z}_p^{\times}} \varphi \mid [\mathfrak{a}] d\mu_{f,\mathfrak{a}},
\]
where $\varphi \mid [\mathfrak{a}]$ is $\varphi$ restricted to $\mathfrak{a}^{-1}\mathbb{Z}_p^{\times}$.
Therefore, a measure $\mu$ on $\Gal(K_{c_0p^{\infty}}/K)$ is equivalent to a collection $\left\lbrace \mu_{\mathfrak{a}} \right\rbrace _{\mathfrak{a} \in \Pic(c_0)}$ of $H$ measures on $\mathbb{Z}_p^{\times}$.

\subsection{Measure associated with a modular form}

Let $g$ be a $p$-adic modular form on $Sh$ over $\W$ and let $\mathfrak{a} \in \Pic(\mathcal{O}_{c_0})$.
Define a $\W$-valued measure $\mu_{g,\mathfrak{a}}$ on $\mathbb{Z}_p$ by
\[
\int_{\mathbb{Z}_p} t^x d\mu_{g,\mathfrak{a}} = g(t_{\mathfrak{a}}) \in \W[[t_{\mathfrak{a}}]],
\]
where $t_{\mathfrak{a}}$ is the Serre--Tate coordinate around $x(\mathfrak{a}) \otimes_{\W} \overline{\mathbb{F}}_p$ and $x(\mathfrak{a})$ is as defined in Section \ref{CMpoints}. Indeed, if $\mathfrak{a}$ is a prime-to-$pN$ fractional ideal of $\mathcal{O}_c$ and $p \nmid c$, then $x(\mathfrak{a})$ has a model defined over $\V := \W \cap K^{\ab}$. If the measures $\mu_{g,\mathfrak{a}}$, for $\mathfrak{a} \in \Pic(\mathcal{O}_{c_0})$, are supported on $\mathbb{Z}_p^{\times}$, then we can put them together to obtain a measure $\mu_g$ on $\Gal(K_{c_0p^{\infty}}/K)$.

\subsection{$p$-depletion of a modular form}

In order to obtain measures supported on $\mathbb{Z}_p^{\times}$, now we introduce the $p$-depletion of a modular form. We follow \cite[\S 3.6]{Brooks}.

Recall the operators $U$ and $V$. Take a QM abelian surface $A$ with ordinary reduction over a $p$-adic field $L$. Then there is a unique $p$-torsion cyclic $\mathcal{O}_B$-submodule $C$ of $A$ which reduces
mod $p$ to the kernel of the Frobenius morphism, that is the canonical subgroup (cf. \cite[Theorem 1.1]{Kas}).
Denote by $\phi_i: A \rightarrow A/C_i$, for $i = 0, \dots, p$, the distinct $p$-isogenies of QM abelian surfaces on $A$ ordered in such a way that $C_0$ is the canonical subgroup of $A$. 
If $t : \bmu_{N^+} \times \bmu_{N^+} \hookrightarrow A[N^+]$ is a $V_1(N^+)$-level structure on $A$, then, since $p \nmid N^+$, $t_i = \phi_i \circ t$ is a $V_1(N^+)$-level
structure on $A/C_i$. Also if $\omega$ is a one-form on $A$, then there is a unique one-form $\omega_i$ on
$A/C_i$ such that $\phi_i ^* \omega_i = \omega$.
If $g$ is a modular form, we can define another modular form $g\mid V$ by
\[
g\mid V (A,t, \omega) := g(A / C_0, 1/p t_0, p \omega_0)
\]
and also a modular form $g \mid U$ by
\[
g\mid U (A,t, \omega) := \sum_{i=1}^{p} g(A / C_i, t_i, \omega_i).
\]
If $[p]$ is the operator on modular forms that is given by
$g \mid [p](A, t, \omega) = g(A, pt, 1/p\omega)$, then $U$ and $V$ are related to the usual $T_p$ operator by
\[
T_p = U + 1/p [p]V.
\]
Furthermore, one has $VU = \id$ and the operators $UV$ and $V U - UV$ are idempotent. The {\bf $p$-depletion} of a modular form $g$ is defined to be
 \[
 g^{(p)} := g \mid (\id - UV) = g \mid (\id -  T_p V +  1/p [p]V^2).
 \]
 
\subsection{Hecke characters and $p$-adic Galois characters} \label{char}
 
Let $K$ be our imaginary quadratic field
and let $m$, $n$ be integers. 

\begin{defi}
A \textbf{Hecke character of $K$ of infinity type $(m,n)$} is a continuous homomorphism
\[
\chi : K^{\times} \backslash \mathbb{A}_K^{\times} \longrightarrow \mathbb{C}^{\times}
\]
satisfying
$
\chi(x \cdot z_{\infty}) = \chi(x) z_{\infty}^m\overline{z}_{\infty}^n,
$
for every $z_{\infty} \in K_{\infty}^{\times}$ and $x \in \hat{K}^{\times}$.
\end{defi}
In particular, the infinite component of $\chi$ is given by $\chi_{\infty}(z) = z^m \overline{z}^n$.
For each prime $\mathfrak{q}$ of $K$, denote by $\chi_{\mathfrak{q}} : K_{\mathfrak{q}}^{\times} \rightarrow \mathbb{C}^{\times}$ the $\mathfrak{q}$-component of $\chi$. The \emph{conductor} of $\chi$ is the largest integral ideal $\mathfrak{c}_f$ of $K$ such that $\chi_{\mathfrak{q}}(u) = 1 $ for each element $u \in 1 + \mathfrak{c}_f \mathcal{O}_{K,\mathfrak{q}}$. As it is known, one can identify a Hecke character $\chi$ with a character on fractional ideals of $\mathcal{O}_K$ prime to $\mathfrak{c}_f$ via the formula
$
\chi(\mathfrak{a}) = \prod_{\mathfrak{q} \mid \mathfrak{a}\ \text{prime}} \chi_{\mathfrak{q}} (\pi_{\mathfrak{q}})^{v_{\mathfrak{q}}(\mathfrak{a})},
$
with $\pi_{\mathfrak{q}}$ a uniformizer at $\mathfrak{q}$; the formula is independent of the choice of
the uniformizer.
So, if $\chi$ has conductor $\mathfrak{c}$ and $\mathfrak{a}$ is any fractional ideal prime to $\mathfrak{c}$, we write $\chi(\mathfrak{a})$ for $\chi(a)$, where $a\in \mathbb{A}_K^{\times}$ is an idele with $a\hat{\mathcal{O}}_K \cap K = \mathfrak{a}$ and $a_{\mathfrak{q}} = 1$ for all $\mathfrak{q} \mid \mathfrak{c}$. 

A Hecke character $\chi$ is called \textbf{anticyclotomic} if $\chi$ is trivial on $\mathbb{A}_{\mathbb{Q}}^{\times}$.
The \textbf{$p$-adic avatar} $\hat{\chi} : K^{\times} \backslash 
\hat{K}^{\times} \rightarrow \mathbb{C}_p^{\times}$ of a Hecke character $\chi$ of infinity type $(m, n)$ is defined by
\[
\hat{\chi} (x) = \chi(x) x_{\mathfrak{p}}^m x_{\overline{\mathfrak{p}}}^n
\]
with $x \in \hat{K}^{\times}$ and $\mathfrak{p}$ the chosen prime above $p$ which splits in $K$. Every $p$-adic Galois character $\rho: G_K \rightarrow \mathbb{C}_p^{\times}$ can be seen as a $p$-adic character $ K^{\times} \backslash \hat{K}^{\times} \rightarrow \mathbb{C}_p^{\times}$ via the geometrically normalized reciprocity law map $\text{rec}_K : K^{\times} \backslash \hat{K}^{\times} \rightarrow \Gal(K^{\ab}/K)$.
 
 A $p$-adic Galois character
 is said to be \textbf{locally algebraic} if it is the $p$-adic avatar of some Hecke character. A locally algebraic
 character is called of infinity type $(m, n)$ if the associated Hecke character is of infinity type
 $(m, n)$, and its conductor is the conductor of the associated Hecke character. 

\subsection{Construction of a measure}

Consider now our modular form $f \in S_k^{\text{new}}(\Gamma_0(N))$, with $k\geq 4$, and let
$F/\mathbb{Q}_p$ be a finite extension containing the image of the Fourier coefficients of $f$. Via the Jacquet--Langlands correspondence we can see $f$ as a modular form in $M_k(Sh, \mathcal{O}_F)$. Take the $p$-depletion $f^{(p)}$ of $f$ and then consider it as a $p$-adic modular form $\hat{f}^{(p)}$ in $V_p(N^+,\mathcal{O}_F)$ of weight $k$.

Fix an anticyclotomic Hecke character $\psi$ of infinity type $(k/2, -k/2)$, and let $c_0\mathcal{O}_K$ be the prime to $p$ part of the conductor of $\psi$.
Take a finite extension of $\W$ obtained by adjoining the values of the Hecke character $\psi$ and, with an abuse of notation, still denote it by $\W$. Let $\hat{\psi}$ be the $p$-adic avatar of $\psi$. The $\W$-valued measures $\mu_{\hat{f}^{(p)},\mathfrak{a}}$ are given by
\[
\psi(\mathfrak{a})N(\mathfrak{a})^{-k/2} \psi_{\mathfrak{p}} \mu_{\hat{f}^{(p)}_{\mathfrak{a}}},\]	
where $\mu_{\hat{f}^{(p)}_{\mathfrak{a}}}$ is defined by
\[
\int_{\mathbb{Z}_p} t^x d\mu_{\hat{f}^{(p)}_{\mathfrak{a}}} = \hat{f}^{(p)}_{\mathfrak{a}}(t_{\mathfrak{a}}) := \hat{f}^{(p)} (t_{\mathfrak{a}}^{N(\mathfrak{a}^{-1})\sqrt{(-D_K)}^{-1}}) \in \W[[t_{\mathfrak{a}}-1]],
\]
and $t_{\mathfrak{a}}$ is the Serre--Tate coordinate around $ x(\mathfrak{a}) \otimes_{\W} \overline{\mathbb{F}}_p$.
\begin{rmk}
	The measure $\mu_{\hat{f}^{(p)}_{\mathfrak{a}}}$ associated with $\hat{f}^{(p)}_{\mathfrak{a}}$ is supported on $\mathbb{Z}_p^{\times}$.
	Indeed, $UV \hat{f}^{(p)} =0$ because $VU=\id$.
	Since $UV$ acts on the expansion in Serre--Tate coordinates as
	$UV g(t) = 1/p\sum_{\zeta \in \bmu_p} g(\zeta t)$ (see \cite[Proposition 4.17]{Brooks}), taking $\phi = \boldsymbol{1}_{\mathbb{Z}_p^{\times}}$ to be the characteristic function of $\mathbb{Z}_p^{\times}$ and using \eqref{twistphi} yields
	$
[\phi]g(t) 
= p^{-1} \sum_{b \in \mathbb{Z}/p\mathbb{Z}} \phi(b) \sum_{\zeta \in \boldsymbol{\mu}_{p}} \zeta^{-b} g(\zeta t) 
= [\boldsymbol{1}_{\mathbb{Z}_p}]g(t) - p^{-1} \sum_{\zeta \in \boldsymbol{\mu}_{p}} g(\zeta t).
	$
Hence, $\mu_{\hat{f}^{(p)},\mathfrak{a}}$ is supported on $\mathbb{Z}_p^{\times}$ as well.	
\end{rmk}

\begin{defi} \label{Lfct}
Let $\psi$ denote an anticyclotomic Hecke character of infinity type $(k/2, -k/2)$ and conductor $c_0\mathcal{O}_K$ with $(c_0,pN^+)=1$. The {\bf measure $\LL_{f,\psi}$ associated with $f$ and $\psi$} is the $\W$-valued measure given by 
	\[
	\LL_{f,\psi} (\varphi) = \sum_{\mathfrak{a} \in \Pic(\mathcal O_{c_0})} \int_{\mathbb{Z}_p^{\times}} \varphi \big|[\mathfrak{a}]d\mu_{\hat{f}^{(p)},\mathfrak{a}}
	\]
	for any continuous function $\varphi:\Gal(K_{c_0p^{\infty}}/K)\rightarrow\W$.
\end{defi}

	Therefore
	\[
	\begin{split}
	\LL_{f,\psi}(\varphi) &= \sum_{\mathfrak{a} \in \Pic(\mathcal O_{c_0})} \psi(\mathfrak{a})N(\mathfrak{a})^{-k/2} \int_{\mathbb{Z}_p^{\times}} \psi_{\mathfrak{p}}\varphi \big|[\mathfrak{a}]d\mu_{\hat{f}^{(p)}_{\mathfrak{a}}}
	= \sum_{\mathfrak{a} \in \Pic(\mathcal O_{c_0})} \psi(\mathfrak{a})N(\mathfrak{a})^{-k/2} \Phi_{\psi_{\mathfrak{p}}\varphi \mid[\mathfrak{a}]d\mu_{\hat{f}^{(p)}_{\mathfrak{a}}}} \mid_{t=1} \\
	&= \sum_{\mathfrak{a} \in \Pic(\mathcal O_{c_0})} \psi(\mathfrak{a})N(\mathfrak{a})^{-k/2} [\psi_{\mathfrak{p}}\varphi \big|[\mathfrak{a}]]\hat{f}^{(p)}_{\mathfrak{a}}(t_{\mathfrak{a}}) \mid_{t=1}
	= \sum_{\mathfrak{a} \in \Pic(\mathcal O_{c_0})} \psi(\mathfrak{a})N(\mathfrak{a})^{-k/2} [\psi_{\mathfrak{p}}\varphi \big|[\mathfrak{a}]]\hat{f}^{(p)}_{\mathfrak{a}}\bigl(x(\mathfrak{a})\bigr).
	\end{split}
	\]
Notice that the second equality holds because $\mu_{\hat{f}^{(p)}_{\mathfrak{a}}}$ is supported on $\mathbb{Z}_p^{\times}$. Indeed, for a measure $\mu$ supported on $\mathbb{Z}_p^{\times}$ one has 
$
\int_{\mathbb{Z}_p^{\times}} 1 d \mu =
\int_{\mathbb{Z}_p} 1 d \mu =\Phi_{\mu} \mid_{t=1}.
$
We are also using the fact that $x(\mathfrak{a})$ is the canonical lifting of $x(\mathfrak{a}) \otimes_{\W} \overline{\mathbb{F}}_p$, that is $t_{\mathfrak{a}}\bigl(x(\mathfrak{a})\bigr)=1$, where $t_{\mathfrak{a}}$ is again the Serre-Tate coordinate around $x(\mathfrak{a}) \otimes_{\W} \overline{\mathbb{F}}_p$.
\begin{rmk}
We are interested in evaluating $\LL_{f,\psi}$ at continuous functions that factor through $\Gal(K_{p^{\infty}}/K)$. In other words, we will view $\LL_{f,\psi}$ as a measure on $\Gal(K_{p^{\infty}}/K)$.
\end{rmk}

Now we state a result that we will use later.
Let $g\in V_p(N^+,\W)$ be a $p$-adic modular form and let $g_{\mathfrak{a}}$ be defined as above, so that $g_{\mathfrak{a}}(t) = g\bigl(t^{N(\mathfrak{a})^{-1}\sqrt{-D_K}^{-1}}\bigr)$ with $t$ the Serre--Tate coordinate around $x(\mathfrak{a}) \otimes_{\W} \overline{\mathbb{F}}_p$.

\begin{prop} \label{3.3}
 If $g\in V_p(N^+,\W)$ and $\phi : (\mathbb{Z}/p^n\mathbb{Z})^{\times} \rightarrow \mathbb{C}^{\times}$ is a primitive Dirichlet character, then
\[
[\phi]g_{\mathfrak{a}}(x(\mathfrak{a})) = p^{-n} G(\phi) \sum_{u \in (\mathbb{Z}/p^n\mathbb{Z})^{\times}} \phi^{-1}(u) g(x(\mathfrak{a}) \star \alpha(u/p^n)),
\]
where $G(\phi) = \sum_{v \in (\mathbb{Z}/p^n\mathbb{Z})^{\times}} \phi(v) \zeta_{p^n}^v$ is the Gauss sum of $\phi$.
\end{prop}

\begin{proof}
The statement follows by applying (\ref{twistphi}) and Lemma \ref{3.2}.
\end{proof}

\subsection{Interpolation properties} \label{interp}

Working on $\hat{B}$ instead of $\GL_2(\hat{\mathbb{Q}})$ and adapting the computations from \cite{Hs}, one can obtain an analogue of \cite[Theorem 3.4]{Hs} in our quaternionic setting and use it, as in \cite{CH}, to get an interpolation formula for our $p$-adic $L$-function evaluated at Galois characters that are $p$-adic avatars of anticyclotomic Hecke characters of infinity type $(n,-n)$ with $n \geq 0$. In particular, one can relate our $p$-adic $L$-function to the Rankin--Selberg $L$-function associated with $f$ and some anticyclotomic Hecke character $\chi$ of infinity type $(k/2+n,k/2-n)$ with $n \geq 0$, i.e., the $L$-function associated with the $G_K$-representation $V_{f,\chi}= V_f(k/2) \otimes \chi$.

In the statement of the following theorem, $\psi$ is, as usual, an anticyclotomic Hecke character of infinity type $(k/2, -k/2)$ and conductor $c_0\mathcal{O}_K$ with $(c_0,pN^+)=1$.

\begin{teor} \label{interpolation-thm}
Let $\hat{\phi}$ be the $p$-adic avatar of an anticyclotomic Hecke character $\phi$ of infinity type $(n,-n)$ with $n \geq 0$ and $p$-power conductor. Then there exists a non-zero constant $C(f, \psi, \phi, K)$ depending on $f$, $\psi$, $\phi$, $K$ such that $C(f, \psi,\phi, K) \cdot L(f, \psi\phi, k/2)$ is an algebraic number and 
\[
\left(\LL_{f,\psi}(\hat{\phi})\right)^2 = C(f, \psi,\phi, K) \cdot L(f, \psi\phi, k/2),
\]
where the equality holds via the fixed embedding $i_p : \overline{\mathbb{Q}} \hookrightarrow \mathbb{C}_p$.
\end{teor}
A proof of this result will appear in a future project.

\section{Generalized Heegner cycles}

Recall that $K$ is an imaginary quadratic field, $N$ is a positive integer with a factorization $N = N^+ N^-$
where $N^+$ is a product of primes that are split in $K$ and $N^-$ is a square-free
product of an even number of primes that are inert in $K$, $B$ is again our indefinite rational quaternion algebra of discriminant $D=N^-$ and $p$ is an odd prime that splits in $K$ and $B$ and such that $(p,N)=1$.

Consider then our Shimura curve $Sh$, our fixed modular form $f$ in $S_k^{\text{new}}(\Gamma_0(N))$ of weight $k = 2r+2 \geq 4$, which can be seen by Jacquet--Langlands as a modular form on $Sh$, and the $r$-fold fiber product $\mathcal{A}^r$ of the universal QM abelian surface over $Sh$ with itself.

Following the work of Brooks, \cite{Brooks}, we want to define generalized Heegner cycles associated with $f$ lying over a Kuga--Sato variety over $Sh$.
Indeed, these cycles will live in the Chow groups of the generalized Kuga--Sato variety $\mathcal{X}_{r} = \mathcal{A}^{r} \times A^r$, where $A$ will be a fixed QM abelian surface. 
Then, to obtain cohomology classes from the generalized Heegner cycles, we will apply the $p$-adic Abel--Jacobi map.
We will construct in this way a system of generalized Heegner classes indexed by fractional ideals of $K$.

\subsection{Kuga--Sato varieties} \label{K-S}

Consider the $r$-fold fiber product $\mathcal{A}^r$ of the universal QM abelian surface $\mathcal{A}$ with itself over $Sh$, which is called the \textbf{$r^{\text{th}}$-Kuga--Sato variety} over $Sh$.

We define the action of the Hecke operator $T_{\ell}$, for $\ell \nmid N^+D$ on the
Kuga--Sato variety $\mathcal{A}^r$
as follows. 
Recall the interpretation of the Hecke operator $T_{\ell}$ on $Sh$ as correspondence.
Let $Sh(\ell)$ be the Shimura curve classifying quadruples $(A,\iota,\nu_{N^+},C)$, where
$(A,\iota,\nu_{N^+})$ is a QM abelian surface with $V_1(N^+)$-level structure endowed with a subgroup $C$ of $A[\ell]$ stable under the action of $\mathcal{O}_B$ and cyclic as $\mathcal{O}_B$-module. $A[\ell]$ has $\ell + 1$ such $\mathcal{O}_B$-submodules, all of them of order $\ell^2$.
Consider the natural forgetful morphism of Shimura curves $\alpha: Sh(\ell) \rightarrow Sh$ and the morphism $\beta: Sh(\ell) \rightarrow Sh$ given by $(A,\iota,\nu_{N^+},C) \mapsto (A/C,\iota_{\psi_C},\nu_{N^+,\psi_C})$, where $\iota_{\psi_C}$, $\nu_{N^+,\psi_C}$ are induced by the isogeny $\psi_C : A \twoheadrightarrow A/C$. Then $T_{\ell}$ is defined by the correspondence
\[
\begin{tikzcd}
 & Sh(\ell) \arrow[dr, "\beta"] \ar[ld, "\alpha", '] & 
\\
Sh &  & Sh,
\end{tikzcd}
\]
which means that $T_{\ell} = \alpha^* \circ \beta_*$, i.e., $T_{\ell} (x) = \sum_{y \in \alpha^{-1}(x)} \beta(y)$. In other words, we recover the definition given in \S \ref{shimura-chapter}.

Now take the fiber product $\mathcal{A}_{\ell} := \mathcal{A} \times_{Sh} Sh(\ell)$, which is the universal QM abelian surface over $Sh(\ell)$, equipped with a subgroup scheme $\mathcal{C}$ of $\mathcal{A}[\ell]$ that is an $\mathcal{O}_B$-module of order $\ell^2$. Consider the quotient $\mathcal{Q} := \mathcal{A}_{\ell}/\mathcal{C}$ with induced QM and level structure and the fiber products $\mathcal{A}_{\ell}^r$ and $\mathcal{Q}^r$ over $Sh(\ell)$. The action of the Hecke operator $T_{\ell}$ on the Kuga--Sato variety $\mathcal{A}^r$ is defined by
the commutative diagram
\[
\begin{tikzcd}
\mathcal{A}^r \arrow[d] & \mathcal{A}_{\ell}^r \arrow[dr] \arrow[rr, "\psi"]  \arrow[l, "p_1", '] && \mathcal{Q}^r \arrow[ld] \arrow[r, "p_2"]& \mathcal{A}^r \arrow[d]
\\
Sh && Sh(\ell) \arrow[ll, "\alpha", '] \arrow[rr, "\beta"] && Sh,
\end{tikzcd}
\]
where the two squares are cartesian, by the formula 
\[ T_{\ell} = {p_1}_* \circ \psi^* \circ p_2^*. \]
Write $\overline{\mathcal{A}}^r$ for the base change of $\mathcal{A}^r$ to $\overline{\mathbb Q}$. The correspondence $T_\ell$ just defined induces an endomorphism of the \'etale cohomology groups $H_{\et}^*(\overline{\mathcal{A}}^r, -)$, which will still be denoted by $T_{\ell}$.

The reader is advised to compare with \cite{Sch} for the definition of Hecke operators on Kuga--Sato varieties over modular curves and with \cite{EdVP} for the case of Shimura curves relative to ``$\Gamma_0$-type'' level structures.

\subsection{Generalized Kuga--Sato varieties}

Fix the QM abelian surface $A$ with $V_1(N^+)$-level structure and CM by $\mathcal{O}_K$ defined in \eqref{CMpoints}. Thanks to the assumption that $p$ splits in $K$, the surface $A$ is ordinary at
$p$. Our {\bf generalized Kuga--Sato variety} is the product $\mathcal{X}_{r} := \mathcal{A}^r \times A^r$. This enlarged Kuga--Sato variety will be the space where our arithmetic cycles will live. As a piece of notation, we shall write $\overline{\mathcal{X}}_{r}$ for the base change of $\mathcal{X}_{r}$ to $\overline{\mathbb Q}$.

The usual Hecke correspondence $T_{\ell}$ for a prime $\ell \nmid N^+D$ on the Kuga--Sato variety $\mathcal{A}^r$ induces a Hecke correspondence $T_{\ell} \times \id$ on $\mathcal{X}_{r}$, which will still be denoted by $T_{\ell}$. 

\subsection{Projectors on Kuga--Sato varieties} \label{projectors-sec}

We will define our algebraic cycles as graphs of morphisms of QM abelian surfaces. In order to make them homologically trivial, we will need to modify them by certain projectors associated with the generalized Kuga--Sato variety. Consider the projectors $P \in \text{Corr}_{Sh}(\mathcal{A}^r)$ and $\varepsilon_A \in \text{Corr}(A^r)$ defined in \cite[\S 6.1]{Brooks}.
Then
\[
PH^*_{\et}(\overline{\mathcal{A}}^r,\mathbb{Z}_p) \subseteq H^{k-1}_{\et}(\overline{\mathcal{A}}^r,\mathbb{Z}_p)
\]
and 
\[
\begin{split}
&\varepsilon_A H^i_{\et}(\overline{A}^{r},\mathbb{Z}_p) = 0\quad \text{if}\ i\neq k-2,\\[1mm] 
&\varepsilon_A H^{k-2}_{\et}(\overline{A}^{r},\mathbb{Z}_p)= \text{Sym}^{2r}eH^1(\overline{A},\mathbb{Z}_p).
\end{split}
\]
Consider the variety $\mathcal{X}_r$ together with the projector $\varepsilon = P \varepsilon_A \in \text{Corr}_{Sh}(\mathcal{X}_r)$.
Thanks to properties of the projectors, one has
\begin{equation} \label{eH}
\begin{split}
&\varepsilon H^{i}_{\et}\bigl(\overline{\mathcal{X}}_{r}, \mathbb{Z}_p\bigl) = 0 \quad \text{if}\ i\neq 2k-3, \\[1mm]
&\varepsilon H^{2k-3}_{\et}\bigl(\overline{\mathcal{X}}_{r}, \mathbb{Z}_p\bigr) \cong
P H^{k-1}_{\et}\bigl(\overline{\mathcal{A}}_{r},\mathbb{Z}_p\bigr) \otimes \text{Sym}^{2r}eH^1_{\et}\bigl(\overline{A},\mathbb{Z}_p\bigr).
\end{split}
\end{equation}
We can prove these relations using the K\"unneth decomposition. More precisely, the K\"unneth decomposition for $\mathcal{X}_{r} = \mathcal{A}_{r} \times A^{r}$ reads
\[
H^{i}_{\et}\bigl(\overline{\mathcal{X}}_{r}, \mathbb{Z}_p\bigr) \cong
\bigoplus_{n+s=i} H^{n}_{\et}\bigl(\overline{\mathcal{A}}_{r},\mathbb{Z}_p\bigr) \otimes H^s_{\et}\bigl(\overline{A}^{r},\mathbb{Z}_p\bigr).
\]
Indeed, the K\"unneth exact sequence for $\mathcal{A}_{r}$ and $A^{r}$ (for which we refer, for example, to \cite[Theorem 8.21]{Mi80} or \cite[\S 22]{LEC}) is
\[
\begin{split}
0 \longrightarrow \bigoplus_{n+s=i} H^{n}_{\et}\bigl(\overline{\mathcal{A}}_{r},\mathbb{Z}_p\bigr) &\otimes H^s_{\et}\bigl(\overline{A}^{r},\mathbb{Z}_p\bigr) \longrightarrow H^i_{\et}\bigl(\mathcal{X}_{r},\mathbb{Z}_p\bigr) \longrightarrow \\
&\longrightarrow \bigoplus_{n+s=i+1} \text{Tor}^{\mathbb{Z}_p}_1\Bigl(H^{n}_{\et}\bigl(\overline{\mathcal{A}}_{r},\mathbb{Z}_p\bigr), H^s_{\et}\bigl(\overline{A}^{r},\mathbb{Z}_p\bigr)\Bigr) \longrightarrow 0.
\end{split}
\]
Since $A$ is an abelian variety, by \cite[Theorem 15.1]{AV} (or \cite[Theorem 12.1]{AV1}), the $p$-adic cohomology of $\overline{A}^{r}$ is a free $\mathbb{Z}_p$-module, so the last term of the sequence above is $0$. As a consequence, we obtain the desired K\"unneth decomposition.

Now we want to apply the projector $\varepsilon$ and the twists. Since
$PH^{i}_{\et}(\overline{\mathcal{A}}_{r},\mathbb{Z}_p) = 0$ for $i \neq k-1$
and
$\varepsilon_A H^i_{\et}(\overline{A}^{r},\mathbb{Z}_p) = 0$ for $i\neq k-2$,
we deduce that 
\[
\varepsilon H^{i}_{\et}(\overline{\mathcal{X}}_{r}, \mathbb{Z}_p) = 0 \quad \text{for}\ i\neq 2k-3,
\]
as in this case all the terms on the right hand side of the K\"unneth decomposition vanish after applying the projectors. For $i=2k-3$ the only term in the sum in the right hand side of the K\"unneth decomposition that does not vanish after the application of the projectors is
$P H^{k-1}_{\et}(\overline{\mathcal{A}}^{r},\mathbb{Z}_p) \otimes \varepsilon_A H^{k-2}_{\et}(\overline{A}^{r},\mathbb{Z}_p)$, hence 
\[
\varepsilon H^{2k-3}_{\et}(\overline{\mathcal{X}}_{r}, \mathbb{Z}_p) =
P H^{k-1}_{\et}(\overline{\mathcal{A}}^{r},\mathbb{Z}_p) \otimes \text{Sym}^{2r}eH^1_{\et}(\overline{A},\mathbb{Z}_p),
\]
as desired.

\subsection{Galois representations and Kuga--Sato varieties over Shimura curves}

Let $V_f$ be the $2$-dimensional Galois representation attached to $f \in S_k^{\text{new}}(\Gamma_0(N))$ by Deligne (\cite{Del}) and let $V_f(k/2)$ be the self-dual twist of $V_f$. As explained, for example, in \cite[\S 2 and \S 3]{Nek92}, $V_f(k/2)$ can be realized as a direct summand of the $(k-1)$-st $p$-adic cohomology group of the Kuga--Sato variety over a certain modular curve.

Let $\phi$ be Euler's function and observe that the index of $\Gamma_{1,N^+}$ in $\Gamma_{0,N^+}$ divides $\phi(N^+)$ (see \cite[p. 4184]{Brooks}). From now on, we work under the following

\begin{assumption} \label{main-assumption}
$p\nmid N\phi(N^+)$.
\end{assumption}

Consider now the $r$-th Kuga--Sato variety $\mathcal{A}^{r}$ over the Shimura curve $Sh$. A similar construction can be performed to obtain the representation $V_f(k/2)$ from the \'etale cohomology group $H^{k-1}_{\et}(\overline{\mathcal{A}^{r}},\mathbb{Q}_p)(k/2)$ of $\mathcal{A}^{r}$. 
Namely, let $F$ be the finite extension of $\mathbb{Q}_p$ generated by the Fourier coefficients of $f$, whose valuation ring will be denoted by $\mathcal O_F$. Moreover, let $P$ be the projector from \S \ref{projectors-sec}. Under Assumption \ref{main-assumption}, one can define a Galois-equivariant surjection
\[
PH^{k-1}_{\et}\bigl(\overline{\mathcal{A}^{r}},\mathbb{Z}_p\bigr)(k/2) \longepi T,
\]
where $T$ is a suitable Galois-stable $\mathcal{O}_F$-lattice inside the $F$-vector space $V_f (k/2)$, whose definition can be found, for example, in \cite{Nek92} and \cite{Ota} (see also \cite{Nek93}). This can be done by adapting the arguments in \cite[\S 5 and Appendix 10.1]{IS} to our setting, which coincides with that of \cite{Brooks}. Since the modifications required are straightforward, we leave them to the interested reader. See also \cite[\S 3.3]{EdVP} for further details.

\subsection{The $p$-adic Abel--Jacobi map}
Recall that for a smooth projective variety $X$ of dimension $d$ over a field $L$ of characteristic $0$ one can associate the \emph{ $p$-adic \'etale Abel--Jacobi map}
\[
\AJ_{p,L}: {\CH^s(X/L)}_0 \longrightarrow 
H^1\Bigl(F,H^{2s-1}_{\et}(\overline{X},\mathbb{Z}_p)(s)\Bigr),
\]
where
$\CH^s(X/L)_0:= \ker(\text{cl}_{X})$ is
the group of homologically trivial cycles of codimension $s$  on $X$ defined over $L$ modulo rational equivalence, i.e. the kernel of the cycle map 
$
\text{cl}_X : \CH^s(X/L) \rightarrow H_{\et}^{2s}(\overline{X},\mathbb{Z}_p(s))^{G_L},
$
with $\mathbb{Z}_p(s)$ the $s$-th Tate twist of the $p$-adic sheaf $\mathbb{Z}_p$ and $G_L$ the absolute Galois group $\Gal(\overline{L}/L)$. See 
\cite[Chapter VI, \S 9]{Mi80} or \cite[\S 23]{LEC} for the definition of the cycle map and \cite[\S 3]{BDP}, for example, for the construction of the $p$-adic Abel--Jacobi map.

\subsection{Generalized Heegner cycles} \label{GHcycles}

For any morphism $\varphi: (A,i,\nu_{N^+}) \rightarrow (A',i',\nu'_{N^+})$ of abelian surfaces with QM by $\mathcal{O}_B$ and $V_1(N^+)$-level structure over a field $L\supseteq \mathbb{Q}$, we can consider its graph $\Gamma_{\varphi} \subseteq A \times A'$.
Consider the point $x$ in $Sh(L)$ corresponding to the class of $ (A',i',\nu'_{N^+})$. 
There is an embedding $A' = \mathcal{A}_{x} \hookrightarrow \mathcal{A}$ of the fiber of $\mathcal{A}$ above $x$ in $\mathcal{A}$ that induces an embedding $ i_{x} : (A')^{r} \hookrightarrow \mathcal{A}^{r}$. Via this embedding, we can view the $(r)$-power of the graph of $\varphi$ inside $\mathcal{A}^{r} \times A^r $:
\[
\Gamma_{\varphi}^{r} \subseteq A^r \times A'^r = A'^r \times A^r \stackrel{i_x \times \id}{\hookrightarrow}  \mathcal{A}^{r} \times A^r = \mathcal{X}_{r}.
\]
We define the \textbf{generalized Heegner cycle} $\Delta_{\varphi}$ associated with $\varphi$ as
\[
\Delta_{\varphi} := \varepsilon \Gamma_{\varphi}^{r} \in \text{CH}^{k-1}(\mathcal{X}_{r}/L),
\]
where $\varepsilon = P \varepsilon_A \in \text{Corr}^{r}_{Sh}(\mathcal{X}_{r},\mathcal{X}_{r})$ and $L$ is a field 
such that $ (A',i',\nu'_{N^+})$ and $\varphi$ are defined over $L$.
One needs to apply the projector $\varepsilon$ to make the cycle $\Delta_{\varphi}$ homologically trivial, that means that we want the image of $\Delta_{\varphi}$ via the cycle map 
$
\cl_{\mathcal{X}_{r}}: \text{CH}^{k-1}(\mathcal{X}_{r}/L)_{0} \rightarrow H^{2k-2}_{\et}(\overline{\mathcal{X}}_{r}, \mathbb{Z}_p(k-1))
$
to be trivial in order to apply the Abel--Jacobi map.
Indeed, the cycle $\Delta_{\varphi}$ is homologically trivial, because, thanks to equation (\ref{eH}), one has $\varepsilon H^{2k-2}_{\et}(\overline{\mathcal{X}}_{r}, \mathbb{Z}_p(k-1)) = 0$.
Therefore, we have that $\varepsilon \text{CH}^{k-1}(\mathcal{X}_{r}/L) \subseteq \text{CH}^{k-1}(\mathcal{X}_{r}/L)_{0}$
and we can consider the image of this cycle under the $p$-adic Abel--Jacobi map
\[
\text{AJ}_{p,L} : \text{CH}^{k-1}(\mathcal{X}_{r}/L)_{0} \longrightarrow H^1\big(L,H^{2k-3}_{\et}(\overline{\mathcal{X}}_{r}, \mathbb{Z}_p(k-1))\big).
\]
Applying the projector $\varepsilon$ in the construction of the Abel--Jacobi map (as in \cite[\S 3.1]{BDP} for Kuga--Sato varieties over modular curves) we obtain a $p$-adic Abel--Jacobi map
\[
\text{AJ}_{p,L} : \text{CH}^{k-1}(\mathcal{X}_{r}/L) \longrightarrow H^1\big(L,\varepsilon H^{2k-3}_{\et}(\overline{\mathcal{X}}_{r}, \mathbb{Z}_p(k-1))\big).
\]
Then, considering the twist in (\ref{eH}), one has
\[
\varepsilon H^{2k-3}_{\et}(\overline{\mathcal{X}}_{r}, \mathbb{Z}_p(k-1)) =
P H^{k-1}_{\et}(\overline{\mathcal{A}}^{r},\mathbb{Z}_p(k/2)) \otimes \text{Sym}^{2r}eH^1_{\et}(\overline{A},\mathbb{Z}_p)(r),
\]
so in the following we will see the Abel--Jacobi map as taking values in
\[
H^1\big(L,P H^{k-1}_{\et}(\overline{\mathcal{A}}^{r},\mathbb{Z}_p)(k/2) \otimes \text{Sym}^{2r}eH^1_{\et}(\overline{A},\mathbb{Z}_p)(r)\big).
\]

\subsection{A distinguished family of generalized Heegner cycles}

 With notation as in \S \ref{CMpoints}, start with the fixed QM abelian surface $A$. For any integer $c$ prime to $N^+$, take the multiplication-by-$c$ isogeny
\[
(A,i,\nu_{N^+}) \stackrel{\phi_c}{\longrightarrow} (A_c,i_c,\nu_{c,N^+}),
\]
which is an isogeny of QM abelian surfaces with $V_1(N^+)$-level structures.
For each class $[\mathfrak{a}]$ in $\Pic \mathcal{O}_c$, where the representative $\mathfrak{a}$ is chosen to be integral and prime to $N^+pc$, 
consider the isogeny
\[
\phi_{\mathfrak{a}} : A \longrightarrow A_{\mathfrak{a}},
\]
defined as the composition
\[
(A,i,\nu_{N^+}) \stackrel{\phi_c}{\longrightarrow} (A_c,i_c,\nu_{c,N^+}) \stackrel{\varphi_{\mathfrak{a}}}{\longrightarrow} (A_{\mathfrak{a}},i_{\mathfrak{a}},\nu_{\mathfrak{a},N^+}),
\]
and then the cycle 
\[
\Delta_{\phi_{\mathfrak{a}}} \in \text{CH}^{k-1}(\mathcal{X}_{r}/K_c).
\]
In fact, both $(A_{\mathfrak{a}},i_{\mathfrak{a}},\nu_{\mathfrak{a},N^+})$ and the isogeny $\phi_{\mathfrak{a}}$ are defined over the ring class field $K_c$ of $K$ of conductor $c$.

\subsection{Generalized Heegner classes} \label{GHclasses}

For any integer $c$ prime to $N^+$, consider the Abel--Jacobi map
\[
\text{AJ}_{p,K_c} : \text{CH}^{k-1}(\mathcal{X}_{r}/K_c)_0 \longrightarrow H^1\bigl(K_c,P H^{k-1}_{\et}(\overline{\mathcal{A}}^{r},\mathbb{Z}_p)(k/2) \otimes \text{Sym}^{2r}eH^1_{\et}(\overline{A},\mathbb{Z}_p)(r)\bigr).
\]
Because $V_{f}(k/2)$ can be realized as a quotient of $P H^{k-1}_{\et}(\overline{\mathcal{A}}^{r},\mathbb{Z}_p(k/2))$, then we can see the Abel--Jacobi map as a map
\[ \text{CH}^{k-1}(\mathcal{X}_{r}/K_c)_{0} \longrightarrow H^1(K_c,T \otimes \text{Sym}^{2r}eH^1_{\et}(\overline{A},\mathbb{Z}_p)(r)),
\]
where $T$ is the Galois stable lattice in $V_f(k/2)$.
Since
$\
eH^1_{\et}(\overline{A},\mathbb{Z}_p) \cong H^1_{\et}(\overline{E},\mathbb{Z}_p),
$
then we have a map
\[
\Phi_{K_c} : \text{CH}^{k-1}(\mathcal{X}_{r}/K_c) \longrightarrow H^1\bigl(K_c,T \otimes \text{Sym}^{2r}H^1_{\et}(\overline{E},\mathbb{Z}_p)(r)\bigr) \longrightarrow H^1\bigl(K_c,T \otimes S(E)\bigr),
\]
where, as in \cite[\S 4.2]{CH}, we set $S(E) := \text{Sym}^{2r}T_p(\overline{E})(-r)$.

We define the \textbf{generalized Heegner class} $z_{\mathfrak{a}}$, associated with an ideal $\mathfrak{a}$ of $\mathcal{O}_c$, to be
\[
z_{\mathfrak{a}} := \Phi_{K_c}(\Delta_{\phi_{\mathfrak{a}}}).
\]

\subsection{$\chi$-components} \label{chi-comp}

As in \cite{CH}, we want to define ``$\chi$-components'' of the generalized Heegner classes and construct classes $z_{c,\chi} \in H^1(K_c, T \otimes \chi)$, for $\chi$ an anticyclotomic Galois character. For this we will use $S(E) = \text{Sym}^{2r}T_p(\overline{E})(-r)$ appearing in the image of the Abel--Jacobi map and we will do the same work as \cite{CH}. 
So, closely following \cite[\S 4.4]{CH}, let us start with a positive integer $c_0$ prime to $pN^+$, and let
$\chi: \Gal(K_{c_0p^{\infty}}/K) \rightarrow \mathcal{O}_F^{\times}$ (possibly enlarging $F$ so that $\Ima(\chi) \subseteq \mathcal{O}_F^{\times}$)
be a locally algebraic anticyclotomic character of infinity type $(j, -j)$ with
$-k/2 < j < k/2$ and conductor $c_0p^s\mathcal{O}_K$.

Recall that $E = \mathbb{C}/\mathcal{O}_K$ (cf. \S 1.5); note that $E$ is denoted with $A$ in \cite{CH}. Consider the abelian variety
$W_{/K} := \Res_{K_1/K}E$,
obtained by Weil restriction of $E$ from $K_1$ to $K$, defined as the product
\[
W = \prod_{\sigma \in \Gal(K_1/K)} E^{\sigma},
\] 
where $E^{\sigma}$ is the curve determined by the polynomials obtained by applying $\sigma$ to the coefficients of the polynomials defining $E$.
The Weil restriction $W$ is again a CM abelian variety but over $K$ and of dimension $[K_1:K]$. The endomorphism ring of $W$, $M := \End_K(W) \otimes \mathbb{Q}$, is a product of CM fields over $K$ and $\dim W = [M : K] = [K_1 : K]$ (see \cite[\S 1]{Rub} and \cite{Wi} for a general introduction to the Weil restriction of abelian varities). Since
\[
T_p(W) = \prod_{\sigma \in \Gal(K_1/K)} T_p(E^{\sigma}) = \text{Ind}^{G_K}_{G_{K_1}}T_p(E),
\]
viewing Tate modules as Galois representations, where $\text{Ind}$ is the induced representation, we have an inclusion $T_p(E) \hookrightarrow T_p(W)$.
Define the $G_K$-module
\[
S(W):= \Sym^{2r} T_p(W)(-r) \otimes_{\mathbb{Z}_p} \mathcal{O}_F = \text{Ind}^{G_K}_{G_{K_1}} S^r(E) \otimes_{\mathbb{Z}_p} \mathcal{O}_F.
\]
By the discussion in \cite[\S 4.4]{CH}, there exists a finite order anticyclotomic character
$\chi_t$ such that $\chi$ can be realized as a direct summand of $S(W) \otimes \chi_t$ as $G_K$-modules. Denote by 
\[
e_{\chi}: S(W) \otimes \chi_t \longepi \chi
\]
the corresponding $G_K$-equivariant projection. The character $\chi_t$ is unique up to multiplication by a character of $\Gal(K_1/K)$ and it has the same conductor as $\chi$. 

Since $T_p(E) \hookrightarrow T_p(W)$ and then $S(E) \hookrightarrow S(W)$, we can see the classes $z_{\mathfrak{a}}$ as elements of $H^1(K_c,T \otimes S(W))$, where $\mathfrak{a}$ is a fractional ideal of $\mathcal{O}_c$ with $c$ divisible by $c_0p^s$.

For each integer $c$ divisible by the conductor $c_0p^s$ of $\chi$, put \[z_{c} \otimes \chi_t := z_c \in H^1(K_c,T \otimes S(W) \otimes \chi_t),\] through the map $H^1(K_c,T \otimes S(W)) \rightarrow H^1(K_c,T \otimes S(W) \otimes \chi_t)$ and define
\begin{equation} \label{chicomp}
z_{c,\chi} := (\id \otimes e_{\chi}) (z_{c} \otimes \chi_t) \in H^1(K_c,T \otimes \chi),
\end{equation} 
the \textbf{$\chi$-component of the class $z_c$}.

\subsection{Compatibility properties}

 First we study compatibility properties of the generalized Heegner classes defined in section \ref{GHclasses}, by examining the action of the Hecke operators, and proving results analoguos to 
\cite[Lemma 4.3 and Proposition 4.4]{CH}.

 Let $I(D_K)$ denote the group of fractional ideals of $K$ that are prime to $D_K$ and let
 $\tilde{\kappa}_E : I(D_K) \rightarrow M^{\times}$
 be the CM character associated to $W$ (as in \cite[\S 4.4]{CH}). 
Denote by $\kappa_E : G_K \rightarrow\mathcal{O}_F^{\times}$ the $p$-adic avatar of $\tilde{\kappa}_E$, possibly enlarging $F$ so that $M \subseteq F$.
 
 \begin{prop}
 	Let $\mathfrak{a}$, $\mathfrak{b}$ be fractional ideals in $\mathcal{O}_c$ prime to $cN^+D_K$. Suppose that $\mathfrak{a}\mathcal{O}_K$ is trivial in $\Pic \mathcal{O}_K$ and put $\alpha:= \tilde{\kappa}_E(\mathfrak{a}) \in K^{\times}$. 
 	Then
 	\[
 	(\id \times \alpha)^*\Delta_{\phi_{\mathfrak{b}}}^{\sigma_{\mathfrak{a}}} = \Delta_{\phi_{\mathfrak{a}\mathfrak{b}}},
 	\]
 	where $\sigma_{\mathfrak{a}} \in \Gal(K^{ab}/K_1)$ corresponds to $\mathfrak{a}$ through the Artin reciprocity map.
 \end{prop}
 
 \begin{proof}
Recall that the QM abelian surfaces $A$ and $A_{\mathfrak{a}}$ are respectively the self-products of elliptic curves $E \times E$ ad $E_{\mathfrak{a}} \times E_{\mathfrak{a}}$ for each fractional ideal $\mathfrak{a}$.
Note that the isogenies $\phi_{\mathfrak{a}} : A \twoheadrightarrow A_{\mathfrak{a}}$ of QM abelian surfaces are the self-product of the isogenies $E \twoheadrightarrow E_{\mathfrak{a}}$, denoted with $\varphi_{\mathfrak{a}}$ by Castella and Hsieh and used by them to define their Heegner cycles (cf. \cite[\S 4.1]{CH}), and also $\alpha \in K^{\times}$ acts on $A$ (and $A_{\mathfrak{a}}$) as the matrix multiplication by $\left( \begin{smallmatrix}
\alpha & 0 \\ 0 & \alpha
\end{smallmatrix}\right)$, hence as the multiplication by $\alpha$ in each component $E$ (and $E_{\mathfrak{a}}$). Because of this, the proof of \cite[Lemma 4.3]{CH} works also in our case, working component by component.  So we obtain that
\[
(\id \times \alpha)^*\Gamma_{\mathfrak{b}}^{\sigma_{\mathfrak{a}}} =  (\alpha \times \id)_*\Gamma_{\mathfrak{b}}^{\sigma_{\mathfrak{a}}} = \Gamma_{\alpha \circ \phi_{\mathfrak{b}}^{\sigma_{\mathfrak{a}}}} = \Gamma_{\phi_{\mathfrak{a}\mathfrak{b}}} = \Gamma_{\mathfrak{a}\mathfrak{b}}.
\]
Because $x(\mathfrak{b})^{\sigma_{\mathfrak{a}}} = (A_{\mathfrak{b}},i_{\mathfrak{b}},\nu_{\mathfrak{b},N^+})^{\sigma_{\mathfrak{a}}} = \mathfrak{a} \star (A_{\mathfrak{b}},i_{\mathfrak{b}},\nu_{\mathfrak{b},N^+}) = (A_{\mathfrak{ab}},i_{\mathfrak{ab}},\nu_{\mathfrak{ab},N^+}) = x_{\mathfrak{ab}}$ as points in $Sh$, we have that the immersions $i_{x(\mathfrak{b})^{\sigma_{\mathfrak{a}}}}$ and $i_{x(\mathfrak{ab})}$ are equal, thus, taking the $r$-power and applying the projector $\varepsilon$, we have that 
\[
(\id \times \alpha)^*\Delta_{\mathfrak{b}}^{\sigma_{\mathfrak{a}}} =  \Delta_{\mathfrak{a}\mathfrak{b}},
\]
and we are done.
\end{proof}
 
\begin{prop} \label{4.4}
Suppose that $p \nmid c$. For all $ n > 1$, one has
\[
T_p z_{cp^{n-1}} = p^{k-2} z_{p^{n-2}} + \Cor_{K_{cp^n}/K_{cp^{n-1}}}(z_{cp^n}).
\]
For $\ell \nmid c$ that is inert in $K$, one has
\[
T_{\ell}z_c = \Cor_{K_{c\ell}/K_c}(z_{c\ell}).
\]
\end{prop}
 
\begin{proof}
The operator $T_{p}$ acts on the cycle $\Delta_{\phi_{cp^{n-1}}}$ coming from the isogeny $\phi_{cp^{n-1}}: A \rightarrow A_{cp^{n-1}}$ of QM abelian surfaces with $V_1(N^+)$-level structure and CM by $K$ in the following way
\[
T_{p} \Delta_{\phi_{cp^{n-1}}} = \sum_{i=1}^{p+1} \Delta_{\phi_i},
\]
where the isogenies $\phi_i: A_i \rightarrow A_{cp^{n-1}}$ are $p^2$-isogenies of QM abelian surfaces with $V_1(N^+)$-level structures. 
These isogenies correspond to the $p + 1$ sublattices of $\mathcal{O}_{cp^{n-1}} \times \mathcal{O}_{cp^{n-1}}$ that are invariant under the action of $\mathcal{O}_B$. 
Since such a sublattice $L$ is determined by $eL$, one can work with sublattices of $\mathcal{O}_c$ of index $p$.
Therefore one can rearrange the computation in the proof of \cite[Proposition 4.4]{CH} to obtain the formula in the statement, using an analogue of \cite[Lemma 4.2]{CH} in this case.
The second part of the proposition can be proved analogously.
 \end{proof}
 
The following result is analogous to \cite[Proposition 4.5]{CH}.
\begin{prop} \label{4.5}
	Let $\mathfrak{a}$ be a fractional ideal of $\mathcal{O}_c$ prime to $cN^+D_K$. Then
\[
\chi_t(\sigma_{\mathfrak{a}}) (\emph{id} \otimes e_{\chi}) z_{c}^{\sigma_{\mathfrak{a}}} = \chi(\sigma_{\mathfrak{a}}) \chi_{\cyc}^{-r}(\sigma_{\mathfrak{a}}) (\emph{id} \otimes e_{\chi}) z_{\mathfrak{a}},
\]
where $\chi_{\cyc}$ is the $p$-adic cyclotomic character.
\end{prop}

\begin{proof}
Denote by $\sigma_{\mathfrak{a}} \in \Gal(K_c/K)$ the image of $\mathfrak{a}$ under the classical Artin reciprocity map. 
We have 
\[
(\text{id} \times \varphi_{\mathfrak{a}\mathcal{O}_K})_* \Gamma_{\phi_{\mathfrak{a}}} 
= (\text{id} \times \phi_{\mathfrak{a}})_* \bigl\{(\phi_{\mathfrak{a}}(z),z) \mid z \in A \bigr\} 
= \bigl\{ (\varphi_{\mathfrak{a}} \phi_c(z),\varphi_{\mathfrak{a}\mathcal{O}_K}z) \mid z \in A \bigr\}
= \Gamma_{\phi_c}^{\sigma_{\mathfrak{a}}},
\]
as $\sigma_{\mathfrak{a}}(z) = \varphi_{\mathfrak{a}\mathcal{O}_K}(z) $ for any $z \in A(\mathbb{C})$ and $\sigma_{\mathfrak{a}}(z) = \varphi_{\mathfrak{a}}(z) $ for any $z \in A_c(\mathbb{C})$.
Because $x_c^{\sigma_{\mathfrak{a}}} = x_{\mathfrak{a}}$, where  $x_{\mathfrak{a}} = (A_{\mathfrak{a}},i_{\mathfrak{a}},\nu_{\mathfrak{a},N^+})$ and $x_c=x_{\mathcal{O}_c}$, one has $i_{x_c^{\sigma_{\mathfrak{a}}}} = i_{x_{\mathfrak{a}}}$. Hence, applying $\varepsilon$ to $\big((\text{id} \times \varphi_{\mathfrak{a}\mathcal{O}_K})_* \Gamma_{\phi_{\mathfrak{a}}} \big)^r= \big(\Gamma_{\phi_c}^{\sigma_{\mathfrak{a}}}\big)^r$, we obtain 
\[
(\text{id} \times \varphi_{\mathfrak{a}\mathcal{O}_K})_* \Delta_{\phi_{\mathfrak{a}}} = \Delta_{\phi_c}^{\sigma_{\mathfrak{a}}}.
\]
Then
\[
z_{c}^{\sigma_{\mathfrak{a}}} = \Phi_{K_c}(\Delta_{\phi_c}^{\mathfrak{a}}) = \Phi_{K_c}((\text{id} \times \varphi_{\mathfrak{a}})_* \Delta_{\phi_{\mathfrak{a}}})= (\text{id} \times \varphi_{\mathfrak{a}\mathcal{O}_K})_* \Phi_{K_c}(\Delta_{\phi_{\mathfrak{a}}})
= (\text{id} \times \varphi_{\mathfrak{a}\mathcal{O}_K})_* z_{\mathfrak{a}}.
\]
Observe that $\varphi_{\mathfrak{a}\mathcal{O}_K}$ acts on $\text{Sym}^{2r}eT_p(\overline{A}) \cong \text{Sym}^{2r}T_p(\overline{E})$ as its first component $\lambda_{\mathfrak{a}\mathcal{O}_K} : E \rightarrow E/E[\mathfrak{a}\mathcal{O}_K]=E_{\mathfrak{a}\mathcal{O}_K}$. From the proof of \cite[Proposition 4.5]{CH}, we know that $\lambda_{\mathfrak{a}\mathcal{O}_K}$ acts on $\text{Sym}^{2r}H^1_{\et}(\overline{W},\mathbb{Z}_p))$ 
as the push-forward $[\tilde{\kappa}_E(\mathfrak{a}\mathcal{O}_K)]_*$ of $\tilde{\kappa}_E(\mathfrak{a}\mathcal{O}_K) \in \End(W)$, which in turn induces the Galois action $\sigma_{\mathfrak{a}}$.
Since $e_{\chi}$ commutes with the action of $G_K$, there are equalities
\[
\begin{split}
\chi_{\cyc}^{r}(\sigma_{\mathfrak{a}}) \chi_{t}(\sigma_{\mathfrak{a}}) e_{\chi}\big(\sigma_{\mathfrak{a}} \otimes \id \otimes \id \cdot y \big) &= e_{\chi}\big((\sigma_{\mathfrak{a}} \otimes  \chi_{\cyc}^{r}(\sigma_{\mathfrak{a}}) \otimes \chi_{t}(\sigma_{\mathfrak{a}})) (y)\big)\\ &= e_{\chi}(\sigma_{\mathfrak{a}} \cdot y) = \chi(\sigma_{\mathfrak{a}})e_{\chi}(y),
\end{split}
\]
for any $y = y \otimes 1 \otimes 1 \in S(W) \otimes \chi_t =
\text{Sym}^{2r}H^1_{\et}(\overline{W},\mathbb{Z}_p) \otimes \chi_{\cyc}^{r} \otimes \chi_t$.
Therefore, viewing $z_{c}^{\sigma_{\mathfrak{a}}}, z_{\mathfrak{a}} \in H^1(K_c,T \otimes S(E)) \subseteq H^1(K_c,T \otimes S(W))$ as in \S \ref{chi-comp} and letting $z_c^{\sigma_{\mathfrak{a}}} := z_c^{\sigma_{\mathfrak{a}}} \otimes 1$, $z_{\mathfrak{a}} := z_{\mathfrak{a}} \otimes 1 \in H^1(K_c,T \otimes S(W) \otimes \chi_t)$, we get
\[
\begin{split}
(\text{id} \otimes e_{\chi}) z_c^{\sigma_{\mathfrak{a}}} &= (\text{id} \otimes e_{\chi}) (\text{id} \otimes [\tilde{\kappa}_E(\mathfrak{a}\mathcal{O}_K)]_*) z_{\mathfrak{a}}
\\
&= \chi_{\cyc}^{-r}(\sigma_{\mathfrak{a}}) \chi_{t}^{-1}(\sigma_{\mathfrak{a}}) (\text{id} \otimes e_{\chi}) (\text{id} \otimes \sigma_{\mathfrak{a}}) z_{\mathfrak{a}}
\\
&= \chi_{\cyc}^{-r}(\sigma_{\mathfrak{a}}) \chi_{t}^{-1}(\sigma_{\mathfrak{a}}) \chi(\sigma_{\mathfrak{a}})(\text{id} \otimes e_{\chi}) z_{\mathfrak{a}},
\end{split}
\]
which completes the proof.
\end{proof}

Finally, we conclude with two more propositions. As one can rearrange the proofs of analogous results in \cite{CH} to work also in our case, as we did in the proofs of the previous results, we will skip details.

\begin{prop} \label{4.6}
Let $\tau$ be the complex conjugation. Then
\[
z_{c,\chi}^{\tau} = w_f \cdot \chi(\sigma) \cdot (z_{c,\chi^{-1}})^{\sigma},
\]
for some $\sigma \in \Gal(K_c/K)$, with $w_f = \pm 1$ the Atkin--Lehner eigenvalue of $f$.
\end{prop}

\begin{proof} Proceed as in the proof of \cite[Lemma 4.6]{CH}. \end{proof}

\begin{prop} \label{4.7}
Let $\ell$ be a prime inert in $K$ such that $\ell \nmid cND_K$. Let $\overline{\lambda}$ be a prime of $\overline{\mathbb{Q}}$ above $\ell$. Denote by $K_{c\ell,\lambda}$ and $K_{c,\lambda}$ the completions of $K_{c\ell}$ and $K_c$, respectively, at the prime above $\ell$ determined by $\overline{\lambda}$, and write $\loc_{\lambda}$ for the localization map. Then
\[
\loc_{\lambda}(z_{c\ell,\chi}) = \Res_{K_{c\ell,\lambda}/K_{c,\lambda}}\bigl(\loc_{\lambda}(z_{c,\chi})^{\frob_{\ell}}\bigr).
\]
\end{prop}

\begin{proof} Proceed as in the proof of \cite[Lemma 4.7]{CH}. \end{proof}

\section{A $p$-adic Gross--Zagier formula}

We want to relate our $p$-adic $L$-function to generalized Heegner cycles. More precisely, we want our $p$-adic $L$-function to satisfy a $p$-adic Gross--Zagier formula that relates Bloch--Kato logarithms of generalized Heegner cycles associated with characters $\chi : \Gal(K_{p_{\infty}}/K)  \rightarrow \mathcal{O}_{\mathbb{C}_p}^{\times}$ of infinity type $(j,-j)$ with $ -k/2 < j < k/2$, with its values, i.e., we look for a formula of the shape
\[
\LL_f(\psi)(\hat{\phi}) = (\text{something}) \cdot \braket{\log(z_{\chi}), *},
\]
where $\hat{\phi} : \Gal(K_{p_{\infty}}/K)  \rightarrow \mathcal{O}_{\mathbb{C}_p}^{\times}$ is the $p$-adic avatar of a Hecke character $\phi$ of infinity type $(-k/2-j,k/2+j)$ and $\chi = \hat{\psi}^{-1} \hat{\phi}^{-1}$.
To establish a Gross--Zagier formula we will link our $p$-adic $L$-function to the differential operator $\theta = t \frac{d}{dt}$ on the Serre--Tate coordinates and then we will use some results of Brooks to obtain a key formula relating this operator $\theta$ applied to the modular form calculated on CM points with our generalized Heegner cycles.

\subsection{The Bloch--Kato logarithm map}\label{BKlog}

Let $F,L$ be finite extensions of $\mathbb{Q}_p$ and $V$ a $F$-vector space with a continuous linear $G_L := \Gal(\overline{L}/L)$-action.
Set $DR_L(V) := H^0(L,B_{\dR}\otimes V) = (B_{\dR} \otimes V)^{G_L}$, where $B_{\dR}$ is the Fontaine's de Rham periods ring.
If $V$ is a de Rham representation (i.e., $\dim_{F} V = \dim_L DR_L(V)$), then we can consider the Bloch--Kato exponential map
\[
\exp_{BK}: \frac{DR_L(V)}{\Fil^0DR_L(V)} \stackrel{\cong}{\longrightarrow} H^1_f(L,V),
\]
that is the connecting homomorphism of the long exact sequence in cohomology coming from the short exact sequence
\[
0 \longrightarrow V \longrightarrow B_{\text{crys}}^{\phi=1} \otimes V \oplus \Fil^0 B_{\dR} \otimes V \longrightarrow B_{\dR} \longrightarrow 0.
\]
See \cite[Definition 3.10, Corollary 3.8.4, Proposition 1.17]{BK}. Here $H^1_f(L,V)$ is the Bloch--Kato finite part
\[
H^1_f(L,V) := \ker\Big(H^1(L,V) \longrightarrow H^1_f(L,V \otimes B_{\crys})\Big),
\]
where $B_{\crys}$ is Fontaine's crystalline period ring.
Consider now the inverse of this map, the Bloch--Kato logarithm map
\[
\log_{BK}: H^1_f(L,V) \stackrel{\cong}{\longrightarrow} \frac{DR_L(V)}{\Fil^0DR_L(V)}.
\]
Since the long exact sequence in cohomology is functorial, the Bloch--Kato logarithm map is functorial as well, i.e., for any $F$-linear and $G_L$-equivariant morphism $V \rightarrow V'$ there is a commutative square
\[
\begin{tikzcd}
H^1_f(L,V) \arrow[r, "\log_{BK}"] \arrow[d] &  \frac{DR_L(V)}{\Fil^0DR_L(V)} \arrow[d]\\
H^1_f(L,V') \arrow[r, "\log_{BK}"] &  \frac{DR_L(V')}{\Fil^0DL_L(V')}.
\end{tikzcd}
\]
Denote by $V^*:=\Hom_F(V,F)$ the dual of $V$ and consider the perfect de Rham pairing
\[
\braket{-,-} : DR_L(V) \times DR_L\bigl(V^*(1)\bigr) \longrightarrow \mathbb{C}_p.
\]
Then we can consider the logarithm $\log_{BK}$ as a map
\begin{equation}\label{logdef}
\log_{BK}: H^1_f(L,V) \stackrel{\cong}{\longrightarrow} \frac{DR_L(V)}{\Fil^0DL_F(V)} \cong \Bigl(\Fil^{0} DR_L\bigl(V^*(1)\bigr)\Bigr)^{\vee},
\end{equation}
which is again functorial.

\subsection{The operator $\theta$} \label{theta}

Recall the differential operator $\theta := t \frac{d}{dt}$ on the power series ring $W[[t-1]]$. For a negative exponent $j$, one can define
\[
\theta^j:= \lim_{i \to \infty} \theta^{j+(p-1)p^i}.
\]
Indeed, this limit makes sense because of \cite[Proposition 4.18]{Brooks}, which implies that $\theta^{j+(p-1)p^m}F \equiv \theta^{j+(p-1)p^n}F \mod p^{n+1}$ for $m\geq n\gg0$ and $F \in W[[t-1]]$, so $\bigl| \theta^{j+(p-1)p^m}F - \theta^{j+(p-1)p^n}F \bigr|_p \leq 1/p^{n+1}$.

For a positive integer $m$ we know from equation (\ref{x^m}) that 
\[[x^m]F = \theta^m F.\] 
The same formula can be obtained by direct computation also for $m$ negative. Furthermore, for any locally constant function $\phi \in \mathcal{C}(\mathbb{Z}_p,\W)$, the formula
\begin{equation} \label{x^mtheta_neg}
\int_{\mathbb{Z}_p^{\times}} \phi(x)x^m dF = {\theta ^m ([\phi]F)|} _{t=1}
\end{equation}
holds also for a negative integer $m$.

\subsection{CM periods and de Rham cohomology classes}\label{6.3}

Recall that $A$ has a model defined over $\mathcal{V} := \W \cap K^{\ab}$.
Fix a non-vanishing global section $\omega_A$ of the line bundle $e\Omega_{A/\mathcal{V}}$ on $A$, defined over $\mathcal{V}$.
Define a $p$-adic period $\Omega_p \in \mathbb{C}_p$ by the rule
\[
\omega_{A} = \Omega_p \hat{\omega}_{A}
\]
where $\hat{\omega}_A$ denote the section $\omega_P$ determined for $A$ as in the last lines of \ref{pmodform} (which depends upon the $p^{\infty}$- level structure on $A$).

Now, take a finite extension $L$ of $\mathbb{Q}_p$ containing
$K_1$ and recall the fixed non-vanishing differential $\omega_{A}$ in $eH^0(A,\Omega_{A/L})$. We can see it as an element of $eH^1_{\dR}(A/L)$. This determines another element $\eta_A$, as in the last lines of \cite[\S 2.8]{Brooks}, so that $\omega_A^i \eta_A^{2r-i}$ is a basis for $\Sym^{2r}eH^1_{dR}(A/L)$, when $i= 0, \dots, 2r$.
We can arrange $\omega_A$ and $\eta_A$ in a way such that they correspond to the elements $\omega_E, \eta_E \in H^1_{\dR}(E/L)$ defined \cite[\S 4.5]{CH}, through the isomorphism $eH^1_{\dR}(A/L) \cong H^1_{\dR}(E/L)$, keeping in mind that the elliptic curve $E$ is denoted by $A$ in \cite{CH}.
%
%
%
%

\subsection{A key formula} \label{keyform}

First we want to prove a formula that relates the operator $\theta$ applied to the modular form calculated on CM points with our generalized Heegner cycles.
The proof of this formula is the same as the proof of \cite[(4.9)]{CH} but with \cite[Proposition 3.24 and Lemmas 3.23, 3.22]{BDP} replaced by \cite[Theorem 7.3, Lemma 8.6 and Proposition 7.4]{Brooks}.

Consider the statement of \cite[Proposition 7.4]{Brooks}. In that statement, $F$ is a finite extension of $\mathbb{Q}_p$, containing
$K_1$, such that the cycle $\Delta_{\varphi}$ associated with an isogeny $\varphi$ (as in \ref{GHcycles}) is defined over $F$, $\omega_f$ is defined as in \cite[\S2.7]{Brooks} and can be seen as an element in $\Fil^{k-1}PH^{k-1}_{\dR}(\mathcal{A}^r/F)$ (as in \cite{Brooks}; see also \cite[Corollary 2.3]{BDP}).
Note that the map $\AJ_F$ in \cite[Proposition 7.4]{Brooks} is the Abel--Jacobi map 
	\[
	\AJ_F : CH^{k-1}(\mathcal{X}_r/F)_{0} \longrightarrow (\Fil^{k-1}\varepsilon H^{4r+1}_{\dR}\bigl(\mathcal{X}_r/F)\bigr)^{\vee}
	\]
	defined in \cite[\S 6.3]{Brooks}, which is the composition of the usual Abel--Jacobi map
	\[
	CH^{k-1}(\mathcal{X}_r/F)_{0} \longrightarrow H^1_f\Bigl(F,\varepsilon H^{4r+1}_{\et}(\overline{\mathcal{X}}_r,\mathbb{Q}_p)(k-1)\Bigr)
	\]
	with the Bloch--Kato logarithm map
	\[
	\log_{\BK}:H^1_f\Bigl(F,\varepsilon H^{4r+1}_{\et}(\overline{\mathcal{X}}_r,\mathbb{Q}_p)(k-1)\Bigr) \longrightarrow
	\frac{DR_F\Bigl(\varepsilon H^{4r+1}_{\et}(\overline{\mathcal{X}}_r,\mathbb{Q}_p)(k-1)\Bigr)}{\Fil^0\Bigl(DR_F\bigl(\varepsilon H^{4r+1}_{\et}(\overline{\mathcal{X}}_r,\mathbb{Q}_p)(k-1)\bigr)\Bigr)}
	\]
for the Galois representation $V = \varepsilon H^{4r+1}_{\et}\bigl(\overline{\mathcal{X}}_r,\mathbb{Q}_p\bigr)(k-1)$.
Actually, the image of the $p$-adic Abel--Jacobi map is contained in the subgroup \[H^1_f(F,\varepsilon H^{4r+1}_{\et}(\overline{\mathcal{X}}_r,\mathbb{Q}_p)(k-1))\] of $H^1(F,\varepsilon H^{4r+1}_{\et}(\overline{\mathcal{X}}_r,\mathbb{Q}_p)(k-1))$, see \cite[Theorem 3.1]{Nek00}.
Since the comparison isomorphism 
	\[
	\Phi:DR_F\bigl(\varepsilon H^{4r+1}_{\et}(\overline{\mathcal{X}}_r,\mathbb{Q}_p)\bigr) \stackrel{\cong}\longrightarrow \varepsilon H^{4r+1}_{\dR}(\mathcal{X}_r/F),
	\]
	for which we refer to
	\cite{Fa}, is compatible with the filtrations, and since the Tate twist shifts the filtration, there
	is a functorial isomorphism
	\[
	\Phi:\frac{DR_F\bigl(\varepsilon H^{4r+1}_{\et}(\overline{\mathcal{X}}_r,\mathbb{Q}_p)(k-1)\bigr)}{\Fil^0 DR_F\bigl(\varepsilon H^{4r+1}_{\et}(\overline{\mathcal{X}}_r,\mathbb{Q}_p)(k-1)\bigr)} \stackrel\cong\longrightarrow \frac{\varepsilon H^{4r+1}_{\dR}(\mathcal{X}_r/F)}{\Fil^{k-1}\varepsilon H^{4r+1}_{\dR}(\mathcal{X}_r/F)}.
	\]
	On the other hand, by Poincar\'e duality, there is an isomorphism 
	\begin{equation} \label{poincare-eq}
	\frac{\varepsilon H^{4r+1}_{\dR}(\mathcal{X}_r/F)}{\Fil^{k-1}\varepsilon H^{4r+1}_{\dR}(\mathcal{X}_r/F)} \cong \Fil^{k-1} \varepsilon H^{4r+1}_{\dR}(\mathcal{X}_r/F) ^{\vee}.
	\end{equation}
	Thus, writing $V := \varepsilon H^{4r+1}_{\et}(\overline{\mathcal{X}}_r,\mathbb{Q}_p)(k-1)$ in this case, we can view $\log_{BK}$ as a map
	\[
	\log_{BK}: H^1_f(F,V) \stackrel{\cong}{\longrightarrow} \frac{DR_F(V)}{\Fil^0DR_F(V)} \cong \Fil^{k-1} \varepsilon H^{4r+1}_{\dR}(\mathcal{X}_r/F) ^{\vee}.
	\]
For more details, see \cite[\S 6.3]{Brooks}.	
 
Recall the elements $\omega_A^i \eta_A^{2r-i}\in \Sym^{2r}eH^1_{dR}(A/F)$ for $i= 0, \dots, 2r$. Then 
\[
\omega_f \otimes \omega_A^i \eta_A^{2r-i} \in \Fil^{r+1}P H^{k-1}_{\dR}(\mathcal{A}^r/F) \otimes \Sym^{2r} eH^1_{\dR}(A/F) \cong \Fil^{k-1}\varepsilon H^{4r+1}_{\dR}(\mathcal{X}_r/F).
\]
See \cite{Brooks}, or \cite[pp. 1060--1062]{BDP} for an explanation in the case of modular curves.

\begin{rmk}
	We can restate \cite[Proposition 7.4]{Brooks} as
\[
\big\langle\log_{BK}(\xi_{\varphi}),\omega_f \otimes \omega_A^i \eta_A^{r-i}\big\rangle = d^i G_i(A',t',\omega'),
\]
	where $\xi_{\varphi}$ is the image of $\Delta_{\varphi}$ under the usual Abel--Jacobi map.
	Indeed, isomorphism \eqref{poincare-eq} sends $\log_{BK}(\xi_{\varphi}) \in \frac{\varepsilon H^{4r+1}_{\dR}(\mathcal{X}_r/F)}{\Fil^{k-1}\varepsilon H^{4r+1}_{\dR}(\mathcal{X}_r/F)}$ to $\big\langle\log_{BK}(\xi_{\varphi}),-\big\rangle$.
\end{rmk}

\begin{rmk} \label{log}
We want to apply the Bloch--Kato logarithm to the localization at $\mathfrak{p}$ of the generalized Heegner classes $z_{\mathfrak{a}}$ (defined in \S \ref{GHclasses}), regarded as elements of $H^1(K_c, T \otimes \chi)$ via the map $\id \otimes e_{\chi}$ (defined in \S \ref{chi-comp}), and also to the classes $z_{\chi} \in H^1(K,T \otimes \chi)$ that will be defined in \eqref{z.chi}. Recall the choice of the prime $\mathfrak{p}$ of $K$ over $p$. If $M$ is a finite extension of $K$, then we write $M_{\mathfrak{p}}$ for the completion of $M$ at the prime over $\mathfrak{p}$ determined by the fixed embedding $i_p: \overline{\mathbb{Q}} \hookrightarrow \mathbb{C}_p$ and denote by $\loc_{\mathfrak{p}}$ the localization map
\[
\loc_{\mathfrak{p}} : H^1(M,V) \longrightarrow H^1(M_{\mathfrak{p}}, V),
\]
for any representation $V$ of $G_M$.
Thus, denote by $\log_{\mathfrak{p}}$ the composition of the localization at $\mathfrak{p}$ with $\log_{BK}$, so that $\log_{\mathfrak{p}}(z_{\mathfrak{a}})$ (respectively, $\log_{\mathfrak{p}}(z_{\chi})$) is the image of the localization $ \loc_{\mathfrak{p}}(z_{\mathfrak{a}}) \in H^1_f(K_{c,\mathfrak{p}},T \otimes \chi)$ (respectively, of the localization $ \loc_{\mathfrak{p}}(z_{\chi}) \in H^1_f(K_{\mathfrak{p}},T \otimes \chi)$) via the Bloch--Kato logarithm. 
%
%
%
Since $V_f$ can be realized as a quotient of $P H^{k-1}_{\et}(\overline{\mathcal{A}}^{r},\mathbb{Z}_p)$, by the comparison isomorphism recalled above we get maps
\[
\begin{split}
&P H^{k-1}_{\dR}(\mathcal{A}^r/K_{c,\mathfrak{p}}) \cong \text{DR}_{K_{c,\mathfrak{p}}}\bigl(P H^{k-1}_{\et}(\overline{\mathcal{A}}^{r},\mathbb{Q}_p)\bigr) \longrightarrow \text{DR}_{K_{c,\mathfrak{p}}}(V_f),\\
&\Sym^{2r}eH^{1}_{\dR}(A/K_{c,\mathfrak{p}}) \cong \text{DR}_{K_{c,\mathfrak{p}}}\bigl(\Sym^{2r}eH^{1}_{\et}(A/K_{c,\mathfrak{p}})\bigr).
\end{split}
\]
Then, we can view the $\omega_f \otimes \omega_A^{j+k/2-1}\eta_A^{k/2 - j -1}$ as elements of $\Fil^{k-1}\text{DR}_{K_{c,\mathfrak{p}}}\bigl(V_f\otimes \Sym^{2r}eH^1_{\et}(\overline{A},\mathbb{Q}_p)\bigr)$, which we will denote in the same way. 
Let $t\in B_{\dR}$ denotes Fontaine's $p$-adic analogue of $2\pi i$ associated with the compatible sequence $\{i_p(\zeta_{p^n})\}$.
Hence, the elements $\omega_f \otimes \omega_A^{j+k/2-1}\eta_A^{k/2 - j -1} \otimes t^{1-k}$ live in $\Fil^0\text{DR}_F\bigl(V_f\otimes \Sym^{2r}eH^1_{\et}(\overline{A},\mathbb{Q}_p)(k-1)\bigr)$, so we can view them in $\Fil^0\text{DR}_F\bigl(V_f(k/2)\otimes \chi\bigr)$.
Because of the functoriality of the logarithm and because the comparison isomorphism preserves the dualities, there is an equality
\[
\big\langle\log_{\mathfrak{p}}(z_{\mathfrak{a}}), \omega_f \otimes \omega_A^{j+k/2-1}\eta_A^{k/2 - j -1} \otimes t^{1-k}\big\rangle = \big\langle\log_{\mathfrak{p}}(\xi_{\phi_{\mathfrak{a}}}),\omega_f \otimes \omega_A^{j+k/2-1}\eta_A^{k/2 - j -1}\big\rangle,
\]
where $\xi_{\phi_{\mathfrak{a}}}$ is the image of $\Delta_{\mathfrak{a}}$ through the usual Abel--Jacobi map $\AJ_{p,K_{c}}$, $\omega_f \otimes \omega_A^{j+k/2-1}\eta_A^{k/2 - j -1} \otimes t^{1-k} \in \Fil^0\text{DR}_{K_{c,\mathfrak{p}}}\bigl(V_f(k/2)\otimes \chi\bigr)$, $\omega_f \otimes \omega_A^{j+k/2-1}\eta_A^{k/2 - j -1} \in \Fil^{k-1} \varepsilon H^{4r+1}_{\dR}(\mathcal{X}_r/K_{c,\mathfrak{p}})$,
$\log_{\mathfrak{p}}(z_{\mathfrak{a}})$ is viewed in $\text{DR}_{K_{c,\mathfrak{p}}}\bigl(V_f(k/2)\otimes \chi\bigr)/\Fil^0\text{DR}_{K_{c,\mathfrak{p}}}\bigl(V_f(k/2)\otimes \chi\bigr)$ and $\log_{\mathfrak{p}}(\xi_{\mathfrak{a}})$ is viewed, as in \cite{Brooks}, in $\varepsilon H^{4r+1}_{\dR}(\mathcal{X}_r/K_{c,\mathfrak{p}})/\Fil^{k-1} \varepsilon H^{4r+1}_{\dR}(\mathcal{X}_r/K_{c,\mathfrak{p}})$
\end{rmk}

Evaluating the $p$-adic $L$-function at the $p$-adic avatar of an anticyclotomic Hecke character of infinity type
$(k/2+j, -k/2-j)$ with $ -k/2<j<k/2$ and conductor $p^n\mathcal{O}_K$ with $n \geq 1$, we get $\theta^{-j-k/2}\hat{f}^{(p)} (x(c_0p^n)^{\sigma_{\mathfrak{a}}})$, with $x(c_0p^n)$ the CM point defined in \S \ref{CMpoints}. Now we want to relate this expression to the image of the Abel--Jacobi map of certain Heegner classes.

\begin{lem} \label{keylem}
Set $z^{(p)}_{\mathfrak{a}} := z_{\mathfrak{a}} - a_p p^{2j-1} z_{\mathfrak{a}\mathcal{O}_{c_0p^{n-1}}} + p^{4j+k-3} z_{\mathfrak{a}\mathcal{O}_{c_0p^{n-2}}}$. Then 
\[
\begin{split}
\theta^{-j-k/2}\hat{f}^{(p)} (x(c_0p^n)^{\sigma_{\mathfrak{a}}}) = \frac{(c_0p^nN(\mathfrak{a}))^{-j-k/2+1}}{\Omega_p^{2j}(j+k/2-1)!}\cdot \big\langle\log_{\mathfrak{p}}(z^{(p)}_{\mathfrak{a}}), \omega_f \otimes \omega_A^{j+k/2-1}\eta_A^{k/2 - j -1} \otimes t^{1-k}\big\rangle.
\end{split}
\]
\end{lem}

\begin{proof}
We have
\[
\begin{split}
\theta^{-j-k/2}\hat{f}^{(p)}\bigl(x(c_0p^n)^{\sigma_{\mathfrak{a}}}\bigr) &= \theta^{-j-k/2}f^{(p)} (A_{c_0p^n}, \iota_{c_0p^n}, \nu_{N^+,c_0p^n}, \hat{\omega}_{c_0p^n})^{\sigma_{\mathfrak{a}}}\\
&= \left(\frac{1}{\Omega_p}\right)^{2j} \theta^{-j-k/2}f^{(p)} \bigl(\mathfrak{a} \star (A_{c_0p^n}, \iota_{c_0p^n}, \nu_{N^+,c_0p^n}, \omega_{c_0p^n})\bigr),
\end{split}
\]
as the form $\theta^{-j-k/2}f^{(p)}$ has weight $-2j$ and, by definition, $\hat{\omega}_{c_0p^n} = (1/\Omega_p) \omega_{c_0p^n}$, with $\omega_{c_0p^n}$ induced by $\omega_A$. Let $\tilde{g}_{j+k/2-1}^{(p)}$ be the ${j+k/2-1}$-st component of the Coleman primitive of $\omega_{f^{(p)}}$. Applying 
\cite[Theorem 7.3]{Brooks} yields
\[
\begin{split}
\theta^{-j-k/2}\hat{f}^{(p)} (x(c_0p^n)^{\sigma_{\mathfrak{a}}}) &= \frac{1}{\Omega_p^{2j}(j+k/2-1)!} \cdot\tilde{g}_{j+k/2-1}^{(p)} \bigl(\mathfrak{a} \star (A_{c_0p^n}, \iota_{c_0p^n}, \nu_{N^+,c_0p^n}, \omega_{c_0p^n})\bigr).
\end{split}
\]
Writing $x_{\mathfrak{a}}$ for $\mathfrak{a} \star(A_{c_0p^n}, \iota_{c_0p^n}, \nu_{N^+,c_0p^n}, \omega_{c_0p^n}) = (A_{\mathfrak{a}}, \iota_{\mathfrak{a}}, \nu_{N^+,\mathfrak{a}}, \omega_{\mathfrak{a}})$, by  
\cite[Lemma 8.6]{Brooks} one has
\[
\begin{split}
\theta^{-j-k/2}\hat{f}^{(p)}\bigl(x(c_0p^n)^{\sigma_{\mathfrak{a}}}\bigr) &= \frac{1}{\Omega_p^{2j}(j+k/2-1)!} \ \Big[ \ \tilde{g}_{j+k/2-1} (x_{\mathfrak{a}}) \\ &- \frac{a_pp^{-j-k/2}}{p^{-2j}} \ \tilde{g}_{j+k/2-1} (x_{\mathfrak{a}\mathcal{O}_{c_0p^{n-1}}}) \\ &+ \frac{1}{p^{-2j +1}} \ \tilde{g}_{j+k/2-1} (x_{\mathfrak{a}\mathcal{O}_{c_0p^{n-2}}}) \Big].
\end{split}
\]
Then, applying
\cite[Proposition 7.4]{Brooks}, and keeping Remark \ref{log} in mind, we get
\[
\begin{split}
&\theta^{-j-k/2}\hat{f}^{(p)}\bigl(x(c_0p^n)^{\sigma_{\mathfrak{a}}}\bigr) =\\ 
&=\frac{1}{\Omega_p^{2j}(j+k/2-1)!} \cdot \Big[ \ \frac{1}{(c_0p^{n}N(\mathfrak{a}))^{j+k/2-1}} \big\langle\log_{\mathfrak{p}}(z_{\mathfrak{a}}), \omega_f \otimes \omega_A^{j+k/2-1}\eta_A^{k/2 - j -1}\otimes t^{1-k}\big\rangle\\ 
&-\frac{a_p p^{j-k/2}}{(c_0p^{n-1}N(\mathfrak{a}))^{j+k/2-1}} \cdot \big\langle\log_{\mathfrak{p}}(z_{\mathfrak{a}\mathcal{O}_{c_0p^{n-1}}}), \omega_f \otimes \omega_A^{j+k/2-1}\eta_A^{k/2 - j -1}\otimes t^{1-k}\big\rangle \\ 
&+\frac{p^{2j-1}}{(c_0p^{n-2}N(\mathfrak{a}))^{j+k/2-1}} \cdot \big\langle\log_{\mathfrak{p}}(z_{\mathfrak{a}\mathcal{O}_{c_0p^{n-2}}}), \omega_f \otimes \omega_A^{j+k/2-1}\eta_A^{k/2 - j -1}\otimes t^{1-k}\big\rangle \Big].
\end{split}
\]
Finally, if we set
\[
z^{(p)}_{\mathfrak{a}} := z_{\mathfrak{a}} - a_p p^{2j-1} z_{\mathfrak{a}\mathcal{O}_{c_0p^{n-1}}} + p^{4j+k-3} z_{\mathfrak{a}\mathcal{O}_{c_0p^{n-2}}},
\]
then 
we obtain
\[
\begin{split}
\theta^{-j-k/2}\hat{f}^{(p)}\bigl(x(c_0p^n)^{\sigma_{\mathfrak{a}}}\bigr) = \frac{(c_0p^nN(\mathfrak{a}))^{-j-k/2+1}}{\Omega_p^{2j}(j+k/2-1)!}\cdot\big\langle\log_{\mathfrak{p}}(z^{(p)}_{\mathfrak{a}}), \omega_f \otimes \omega_A^{j+k/2-1}\eta_A^{k/2 - j -1} \otimes t^{1-k}\big\rangle,
\end{split}
\]
as was to be shown. \end{proof}

\subsection{A $p$-adic Gross--Zagier formula}

In this section we finally state and prove the $p$-adic Gross--Zagier formula that we are interested in, which relates our $p$-adic $L$-function to the Bloch--Kato logarithm of generalized Heegner classes.

First, we define a cohomology class $z_{\chi} \in H^1(K,T \otimes \chi)$ associated with $f$ and $\chi$, which will be linked with the $p$-adic $L$-function by the $p$-adic Gross--Zagier type formula. Recall that $f\in S_k^{\text{new}}(\Gamma_0(N))$ is our fixed modular form and
$\chi: \Gal(K_{c_0p^{\infty}}/K) \rightarrow \mathcal{O}_F^{\times}$
is a locally algebraic anticyclotomic character of infinity type $(j,-j)$ with $-k/2 < j < k/2$ and conductor $c_0p^n\mathcal{O}_K$, where $c_0$ is prime to $pN^+$. Put
\begin{equation} \label{z.chi}
\begin{split}
z_{\chi} :&= \Cor_{K_{c_0p^n}/K} z_{c_0p^n,\chi}
= \sum_{\sigma \in \Gal(K_{c_0p^n}/K)} \sigma \cdot (id \otimes e_{\chi})(z_{c_0p^n}\otimes \chi_t)\\
&= \sum_{\sigma \in \Gal(K_{c_0p^n}/K)} (id \otimes e_{\chi})\bigl( \chi_t(\sigma) z_{c_0p^n}^{\sigma}\bigr)
= \sum_{\mathfrak{a} \in \Pic\mathcal{O}_{c_0p^n}} (id \otimes e_{\chi})\bigl( \chi_t(\sigma_{\mathfrak{a}}\bigr) z_{c_0p^n}^{\sigma_{\mathfrak{a}}})\\
&= \sum_{\mathfrak{a} \in \Pic\mathcal{O}_{c_0p^n}} \chi \varepsilon_{cyc}^{-r}(\sigma_{\mathfrak{a}})(id \otimes e_{\chi})(z_{\mathfrak{a}}),
\end{split}
\end{equation}
where the last equality holds by Proposition \ref{4.5}.

It is convenient to use, in the statement of the following theorem, the symbol $\doteq$ to indicate that the claimed equality holds up to an explicit non-zero multiplicative factor that is comparatively less important than the main terms. Recall that $p\mathcal O_K= \mathfrak{p} \overline{\mathfrak{p}}$ with $ \mathfrak{p}$ and $\overline{\mathfrak{p}}$ distinct maximal ideals of $\mathcal O_K$.

\begin{teor} \label{GZ}
Let $\psi$ be an anticyclotomic Hecke character of infinity type $(k/2, -k/2)$ and conductor $c_0\mathcal{O}_K$ with $(c_0,Np) = 1$. If $\hat{\phi} : \Gal(K_{c_0p^{\infty}}/K) \rightarrow \mathcal{O}_{\mathbb{C}_{p}}^{\times}$ is the $p$-adic avatar of an anticyclotomic Hecke character $\phi$ of infinity type
$(k/2+j, -k/2-j)$ with $ -k/2<j<k/2$ and conductor $p^n\mathcal{O}_K$, $n \geq 1$, then 
	\begin{displaymath}
	\frac{\LL_{\mathfrak{p},\psi}(f)(\hat{\phi}^{-1})}{\Omega_p^{*}}\doteq  
	\big\langle\log_\mathfrak{p}(z_{f,\chi}),\omega_{f}\otimes \omega_A^{k/2 + j -1}\eta_A^{k/2 -j -1} \otimes t^{1-k}\big\rangle,
	\end{displaymath}
where, as before, $\chi := \hat{\psi}^{-1}\hat{\phi}$ and $ \log_{\mathfrak{p}} := \log_{BK} \circ \loc_{\mathfrak{p}}$.
\end{teor}

\begin{proof}
By definition of $\LL_{f,\psi}$ (cf. Definition \ref{Lfct}), one has
	\[
	\LL_{f,\psi}(\hat{\phi}^{-1}) 
	= \sum_{\Pic\mathcal{O}_{c_0}} \psi(\mathfrak{a})N(\mathfrak{a})^{-k/2} \int_{\mathbb{Z}_p^{\times}} \psi_{\mathfrak{p}}\hat{\phi}^{-1} \big|[\mathfrak{a}]d\mu_{\hat{f}^{(p)}_{\mathfrak{a}}}.
	\]
	Here $\hat{\phi}^{-1} \big|[\mathfrak{a}] : \mathbb{Z}_p^{\times} \rightarrow \mathbb{C}_p^{\times}$ is given by $\bigl(\hat{\phi}^{-1}\big|[\mathfrak{a}]\bigr)(x):= \hat{\phi}^{-1}(x) \hat{\phi}^{-1}(\mathfrak{a})$, where $x$ is viewed as an element of $\hat{K}^{\times}$ via the chosen embedding $\mathbb{Z}_p^{\times} \cong \mathcal{O}_{K,\mathfrak{p}} \hookrightarrow K_{\mathfrak{p}}^{\times} \hookrightarrow \hat{K}^{\times}$. It follows that 
	\[
	\bigl(\hat{\phi}^{-1}\big|[\mathfrak{a}]\bigr)(x)= \phi^{-1}(x)x_{\mathfrak{p}}^{-k/2-j}x_{\overline{\mathfrak{p}}}^{k/2+j} \phi^{-1}(a) = \phi^{-1}_{\mathfrak{p}}(x)x^{-k/2-j} \phi^{-1}(a),
	\]
	as $a \in \hat{K}^{(c_0p)\times}$ satisfies $a\hat{\mathcal{O}}_K\cap K = \mathfrak{a}$. Therefore
	\[
	\begin{split}
	\LL_{f,\psi}(\hat{\phi}^{-1})
	&= \sum_{\Pic(\mathcal{O}_{c_0})} \psi(\mathfrak{a})\phi^{-1}(\mathfrak{a})N(\mathfrak{a})^{-k/2} \int_{\mathbb{Z}_p^{\times}} \psi_{\mathfrak{p}}\phi^{-1}_{\mathfrak{p}}(x)x^{-k/2-j} d\mu_{\hat{f}^{(p)}_{\mathfrak{a}}}
	\end{split}
	\] 
Then, in light of (\ref{x^mtheta_neg}), we can bring out the differential operator $\theta = t_{\mathfrak{a}} \frac{d}{dt_{\mathfrak{a}}}$ from the integral of $x^{-k/2-j}$ and, using the fact that $x(\mathfrak{a})$ is the canonical lifting of $x(\mathfrak{a})\otimes_{\W} \overline{\mathbb{F}}_p$, we get
\[
\begin{split}
\LL_{f,\psi}(\hat{\phi}^{-1}) = \sum_{\Pic\mathcal{O}_{c_0}} \psi(\mathfrak{a})\phi^{-1}(\mathfrak{a})N(\mathfrak{a})^{-k/2} \theta^{-k/2-j} [\psi_{\mathfrak{p}}\phi^{-1}_{\mathfrak{p}}]\hat{f}^{(p)}_{\mathfrak{a}}\bigl(t_{\mathfrak{a}}(x_{\mathfrak{a}})\bigr),
\end{split}
\]
where $t_{\mathfrak{a}}$ is the Serre--Tate coordinate around the reduction $x(\mathfrak{a})\otimes_{\W} \overline{\mathbb{F}}_p$.

Setting $\xi := \psi^{-1} \phi$, then we have
\[
\begin{split}
\LL_{f,\psi}(\hat{\phi}^{-1}) &=\sum_{\Pic(\mathcal{O}_{c_0})} \psi(\mathfrak{a})\phi^{-1}(\mathfrak{a})N(\mathfrak{a})^{-k/2}  [\psi_{\mathfrak{p}}\phi^{-1}_{\mathfrak{p}}]\theta^{-k/2-j}\hat{f}^{(p)}_{\mathfrak{a}}\bigl(t_{\mathfrak{a}}(x_{\mathfrak{a}})\bigr)\\
&= \sum_{\Pic(\mathcal{O}_{c_0})} \xi^{-1}(\mathfrak{a})N(\mathfrak{a})^{-k/2}  [\xi^{-1}_{\mathfrak{p}}]\theta^{-k/2-j}\hat{f}^{(p)}_{\mathfrak{a}}\bigl(t_{\mathfrak{a}}(x_{\mathfrak{a}})\bigr).
\end{split}
\]
Note that, since $\xi^{-1}$ is a character of conductor $c_0p^n\mathcal{O}_K$, $\xi_{\mathfrak{p}}^{-1}$ is a primitive Dirichlet character mod $p^n$ via the isomorphism $(\mathbb{Z}/p^n\mathbb{Z})^{\times} \cong(\mathcal{O}_{K,\mathfrak{p}}/\mathfrak{p}^n)^{\times} \cong \mathcal{O}_{K,\mathfrak{p}}^{\times}/1 + \mathfrak{p}^n$.
By Proposition \ref{3.3}, since $(N(\mathfrak{a})\sqrt{-D_K})^{k/2+j}(\theta^{-j-k/2}\hat{f}^{(p)})_{\mathfrak{a}} = \theta^{-j-k/2}\hat{f}^{(p)}_{\mathfrak{a}}$, we obtain
\[
\begin{split}
\LL_{f,\psi}(\hat{\phi}^{-1}) &=\sum_{\Pic(\mathcal{O}_{c_0})} \xi^{-1}(\mathfrak{a})N(\mathfrak{a})^{-k/2}
\bigl(N(\mathfrak{a})\sqrt{-D_K}\bigr)^{k/2+j} p^{-n}G(\xi_{\mathfrak{p}}^{-1})\\
&\cdot\sum_{u\in (\mathbb{Z}/p^n\mathbb{Z})^{\times}} \xi_{\mathfrak{p}}(u) \theta^{-k/2-j} \hat{f}^{(p)}\bigl(x_{\mathfrak{a}}\star \alpha(u/p^n)\bigr).\\
\end{split}
\]
 For positive exponents one obtains $\bigl(N(\mathfrak{a})\sqrt{-D_K}\bigr)^{-m}(\theta^{m}\hat{f}^{(p)})_{\mathfrak{a}} = \theta^{m}\hat{f}^{(p)}_{\mathfrak{a}}$ by an easy computation; because $\theta^{m} = \lim_{i \to \infty} \theta^{m+(p-1)p^i}$, one has the same formula for negative exponents.
Now, as in the proof of \cite[Theorem 4.9]{CH} and applying Lemma \ref{keylem}, we obtain
\[
\begin{split}
\frac{\LL_{f,\psi}(\hat{\phi}^{-1})}{\Omega_p^{-2j}} &=\frac{c_0^{-j-k/2+1}(\sqrt{-D_K})^{k/2+j} p^{n(-j-k/2)}G(\xi_{\mathfrak{p}}^{-1})\chi_{\mathfrak{p}}(p^n)}{(j+k/2-1)!}\\
&\cdot \sum_{\Pic(\mathcal{O}_{c_0p^n})} \chi \chi_{\cyc}^{1-k/2}(\sigma_{\mathfrak{a}})
\big\langle\log_{\mathfrak{p}}(z^{(p)}_{\mathfrak{a}}), \omega_f \otimes \omega_A^{j+k/2-1}\eta_A^{k/2 - j -1}\otimes t^{1-k}\big\rangle.
\end{split}
\]
Finally, by the argument in \cite[Theorem 4.9]{CH} and by (\ref{z.chi}), we get
\[
\frac{\LL_{f,\psi}(\hat{\phi}^{-1})}{\Omega_p^{-2j}}= C'\cdot
\big\langle\log_{\mathfrak{p}}(z_{\chi}), \omega_f \otimes \omega_A^{j+k/2-1}\eta_A^{k/2 - j -1} \otimes t^{1-k}\big\rangle,
\]
with
\[
C':=\frac{c_0^{-j-k/2+1}(\sqrt{-D_K})^{k/2+j} p^{n(-j-k/2)}G(\xi_{\mathfrak{p}}^{-1})\chi_{\mathfrak{p}}(p^n)}{(j+k/2-1)!},
\]
and the theorem is proved.
\end{proof}

\section{Reciprocity law and Selmer groups}

In this last section we want to extend to our setting the reciprocity law of \cite[Theorem 5.7]{CH}, relaxing the Heegner hypothesis and making use of generalized Heegner cycles on generalized Kuga--Sato varieties over Shimura curves and our $p$-adic $L$-function. This result will be important to prove (under certain assumptions) the vanishing of the Selmer group associated with the twisted representation $V_{f,\chi} := V_f(k/2) \otimes \chi$.

\subsection{The algebraic anticyclotomic $p$-adic $L$-function $\mathcal{L}_{\mathfrak{p},\psi}$}

We now construct an algebraic $p$-adic $L$-function as a sort of image of some Iwasawa cohomology class, coming from generalized Heegner classes, under a big logarithm map.
Assume for this section that our modular form $f\in S_k^{new}(\Gamma_0(N))$ is $p$-ordinary, i.e., that the $p$-th Fourier coefficient $a_p$ is a unit of $\mathcal{O}_F$.

\subsubsection{Perrin-Riou's big logarithm}

Let $G$ be a commutative compact $p$-adic Lie group and $L$ a complete discretely valued extension of $\mathbb{Q}_p$.  
Recall that a \emph{$p$-adic Lie group} is a group $G$ endowed with a structure of a manifold over $\mathbb{Q}_p$ such that the group operation is locally analytic. For the definitions of a manifold over $\mathbb{Q}_p$ and of a locally analytic map, the reader is referred to \cite{Schneider}.

Consider the noetherian topological
$O_L$-algebra $\mathcal{O}_L \llbracket G \rrbracket$; if $L/\mathbb{Q}_p$ is a finite extension then it is compact. Put $\Lambda_{L}(G) := L \otimes_{\mathcal{O}_L} \mathcal{O}_L \llbracket G \rrbracket$, which is also noetherian; it is isomorphic to
the continuous dual of the space $C(G, L)$ of continuous $L$-valued functions on $G$ (cf. \cite[\S 2.2]{LZ14}). 
Now let
$\mathcal{H}_{L}(G)$ denote the space of $L$-valued locally analytic distributions on $G$
i.e., the continuous
dual of the space $C^{la}(G, L)$ of $L$-valued locally analytic functions on $G$.
There is an injective algebra homomorphism \[\Lambda_L(G) \longmono \mathcal{H}_L(G)\]
(see \cite[Proposition 2.2.7]{Eme}), dual to the dense inclusion $C^{la}(G, L) \hookrightarrow C(G, L)$. We
endow $\mathcal{H}_L(G)$ with its natural topology as an inverse limit of Banach spaces, with
respect to which the map $\Lambda_L(G) \hookrightarrow \mathcal{H}_L(G)$ is continuous.

If $L$ is a finite unramified extension of $\mathbb{Q}_p$ and $G$ is the Galois group of a $p$-adic Lie extension $L_{\infty}=\cup_n L_n$ with $L_n/L$ finite and Galois, then define the Iwasawa cohomology group
\begin{displaymath}
H^1_{\Iw}(L_{\infty},V):=\Bigl(\varprojlim_n H^1(L_n,T)\Bigr)\otimes_{\mathbb{Z}_p} \mathbb{Q}_p,
\end{displaymath}
where $V$ is a $p$-adic $G_L$-representation and $T$ is a Galois stable lattice. The definition is independent from the choice of $T$. 
Now let $\hat{F}^{\text{ur}}$ denote the composite of $\hat{\mathbb{Q}}_p^{\text{ur}}$ with a finite extension $F$ of $\mathbb{Q}_p$.

Suppose that $V$ is a crystalline $F$-representation of $G_L$ with non negative Hodge--Tate weights (for us the $p$-adic cyclotomic character has Hodge--Tate weight $+1$) and that $V$ has no quotient isomorphic to the trivial representation. Let $\mathfrak{F}$ be a relative height one Lubin--Tate formal group over $\mathcal{O}_L/\mathbb{Z}_p$
and let $\Gamma := \Gal(L(\mathfrak{F}_{p^{\infty}})/L) \cong \mathbb{Z}_p^{\times}$.
Let $B_{\crys}$ be the Fontaine's crystalline period ring and put $\mathbf{D}_{\crys,L}(V):= (V \otimes_{\mathbb{Q}_p} B_{\crys})^{G_L}$. Assume that $V^{G_{L(\mathfrak{F}_{p^\infty})}} = 0$.

\begin{teor} \label{5.1}
	There exists a $\mathbb{Z}_p\llbracket \Gamma \rrbracket$-linear map
	\begin{displaymath}
	\mathcal{L}_V : H^1_{\Iw}(L(\mathfrak{F}_{p^{\infty}}),V) \longrightarrow \mathcal{H}_{\hat{F}^{ur}}(\Gamma) \otimes_L \mathbf{D}_{\crys,L}(V)
	\end{displaymath}
	such that for any $\mathbf{z} \in H^1_{\Iw}(L(\mathfrak{F}_{p^{\infty}}),V)$ and any locally algebraic character $\chi : \Gamma \rightarrow \overline{\mathbb{Q}}_p^{\times}$ of Hodge-Tate weight $j$ and conductor $p^n$ there is an equality
	\begin{displaymath}
	\mathcal{L}_V(\mathbf{z})(\chi)= \varepsilon(\chi^{-1}) \cdot  \frac{\Phi^n P(\chi^{-1},\Phi)}{P(\chi,p^{-1}\Phi^{-1})} \cdot 
	\begin{cases} 
	\frac{(-1)^{-j-1}}{(-j-1)!} \cdot \log_{L,V \otimes \chi^{-1}} (\mathbf{z^{\chi^{-1}}})\otimes t^{-j} & \text{if}\ j<0, \\[2mm]
	j! \cdot \exp^*_{L,(V \otimes \chi^{-1})^*(1)} (\mathbf{z^{\chi^{-1}}})\otimes t^{-j} & \text{if}\ j \geq 0,
	\end{cases}
	\end{displaymath}
	where
	\begin{itemize}
		\item $\varepsilon(\chi^{-1})$ and $P(\chi^{\pm1},-)$ are the $\varepsilon$-factor and the $L$-factor (see \cite[p. 8]{LZ14});
		\item $\Phi$ denotes the crystalline Frobenius operator on $\overline{\mathbb{Q}}_p \otimes_L \mathbf{D}_{\crys,L}(V)$ acting trivially on first factor;
		\item $\mathbf{z^{\chi^{-1}}} \in H^1(L,V \otimes \chi^{-1})$ is the specialization of $\mathbf{z}$ at $\chi^{-1}$.
	\end{itemize}
\end{teor}

\begin{proof} This is \cite[Theorem 5.1]{CH}. \end{proof}

In the statement of this result, $\exp^*_{L,(V \otimes \chi^{-1})^*(1)}$ is the dual exponential map
\[
\exp^*_{L,(V \otimes \chi^{-1})^*(1)}: H^1(L,V \otimes \chi^{-1})) \rightarrow \Fil^0DR_{L}(V \otimes \chi^{-1})
\]
of the Bloch--Kato exponential map of $V(\chi^{-1})^*(1)$.

We will apply Theorem \ref{5.1} to some representation $\mathscr{F}^+V$ attached to a twist of $V_f(k/2)$ to obtain a map $\mathcal{L}_{\mathscr{F}^+V}$.
Let $\psi$ be an anticyclotomic Hecke character of infinity type $(k/2,-k/2)$ and conductor $c\mathcal{O}_K$ with $p \nmid c$ and let $\hat{\psi}: \Gal(K_{cp^{\infty}}/K) \rightarrow \mathbb{C}_p^{\times}$ be its $p$-adic avatar. Let $F$ be a finite extension of $\mathbb{Q}_p$ containing the Fourier coefficients of $f$ and the image of $\hat{\psi}$, so $\hat{\psi} : \Gal(K_{cp^{\infty}}/K) \rightarrow \mathcal{O}_F^{\times}$.
Since $p \nmid N$, if $V_f$ is the $F$-linear Galois representation of $G_{\mathbb{Q}}$ associated with $f$, then $V_f\mid_{G_{\mathbb{Q}_p}}$ is crystalline.
Because $f$ is $p$-ordinary, there is an exact sequence of $G_{\mathbb{Q}_p}$-modules
\begin{displaymath}
0 \longrightarrow \mathscr{F}^+V_f \longrightarrow V_f \longrightarrow \mathscr{F}^-V_f \longrightarrow 0
\end{displaymath}
with $\mathscr{F}^{\pm}V_f \cong F$ and $\mathscr{F}^+V_f$ unramified (see \cite[Theorem 2.1.4]{Wiles} or \cite[\S 12.5.3]{NekSC}).
Recall that $T \subseteq V_f(k/2)$ is a Galois-stable lattice; set
\begin{displaymath}
\begin{split}
\mathscr{F}^+T &:= \mathscr{F}^+V_f(k/2) \cap T;\\
V &:= V_f(k/2) \otimes \hat{\psi}_{\mathfrak{p}}^{-1},\quad \text{where}\ \hat{\psi}_{\mathfrak{p}} := \hat{\psi} \mid_{K_{\mathfrak{p}}};\\
\mathscr{F}^{\pm}V &:= \mathscr{F}^{\pm}V_f(k/2) \otimes \hat{\psi}_{\mathfrak{p}}^{-1}.
\end{split}
\end{displaymath}
Consider the dual representation $V^* := \Hom_F(V,F)$ of $V$ and, with notation as above, define $\mathscr{F}^{\pm}V^* := \Hom_F(\mathscr{F}^{\mp}V, F)$.
Let $L_{\infty}/L$ be the $\mathfrak{p}$-adic completion of $K_{cp^{\infty}}/K_c$.
The big logarithm $\mathcal{L}_{\mathscr{F}^+V}$, obtained applying Theorem \ref{5.1} to $\mathscr{F}^+V$ as a representation of $G_L$, is a map
\begin{displaymath}
\mathcal{L}_{\mathscr{F}^+V} : H^1_{\Iw}\bigl(L(\mathfrak{F}_{p^{\infty}}),\mathscr{F}^+V\bigr) \longrightarrow \mathcal{H}_{\hat{F}^{\text{ur}}}\bigl(\Gal(L(\mathfrak{F}_{p^{\infty}})/L)\bigr) \otimes_L \mathbf{D}_{\crys,L}(\mathscr{F}^+V).
\end{displaymath}
Since $L_{\infty} \subseteq L(\mathfrak{F}_{p^{\infty}})$, we can restrict $\mathcal{L}_{\mathscr{F}^+V}$ to the Galois group $\Gamma:= \Gal(L_{\infty}/L) \cong \Gal(K_{cp^{\infty}}/K_c)$ to obtain a map 
\begin{displaymath}
H^1_{\Iw}(L_{\infty},\mathscr{F}^+V) \longrightarrow \mathcal{H}_{\hat{F}^{\text{ur}}}(\Gamma) \otimes_L \mathbf{D}_{\crys,L}(\mathscr{F}^+V).
\end{displaymath}
Recall the element $\omega_f \in DR_L(V_f)$ attached to $f$ as in \ref{keyform}.
Let $t\in B_{\dR}$ denotes again Fontaine's $p$-adic analogue of $2\pi i$ associated with the compatible sequence $\{i_p(\zeta_{p^n})\}$.
Define the class
\begin{displaymath}
\omega_{f,\psi} := \omega_f \otimes t^{-k} \otimes \omega_{\psi} \in \mathbf{D}_{\crys,L}(V^*(1)),
\end{displaymath}
where $\omega_{\psi}\in\mathbf{D}_{\crys,L}(\hat{\psi}_{\mathfrak{p}}(-k/2))$ is as in \cite[\S 5.3]{CH}.
Denote again by $\omega_{f,\psi}$ its image under the projection
$\mathbf{D}_{\crys,L}(V^*(1)) \twoheadrightarrow \mathbf{D}_{\crys,L}(\mathscr{F}^-V^*(1))$. 

There is a pairing
\begin{displaymath}
\braket{-,-} : \mathcal{H}_{\hat{F}^{\text{ur}}}(\Gamma) \otimes_L \mathbf{D}_{\crys, L}(\mathscr{F}^+V) \times \mathbf{D}_{\crys, L}(\mathscr{F}^-V^*(1)) \longrightarrow \mathcal{H}_{\hat{F}^{\text{ur}}}(\Gamma).
\end{displaymath}
Recall that $\mathbf{D}_{\crys, L}(\mathscr{F}^+V) = (B_{\crys} \otimes_{\mathbb{Q}_p} \mathscr{F^+}V)^{G_{L}}$ and $\mathbf{D}_{\crys, L}(\mathscr{F}^-V^*(1)) = (B_{\crys} \otimes_{\mathbb{Q}_p} \mathscr{F^-}V^*(1))^{G_{L}}$. 
Finally, the composition of $\mathcal{L}_{\mathscr{F}^+V}$ with the map
\begin{displaymath}
\braket{-, \omega_{f,\psi}}: \mathcal{H}_{\hat{F}^{\text{ur}}}(\Gamma) \otimes_L \mathbf{D}_{\crys, L}(\mathscr{F}^+V) \longrightarrow \mathcal{H}_{\hat{F}^{\text{ur}}}(\Gamma)
\end{displaymath}
has image contained in the Iwasawa algebra $\Lambda_{\hat{F}^{\text{ur}}}(\Gamma) := \mathcal{O}_{\hat{F}^{\text{ur}}}\llbracket \tilde{\Gamma} \rrbracket \otimes \hat{F}^{\text{ur}}$. For details, see \cite[Lemma 5.5]{CH}.

\subsubsection{Iwasawa classes associated with generalized Heegner classes}

Consider the Iwasawa cohomology group
\begin{displaymath}
H^1_{\Iw}(K_{cp^{\infty}},T):=\biggl(\varprojlim_n H^1(\Gal(K'/K_{cp^n}),T)\biggr)\otimes_{\mathbb{Z}_p} \mathbb{Q}_p,
\end{displaymath}
where $K'$ is the maximal extension of $K$ unramified outside the primes above $pNc$ (the representation $T$ is unramified outside the prime above $pN$). Let $\alpha$ denote the root of the Hecke polynomial $x^2 -a_px+p^{k-1}$ that is a $p$-adic unit. For each fractional ideal $\mathfrak{a}$ of $\mathcal{O}_c$ prime to $cNpD_K$, recall the cohomology class $z_{\mathfrak{a}}$ introduced in \S \ref{GHclasses} and define the class
\[
z_{\mathfrak{a},\alpha} := \begin{cases}
z_{\mathfrak{a}} - \frac{p^{k-2}}{\alpha} \cdot z_{\mathfrak{a}\mathcal{O}_{c/p}} \ &\text{if}\ p\,|\,c\\[3mm]
\frac{1}{\# \mathcal{O}_c^{\times}} \left( 1 - \frac{p^{k/2 -1}}{\alpha} \sigma_{\mathfrak{p}}\right)  \left( 1 - \frac{p^{k/2 -1}}{\alpha} \sigma_{\overline{\mathfrak{p}}}\right) \cdot z_{\mathfrak{a}}\ &\text{if}\ p \nmid c,
\end{cases}
\]
which lives in $H^1\bigl(K_c,T \otimes S(E)\bigr)$. Here, $\sigma_{\mathfrak{p}}, \sigma_{\overline{\mathfrak{p}}} \in \Gal(K_c/K)$ are the Frobenius elements of $\mathfrak{p}$ and $\overline{\mathfrak{p}}$. By Proposition \ref{4.4}, one knows that
\[
\Cor_{K_{cp}/K_c}(z_{cp, \alpha}) = \alpha \cdot z_{c,\alpha}.
\]
Now consider the projection $e_{\chi}$ for $\chi = \boldsymbol{1}$ and write $z_{\mathfrak{a}, \alpha,\boldsymbol{1}}$ for the image of $z_{\mathfrak{a},\alpha}$ under $\id \otimes e_{\boldsymbol{1}} : H^1\bigl(K_c,T \otimes S(W)\bigr) \rightarrow H^1(K_c,T)$.
Thus, it makes sense to consider the element 
\[
\boldsymbol{z}_{c, \alpha} := \varprojlim_n \alpha^{-n} z_{cp^n,\alpha, \boldsymbol{1}}
\]
in the Iwasawa cohomology group $H^1_{\Iw}(K_{cp^{\infty}},T)$.

There is an isomorphism
\begin{displaymath}
H^1_{\Iw}(K_{cp^{\infty}}, T) \cong H^1\bigl(K_c, T \otimes \mathcal{O}_F\llbracket \Gamma \rrbracket\bigr),
\end{displaymath}
where $\Gamma := \Gal(K_{cp^{\infty}}/K_c)$ (see the proof of \cite[Proposition 2.4.2]{LZ16}). 
Put $\tilde{\Gamma}_{c} := \Gal(K_{cp^{\infty}}/K)$ and consider the map
\begin{displaymath}
H^1(K_{cp^{\infty}}, T) \cong H^1\bigl(K_c, T \otimes \mathcal{O}_F\llbracket \Gamma \rrbracket\bigr) \longrightarrow H^1\bigl(K_c, T \otimes \mathcal{O}_F\llbracket \tilde{\Gamma}_{c} \rrbracket\bigr);
\end{displaymath}
we can view the classes $\boldsymbol{z}_{c,\alpha}$ as elements of $H^1\bigl(K_c, T \otimes \mathcal{O}_F\llbracket \tilde{\Gamma}_{c} \rrbracket\bigr)$. Then set
	\begin{equation} \label{zf}
	\boldsymbol{z}_f := \Cor_{K_c/K}(\boldsymbol{z}_{c,\alpha}) \in H^1\bigl(K, T \otimes \mathcal{O}_F\llbracket \tilde{\Gamma}_{c} \rrbracket\bigr),
	\end{equation}
where the subscript $f$ is meant to remind that the class above, like the others already defined, depends on it.

For any character $\chi: \tilde{\Gamma}_{c} \rightarrow \mathcal{O}_{\mathbb{C}_p}^{\times}$, we can consider the twist $\boldsymbol{z}_f ^{\chi} \in H^1(K,T \otimes \chi)$ of $\boldsymbol{z}_f$ through the $\chi$-specialization map
\[
H^1\bigl(K,T \otimes \mathcal{O}_F\llbracket \tilde{\Gamma}_{c} \rrbracket\bigr) \longrightarrow H^1\Bigl(K,T \otimes \mathcal{O}_F\llbracket \tilde{\Gamma}_{c} \rrbracket \otimes_{\mathcal{O}_F\llbracket \tilde{\Gamma}_{c} \rrbracket} \chi\Bigr) = H^1(K,T \otimes \chi),
\]
where $\chi$ is extended to $\chi: \mathcal{O}_F\llbracket \tilde{\Gamma}_{c} \rrbracket \rightarrow \mathcal{O}_{F}$ in the obvious way, possibly enlarging $F$ by adding the image of $\chi$.
Suppose that $\chi$ is non-trivial, of finite order and with conductor $cp^n$; then 
\begin{equation} \label{zchi}
\boldsymbol{z}_f^{\chi} = \alpha^{-n} z_{\chi},
\end{equation}
where $z_{\chi} \in H^1(K,T \otimes \chi)$ is as in (\ref{z.chi}). See \cite[Lemma 5.4]{CH} for details.

\subsubsection{The algebraic anticyclotomic $p$-adic $L$-function} \label{algL}

We want to apply the logarithm map $\mathcal{L}_{\mathscr{F}^+V}$ to the localization at $\mathfrak{p}$ of the classes $\boldsymbol{z}_{c,\alpha} \otimes \hat{\psi}^{-1}$, so we need to check that these classes actually lie in $H^1_{\Iw}(L_{\infty},\mathscr{F}^+V) = H^1_{\Iw}(K_{cp^{\infty},\mathfrak{p}},\mathscr{F}^+V) \hookrightarrow H^1_{\Iw}(K_{cp^{\infty},\mathfrak{p}},V)$.

Similarly to what we said in \S \ref{keyform}, by \cite{Nek00}, one knows that $z_{\mathfrak{a}}$ lies in the Bloch--Kato Selmer group $\Sel(K_{c}, T \otimes S(E))$, indeed $z_{\mathfrak{a}}$ is the image, through a morphism of $G_{K_{c}}$-modules, of a cohomology class in $\Sel(K_c,\varepsilon H^{4r+1}_{\et}(\overline{\mathcal{X}}_r,\mathbb{Z}_p)(k-1))$, which is the Abel--Jacobi image of the generalized Heegner cycle $\Delta_{\mathfrak{a}}$. Recall that the Bloch--Kato Selmer group $\Sel(F,M)$ of a $G_F$-representation $M$, with $F$ number field, is the subspace of elements $x$ of $H^1(G_F,M)$ such that for all finite place $v$ of $F$, the localization $\loc_v(x) \in H^1_f(F_v,M)$. See \S \ref{Sel} for more precise definitions.
Thus, $z_{cp^n,\alpha,\mathbf{1}} \in \Sel(K_{cp^n},T)$ as well, so $\loc_{\mathfrak{p}}(z_{cp^n,\alpha,\mathbf{1}}) \in H^1_f(K_{cp^n,\mathfrak{p}},T)$.
But $H^1_f(K_{cp^n,\mathfrak{p}},T)$ is identified with the image of the map $H^1(K_{cp^n,\mathfrak{p}},\mathscr{F}^+T) \rightarrow H^1(K_{cp^n,\mathfrak{p}},T)$ (cf. \cite[\S 5.5]{CH}), so we can view $\loc_{\mathfrak{p}}(z_{cp^n,\alpha,\mathbf{1}}) \in H^1(K_{cp^n,\mathfrak{p}},\mathscr{F}^+T)$.
Since $\boldsymbol{z}_{c,\alpha}$ is the inverse limit of the classes $\alpha^{-n} z_{cp^n,\alpha,\mathbf{1}}$, one has
\begin{displaymath}
\loc_{\mathfrak{p}}(\boldsymbol{z}_{c,\alpha}) \in 
H^1_{\Iw}(L_{\infty},\mathscr{F}^+T).
\end{displaymath}
We conclude that
\begin{displaymath}
\loc_{\mathfrak{p}}(\boldsymbol{z}_{c,\alpha}\otimes\hat{\psi}^{-1}) \in H^1_{\Iw}(L_{\infty},\mathscr{F}^+V).
\end{displaymath}
Now, using notation similar to that in \cite{CH}, we can define
\begin{displaymath}
\begin{split}
\mathcal{L}^*(\boldsymbol{z}_f \otimes \hat{\psi}^{-1}) &:= \Cor_{K_c/K}(\mathcal{L}_{\mathscr{F}^+V}(\loc_{\mathfrak{p}}(\boldsymbol{z}_{c,\alpha}\otimes\hat{\psi}^{-1})))\\
&= \sum_{\sigma \in \tilde{\Gamma}_c/\Gamma} \mathcal{L}_{\mathscr{F}^+V}(\loc_{\mathfrak{p}}(\boldsymbol{z}^{\sigma}_{c,\alpha}\otimes\hat{\psi}^{-1}))\hat{\psi}(\sigma^{-1})
\in \mathbf{D}_{\crys,L}(\mathscr{F}^+V) \otimes \Lambda_{\hat{F}^{\text{ur}}}(\tilde{\Gamma}_c),
\end{split}
\end{displaymath}
	where $\Lambda_{\hat{F}^{\text{ur}}}(\tilde{\Gamma}_c) = \mathcal{O}_{\hat{F}^{\text{ur}}}\llbracket \tilde{\Gamma}_c \rrbracket \otimes \hat{F}^{\text{ur}}$.
	Finally, consider the restriction
	\begin{equation}
	\mathcal{L}_{\psi}(\boldsymbol{z}_f) := \Res_{K_{p^{\infty}}}(\mathcal{L}^*(\boldsymbol{z}_f \otimes \hat{\psi}^{-1})) \in \mathbf{D}_{\crys, L}(\mathscr{F}^+V) \otimes \Lambda_{\hat{F}^{\text{ur}}}(\tilde{\Gamma}),
	\end{equation}
	where $ \Res_{K_{p^{\infty}}} : \tilde{\Gamma}_c = \Gal(K_{cp^{\infty}}/K) \rightarrow \tilde{\Gamma} = \Gal(K_{p^{\infty}}/K)$ is the restriction map.

\subsection{Reciprocity law}

We start by giving a sketch of the proof of the following theorem, which is analogous to that of \cite[Theorem 5.7]{CH}.

	\begin{teor} \label{5.7}
		Let $\psi : K^{\times} / \mathbb{A}^{\times}_K \rightarrow \mathbb{C}^{\times}$ be an anticyclotomic Hecke character of infinity type $(k/2, -k/2)$ and conductor $c\mathcal{O}_K$ with $ p \nmid c$ and suppose that $f$ is $p$-ordinary.
		Then
\begin{displaymath}
\big\langle\mathcal{L}_{\psi}(\boldsymbol{z}_f), \omega_f \otimes t^{-k}\big\rangle = -c^{k/2-1}(\sqrt{-D_K})^{-k/2} \cdot \LL_{f,\psi} \cdot \sigma_{-1,\mathfrak{p}} \in \Lambda_{\hat{F}^{\emph{ur}}}(\tilde{\Gamma}),
\end{displaymath}
		where $\sigma_{-1,\mathfrak{p}}:= \emph{rec}_{\mathfrak{p}(-1)|_{K{p^{\infty}}}} \in \tilde{\Gamma}$ is an element of order $2$.
	\end{teor}
	
\begin{proof}[Sketch of proof]
For any $ n > 1$, let $ \hat{\phi} : \Gal(K_{p^{\infty}/K}) \rightarrow \mathbb{C}_{p}^{\times}$ be the $p$-adic avatar of a Hecke character $\phi$ of infinity type
$(k/2, -k/2)$ and conductor $p^n$. Moreover, define the finite order character $\chi := \hat{\psi}^{-1}\hat{\phi}$.
Recall that, by (\ref{zchi}), we have
\begin{displaymath}
\boldsymbol{z}^{\chi} = \alpha^{-n} \cdot z_{\chi}. 
\end{displaymath}
Now, since our $\omega_A$ and $\eta_A$ are chosen, in \S \ref{6.3}, to be compatible with the ones of \cite{CH}, so that $\omega_A\eta_A = t$ (see \cite[\S 5.3]{CH}), applying Theorem \ref{GZ} with $j = 0$ and \eqref{zchi} yields the expression $\big\langle\log_\mathfrak{p}(\boldsymbol{z}_{f}^{\chi}) \otimes t^{k/2},\omega_{f}\otimes t^{-k}\big\rangle$ from $\LL_{f,\psi}(\hat{\phi}^{-1})$. Now, performing the same computation as in the proof of \cite[Theorem 5.7]{CH} and applying Theorem \ref{5.1} to the expression $\big\langle\mathcal{L}_{\psi}(\boldsymbol{z}_f), \omega_f \otimes t^{-k}\big\rangle$, we obtain the formula of the statement evaluated at $\hat{\phi}^{-1}$ for any $p$-adic avatar $\hat{\phi}$ as above. By an argument that is formally identical to the one at the end of the proof of \cite[Theorem 5.7]{CH}, one gets the desired equality.
\end{proof}

Now we state the reciprocity law that is the counterpart of \cite[Corollary 5.8]{CH}.

\begin{teor}\label{rec}
Let $\chi : \Gal(K_{p^{\infty}}/K) \rightarrow \mathcal{O}_F^{\times}$ be a locally algebraic $p$-adic Galois character of infinity type $(j, -j)$ with $j \geq k/2$ and conductor $cp^n\mathcal{O}_K$ with $ p \nmid c$ and suppose that $f$ is $p$-ordinary. Then
\[
\braket{\exp^*(\loc_{\mathfrak{p}}(\boldsymbol{z}^{\chi^{-1}})), \omega_f \otimes \omega_A^{{-k/2-j}} \eta_A^{-k/2 - j}}^2 = D(f, \psi, \chi\psi^{-1}, K) \cdot L(f, \chi, k/2),
\]
where $D(f, \psi, \chi\psi^{-1}, K)$ is non-zero a constant depending on $f, \chi, K$ and $\psi$, which is a Hecke character of infinity type $(k/2, -k/2)$ and conductor $c$.
\end{teor}

\begin{proof}[Sketch of proof]
Let $ \hat{\psi} : \Gal(K_{p^{\infty}/K}) \rightarrow \mathbb{C}_{p}^{\times}$ be the $p$-adic avatar of a Hecke character $\psi$ of infinity type $(k/2, -k/2)$ and conductor $c$, so that $\hat{\phi} := \chi \hat{\psi}^{-1}$ is a locally algebraic character of infinity type $(j-k/2, -j + k/2)$ and conductor $p^n$.
The proof proceeds by using Theorem \ref{5.1} to extract the expression $\big\langle\exp^*(\loc_{\mathfrak{p}}(\boldsymbol{z}^{\chi^{-1}})), \omega_f \otimes \omega_A^{{-k/2-j}} \eta_A^{-k/2 + j }\big\rangle^2$ from $\braket{\mathcal{L}_{\mathfrak{p},\psi}(\boldsymbol{z}), \omega_f \otimes t^{-k}}(\hat{\phi})$. Now one can take the square and apply Theorem \ref{5.7} to recover the square of the $p$-adic $L$-function $\LL_{f,\psi}(\hat{\phi})$, and then use the interpolation formula of Theorem \ref{interpolation-thm} to obtain the statement. The constant $D(f,\psi, \chi\psi^{-1},K)$ turns out to be
\[
D(f,\psi, \chi\psi^{-1},K) =\; 
\alpha^{-2n} \epsilon(0,\phi_{\mathfrak{p}}^{-1}\psi_{\mathfrak{p}}^{-1})^{-2} p^{nk} (j-k/2+1)!^{-2}
\cdot\Omega^{-4j} c^{k-2} \sqrt{-D_K}^{-k}C(f,\psi, \chi\psi^{-1},K),
\]
where $\Omega \in \W^{\times}$ is the constant denoted with $\Omega_p$ in the proof of \cite[Corollary 5.8]{CH}.
For the details of the computation, see the proof of \cite[Corollary 5.8]{CH}.
\end{proof}

\subsection{The anticyclotomic Euler system method}

In this section we apply the Kolyvagin-type method developed in \cite[Section 7]{CH} to our system of Heegner classes, in order to deduce results on the Selmer group of the representation $V_{f,\chi} := V_f(k/2)\mid_{G_K} \otimes \chi$, with $\chi : \Gal(K_{c_0p^{\infty}}/K) \rightarrow \mathcal{O}_F^{\times}$ a locally algebraic $p$-adic Galois character of infinity type $(j, -j)$ and conductor $c_0p^s\mathcal{O}_K$ with $c=c_0p^s$ and $ (pN,c_0)=1$. 

First of all, we introduce the objects and the properties of the Kolyvagin method employed in \cite{CH}. Then, we will apply it to our system of generalized Heegner cycles and, finally, we will deduce results on Selmer groups. As will be clear, we follow \cite{CH} closely.

\subsubsection{Anticyclotomic Euler systems}

Let $\mathcal{G}_n := \Gal(K_n/K)$ and let $H^{1}(K_n,-)$ denote the cohomology group with respect to $\Gal(K^{\Sigma_n}/K_n)$, where $\Sigma_n$ is the finite set containig the prime factors of $pNc_0n$ and $K^{\Sigma_n}$ is the maximal extension of $K$ unramified outside the primes above $\Sigma_n$.

By \cite[Proposition 3.1]{Nek92}, there is a $G_{\mathbb{Q}}$-equivariant $\mathcal{O}_F$-linear perfect pairing
\begin{displaymath}
\braket{-,-} : T \times T \longrightarrow \mathcal{O}_F(1)
\end{displaymath}
that induces for each local field $L$ the local Tate pairing
\begin{equation}
\braket{-,-}_L : H^1(L,T) \times H^1(L,T) \longrightarrow \mathcal{O}_F.
\label{Tate-pairing}
\end{equation}
Here $T$ is the $G_{\mathbb{Q}}$-stable $\mathcal{O}_F$-lattice inside $V_f(k/2)$ that was fixed before. Let $\varpi$ be a uniformizer of $\mathcal{O}_F$ and let $ \mathbb{F} := \mathcal{O}_F/(\varpi)$ be the residue field. For every integer $M\geq1$, set $T_M:= T / \varpi^MT$.

For us, $\ell$ will always denote a prime inert in $K$ and $\lambda$ will be the unique prime of $K$ above $\ell$; denote by $\frob_{\ell}$ the Frobenius element of $\lambda$ in $G_K$. Let $H^1_f(K_{\lambda},-)$ be the finite part of $H^1(K_{\lambda},-)$, where $K_{\lambda}$ is the completion of $K$ at $\lambda$, and denote by $\loc_{\ell} : H^1(K,-) \rightarrow H^1(K_{\lambda},-)$ the localization map at $\ell$.

Let $\mathcal{S}$ be the set of square-free products of primes $\ell$ inert in $K$ with $\ell \nmid 2pNc_0$. Let $\tau$ denote complex conjugation.

\begin{defi}
	An \textbf{anticyclotomic Euler system} for $T$ and $\chi$ is a collection $\mathbf{c} = \left\lbrace c_n\right\rbrace _{n \in \mathcal{S}}$ of classes $c_n \in H^1(K_{nc}, T \otimes \chi^{-1})$ such that for any $ n = m \ell \in \mathcal{S}$ the following properties hold:
	\begin{enumerate}[noitemsep]
		\item $\Cor_{K_{nc}/K_{mc}}(c_n) = a_{\ell}(f) \cdot c_m$;
		\item $\loc_{\ell}(c_n) = \Res_{K_{mc,\lambda}/K_{nc,\lambda}}(\loc_{\ell}(c_m)^{\frob_{\ell}})$;
		\item if $\chi^2 = 1$ then $c_n^{\tau} = w_f \cdot \chi(\sigma) \cdot c_n^{\sigma}$ for some $\sigma \in \Gal(K_{nc}/K)$,
	\end{enumerate}
	where $w_f \in \left\lbrace \pm 1 \right\rbrace $ is the Atkin--Lehner eigenvalue of $f$.
\end{defi}

\subsubsection{Kolyvagin's derivative classes}

Define the constant $\beta$ as in \cite[(7.2)]{CH}. For any integer $M\geq1$, denote by
$\mathcal{S}_M \subseteq \mathcal{S}$ the set of square-free products of primes $\ell$ such that
\begin{enumerate}[noitemsep]
	\item $\ell$ is inert in $K$;
	\item $\ell \nmid 2c_0Np$;
	\item $\varpi^M \mid \ell+1, a_{\ell}(f)$;
	\item $\varpi^{M+\beta+1} \nmid \ell + 1 \pm a_{\ell}(f) \ell^{1-k/2}$.
\end{enumerate}
A prime number satisfying all these conditions is called an \emph{$M$-admissible (Kolyvagin) prime}. By using the \v{C}ebotarev density theorem, it can be checked that there exist infinitely many $M$-admissible primes.

Put $G_n := \Gal(K_n/K_1) \cong \Gal(K_{nc}/K_c) \subseteq \mathcal{G}_{nc}$. Let $n \in \mathcal{S}_M$; since $n$ is square-free, there is a splitting $ G_n = \prod_{\ell \mid n} G_{\ell}$. Moreover, each $\ell\,|\,n$ is inert in $K$, so the group $G_{\ell} \cong \Gal(K_{\ell}/K_1)$ is cyclic of order $\ell + 1$. Fix a generator $\sigma_{\ell}$ for each $G_{\ell}$ and put
\begin{displaymath}
\begin{split}
D_{\ell} &:= \sum_{i=1}^{\ell} i \sigma_{\ell}^i \in \mathbb{Z}[G_{\ell}],\\
D_n &:= \prod_{\ell \mid n} D_{\ell} \in \mathbb{Z}[G_n] \subseteq \mathcal{O}_F[\mathcal{G}_{nc}].
\end{split}
\end{displaymath}
Now we choose a positive integer $M'$ such that $p^{M'}$ annihilates
\begin{enumerate}[noitemsep]
	\item the kernel and the cokernel of $\Res_{K/K_n} : H^1(K,T_M \otimes \chi^{-1}) \rightarrow H^1(K_n,T_M \otimes \chi^{-1})^{\mathcal{G}_n}$ for all $n,M \in \mathbb{Z}_+$;
	\item the local cohomology groups $H^1(K_v,T_M \otimes \chi^{-1})$ for all places $v \mid c_0N$.
\end{enumerate}
One can prove that such an integer exists as in \cite[Proposition 6.3, Corollary 6.4 and Lemma 10.1]{Nek92}. 
Consider an anticyclotomic Euler system $\mathbf{c} = \left\lbrace c_n \right\rbrace$ for $T$ and $\chi$. Denote by $\Red_M$ the reduction $H^1(-,T \otimes \chi^{-1}) \rightarrow H^1(-, T_M \otimes \chi^{-1})$. 
For $n \in \mathcal{S}_M$, we want to apply the derivative operators $D_n$ to the classes $c_n$.
For each $n \in \mathcal{S}_M$ there is a unique
class $\mathcal{D}_M(n) \in H^1(K_c,T_M \otimes \chi^{-1})$ such that
\begin{displaymath}
\Res_{K/K_{n}}(\mathcal{D}_M(n)) = p^{3M'} \Red_M(D_n c_n),
\end{displaymath}
because of the properties of $M'$. 
Define the derivative class by
\begin{equation} \label{P1}
P_{M,\chi^{-1}}(n) := \Cor_{K_c/K}(\mathcal{D}_M(n)) \in H^1(K,T_M \otimes \chi^{-1}).
\end{equation}

\subsubsection{Local conditions}
Now we introduce Selmer groups imposing local conditions at $p$ at the cohomology classes, through the choices of subspaces of 
$H^1(K_{\mathfrak{p}},V_f(k/2) \otimes \chi^{-1}) \oplus H^1(K_{\overline{\mathfrak{p}}},V_f(k/2) \otimes \chi^{-1})$.
Recall that $ p = \mathfrak{p} \overline{\mathfrak{p}}$ splits in $K$.

Let $\mathcal{F} \subseteq 
H^1(K_{\mathfrak{p}},V_f(k/2) \otimes \chi^{-1}) \oplus H^1(K_{\overline{\mathfrak{p}}},V_f(k/2) \otimes \chi^{-1})$
be an $F$-subspace and let $\mathcal{F}^*\subseteq
H^1(K_{\mathfrak{p}},V_f(k/2) \otimes \chi) \oplus H^1(K_{\overline{\mathfrak{p}}},V_f(k/2) \otimes \chi)$ be the orthogonal complement of $\mathcal{F}$ with respect to the local Tate pairing of equation (\ref{Tate-pairing}). Assume that $\mathcal{F}^*=\mathcal{F}$ if $\chi^2=1$.
Define $\mathcal{F}_T \subseteq  
H^1(K_{\mathfrak{p}},V_f(k/2) \otimes \chi^{-1}) \oplus H^1(K_{\overline{\mathfrak{p}}},V_f(k/2) \otimes \chi^{-1})$
to be the $F$-subspace obtained as the inverse image of $\mathcal{F}$ under the
direct sum of the maps
$H^1(K_{\mathfrak{p}},T \otimes \chi^{-1}) \rightarrow H^1(K_{\mathfrak{p}},V_f(k/2) \otimes \chi^{-1})$, $H^1(K_{\overline{\mathfrak{p}}},T \otimes \chi^{-1}) \rightarrow H^1(K_{\overline{\mathfrak{p}}},V_f(k/2) \otimes \chi^{-1})$
and $\mathcal{F}_M \subseteq 
H^1(K_{\mathfrak{p}},T_M \otimes \chi^{-1}) \oplus H^1(K_{\overline{\mathfrak{p}}},T_M \otimes \chi^{-1})$ 
as the image of $\mathcal{F}_T$ through the reduction map. Put
$Y_M := T_M \otimes \chi^{-1}$. 
Now define
\[
\Sel^{(n)}_{\mathcal{F}}(K,Y_M) := \Large \left\lbrace x \in H^1(K,Y_M) \mid 
\normalsize
\begin{aligned}
\displaystyle
&\loc_v(x) \in H^1_f(K_v,Y_M) && \text{if}\ v \nmid pn\\
&\loc_p \in \mathcal{F}_M && \text{if}\ p \nmid n
\end{aligned}
\Large\right\rbrace,
\]
where $\loc_p = \loc_{\mathfrak{p}} \oplus \loc_{\overline{\mathfrak{p}}}$.
Note that if $p \mid n$ then the choice of $\mathcal{F}_M$ is irrelevant.
If $n=1$ we abbreviate $\Sel_{\mathcal{F}}(K,Y_M) := \Sel^{(1)}_{\mathcal{F}}(K,Y_M)$. Define then
\begin{displaymath}
\Sel_{\mathcal{F}}(K, T \otimes \chi^{-1}) := \varinjlim_M \Sel_\mathcal{F}(K,Y_M).
\end{displaymath}
If $\mathbf{c} = \left\lbrace c_n\right\rbrace _{n \in \mathcal{S}}$ is an anticyclotomic Euler system for $T$ and $\chi$, let
\begin{displaymath}
\mathbf{c}_K := \Cor_{K_c/K}(c_1) \in H^1(K,T \otimes \chi^{-1}).
\end{displaymath}
Then
\begin{displaymath}
P_{M,\chi^{-1}}(1) = p^{3M'} \Red_{M}(\mathbf{c}_K),
\end{displaymath}
since the square
\\
\[
\begin{tikzcd}
H^1(K_c,T \otimes \chi^{-1}) \arrow[r, "\Cor_{K_c/K}"] \arrow[d, "\Red_M"] 
&  H^1(K, T \otimes \chi^{-1}) \arrow[d, "\Red_M"]
\\
H^1(K_c, T_M \otimes \chi^{-1}) \arrow[r, "\Cor_{K_c/K}"] 
&  H^1(K, T_M \otimes \chi^{-1})
\end{tikzcd}
\]
\\
is commutative.
Using \cite[Proposition 10.2]{Nek92} we obtain that 
\[P_{M,\chi^{-1}}(n) \in \Sel_{\mathcal{F}}^{(np)}(K,Y_M).\]
Note that $p\,|\,pn$, so this holds for any $\mathcal{F}$.

\begin{defi}
	An anticyclotomic Euler system $\mathbf{c} = \left\lbrace c_n \right\rbrace_{n \in \mathcal{S}}$ for $T$ and $\chi$ has \textbf{local condition $\mathcal{F}$} if it satisfies
	\begin{itemize}
		\item[4.] $\mathbf{c}_K \in \Sel_{\mathcal{F}}(K,T \otimes \chi^{-1})$ and $\mathbf{c}_K^{\tau} \in \Sel_{\mathcal{F^*}}(K,T \otimes \chi)$, that is $\loc_{p}(\mathbf{c}_K) \in \mathcal{F}_T$ and $\loc_{p}(\mathbf{c}_K^{\tau}) \in \mathcal{F}^*_T$;
		\item[5.] for every $M$ and $n \in \mathcal{S}_M$, one has \[P_{M,\chi^{-1}}(n) \in \Sel_{\mathcal{F}}^{(n)}(K,T_M \otimes \chi^{-1}),\] 
		that is, $\loc_p(P_{M,\chi^{-1}}(n)) \in \mathcal{F}_M$
		(these two conditions are equivalent because $P_{M,\chi^{-1}}(n) \in \Sel_{\mathcal{F}}^{(np)}(K,Y_M)$).
	\end{itemize}
\end{defi}

Now we state an important technical result.

\begin{teor} \label{keyth}
	Let $\mathbf{c} = \left\lbrace c_n \right\rbrace_{n \in \mathcal{S}}$ be an anticyclotomic Euler system for $T$ and $\chi$ with local condition $\mathcal{F}$. If $\mathbf{c}_K \neq 0$, then
	\begin{displaymath}
	\Sel_{\mathcal{F}^*}(K,V \otimes \chi) = F \cdot \mathbf{c}_K^{\tau}.
	\end{displaymath}
\end{teor}

\begin{proof} This is \cite[Theorem 7.3]{CH}. \end{proof}

\subsubsection{Construction of Euler systems for generalized Heegner cycles}

We keep notation and assumptions introduced at the beginning of this section, but now we assume that
\begin{itemize}
	\item $\chi$ has infinity type $(j,-j)$ with $j \geq k/2$;
	\item $f$ is ordinary at $p$.
\end{itemize}
Recall the cohomology class $\boldsymbol{z}_f \in H^1\bigl(K, T \otimes \mathcal{O}_F\llbracket \tilde{\Gamma}_{c} \rrbracket\bigr)$ defined in (\ref{zf}) and consider its $\chi$-specialization $\boldsymbol{z}_f^{\chi} \in H^1(K,T \otimes \chi)$.  
Let us consider for $v = \mathfrak{p},\overline{\mathfrak{p}}$ the subspace $\mathcal{L}_v$ of $H^1(K_v,V \otimes \chi)$ spanned by $\loc_v(\boldsymbol{z}_f^{\chi})$ and put 
\[ 
\mathcal{L}_{v,T} := \mathcal{L}_v \cap H^1(K,T \otimes \chi).
\]
Set $\mathcal{L}^* := \mathcal{L}_{\mathfrak{p}}^* \oplus \mathcal{L}_{\overline{\mathfrak{p}}}^*$.
Choose $M'$ large enough so that $p^{M'}H^1(K_v,T)_{\text{tor}}=0$ for $ v= \mathfrak{p},\overline{\mathfrak{p}}$.
Recall the cohomology classes $\boldsymbol{z}_{m,\alpha} \in H^1(K_{m^{\infty}},T)$ and for $n \in \mathcal{S}$ set
\begin{displaymath}
c_n := \boldsymbol{z}_{cn,\alpha}^{\chi^-1} \in H^1(K_{nc}, T \otimes \chi^{-1}),
\end{displaymath}
where $\boldsymbol{z}_{cn,\alpha}^{\chi^-1}$ is the specialization at $\chi^{-1}$ obtained via the map 
\[
H^1_{\Iw}(K_{ncp^{\infty}},T) \longrightarrow 
H^1(K_{nc},T \otimes \chi^{-1}).
\]
Finally, define
\begin{equation}
\mathbf{c}:= \left\lbrace c_n \right\rbrace_{n \in \mathcal{S}} = \left\lbrace \boldsymbol{z}_{cn,\alpha}^{\chi^-1}\right\rbrace _{n \in \mathcal{S}}.
\end{equation}
We would like to prove that this collection of cohomology classes is an anticyclotomic Euler system with local condition $\mathcal{L}^*$.
\begin{prop} \label{eulsys}
	The collection $\mathbf{c}:= \left\lbrace c_n \right\rbrace_{n \in \mathcal{S}}$ is an anticyclotomic Euler system for $T$ and $\chi$ with local condition $\mathcal{L}^*$. Moreover, $\mathbf{c}_K = \boldsymbol{z}_f^{\chi^{-1}}$.
\end{prop}
\begin{proof}
	If $\mathbf{c}$ is an anticyclotomic Euler system, then
\[
\mathbf{c}_K = \Cor_{K_c/K}(c_1) = \Cor_{K_c/K}(\boldsymbol{z}_{c,\alpha}^{\chi^{-1}}) =\Cor_{K_c/K}(\boldsymbol{z}_{c,\alpha})^{\chi^{-1}} = \boldsymbol{z}_f^{\chi^{-1}}.
\]
	Analogously to (\ref{zchi}), we obtain 
	\begin{displaymath}
	\boldsymbol{z}_{n,\alpha}^{\psi} = \alpha^{-t} \cdot z_{n,\psi}
	\end{displaymath}
	for each non-trivial finite order character $\psi : \Gal(K_{c_0 p^{\infty}}/K_{c_0}) \rightarrow \mathcal{O}_{\mathbb{C}_p}^{\times}$ of conductor $c=c_0p^s$.
	By Propositions \ref{4.4}, \ref{4.6}, \ref{4.7}, we know that
	\begin{enumerate}[noitemsep]
		\item $\Cor_{K_{nc}/K_{mc}}(z_{nc,\psi}) = a_{\ell}(f) \cdot z_{mc,\psi}$;
		\item $\loc_{\ell}(z_{nc,\psi}) = \Res_{K_{mc,\lambda}/K_{nc,\lambda}}(\loc_{\ell}(z_{mc,\psi})^{\frob_{\ell}})$;
		\item $z_{nc,\psi}^{\tau} = w_f \cdot \psi(\sigma) \cdot z_{nc,\psi^{-1}}^{\sigma}$ for some $\sigma \in \Gal(K_{nc}/K)$.
	\end{enumerate}
	Upon taking $\psi=\mathbf{1}$, we deduce these properties for the classes $\boldsymbol{z}_{nc,\alpha}$ and then for $\boldsymbol{z}_{nc,\alpha}^{\chi^{-1}}$ by specializing the relations at $\chi^{-1}$. This proves that $\mathbf{c}$ is an anticyclotomic Euler system. The last thing we need to show is that $\mathbf{c}$ has local condition $\mathcal{L}^*$, which can be checked as in the proof of \cite[Proposition 7.8]{CH}.
\end{proof}

\begin{prop} \label{loc}
	If $\loc_\mathfrak{p}\Bigl(\boldsymbol{z}_f^{\chi^{-1}}\Bigr) \neq 0$, then $\Sel\bigl(K,V_f(k/2) \otimes \chi\bigr) =0$.
\end{prop}
\begin{proof}
	The proof is completely analogous to that of \cite[Theorem 7.9]{CH}, so we will briefly sketch the arguments. For each choice of subspaces $\mathcal{F}_v \subset H^1\bigl(K_v,V_f(k/2) \otimes \chi\bigr)$ with $v = \mathfrak{p}, \overline{\mathfrak{p}}$, consider the ''generalized Selmer group'' given by
	{
	\begin{displaymath}
H^1_{\mathcal{F}_{\mathfrak{p}}, \mathcal{F}_{\overline{\mathfrak{p}}}}\bigl(K, V_f(k/2) \otimes \chi\bigr) := 
	\Large \left\lbrace x \in H^1(K, V_f(k/2) \otimes \chi) \;\,\bigg|\;\, 
	\normalsize
	\begin{aligned}
	&\loc_v(x) \in H^1_f(K_v, V_f(k/2) \otimes \chi) && \text{if}\ v \nmid p\\
	&\loc_v(x) \in \mathcal{F}_v && \text{if}\ v \,|\, p
	\end{aligned}
	\Large \right\rbrace.
	\end{displaymath}}
	\!\!Thanks to Proposition \ref{eulsys}, we know that $\mathbf{c}$ is an anticyclotomic Euler system for $T$ and $\chi$ with local condition $\mathcal{L}^*$ such that $\mathbf{c}_K = \boldsymbol{z}_f^{\chi^{-1}}$. Since $\loc_{\mathfrak{p}}(\boldsymbol{z}_f^{\chi^{-1}}) \neq 0$, it follows that $\boldsymbol{z}_f^{\chi^{-1}} \neq 0$, so Theorem \ref{keyth} ensures that \[H^1_{\mathcal{L}_{\mathfrak{p}},\mathcal{L}_{\overline{\mathfrak{p}}}}\bigl(K, V_f(k/2) \otimes \chi\bigr) = \Sel_{\mathcal{L}}\bigl(K,V_f(k/2) \otimes \chi\bigr) = F \cdot (\boldsymbol{z}_f^{\chi^{-1}})^{\tau} = F \cdot \boldsymbol{z}_f^{\chi}.\]
	We have that 
	\begin{displaymath}
	H^1_{\mathcal{L}_{\mathfrak{p}},0}\bigl(K, V_f(k/2) \otimes \chi\bigr) \subseteq H^1_{\mathcal{L}_{\mathfrak{p}},\mathcal{L}_{\overline{\mathfrak{p}}}}\bigl(K, V_f(k/2) \otimes \chi\bigr) = F \cdot \boldsymbol{z}_f^{\chi}
	\end{displaymath}
	Since $ \loc_{\overline{\mathfrak{p}}}(\boldsymbol{z}_f^{\chi})^{\tau} = \loc_{\mathfrak{p}}(\mathbf{z}_f^{\chi^{-1}})$, also $\loc_{\overline{\mathfrak{p}}}(\boldsymbol{z}_f^{\chi})^{\tau} \neq 0$, hence $\boldsymbol
	{z}_f^{\chi} \notin H^1_{\mathcal{L}_{\mathfrak{p}},0}(K, V_f(k/2) \otimes \chi)$. It follows that 
	\begin{displaymath}
	H^1_{\mathcal{L}_{\mathfrak{p}},0}(K, V_f(k/2) \otimes \chi) = 0.
	\end{displaymath}
	There is a Poitou--Tate exact sequence
	\begin{displaymath}
	\begin{split}
	0 &\longrightarrow H^1_{0,\emptyset}(K,V_f(k/2)\otimes \chi^{-1}) \longrightarrow H^1_{\mathcal{L}_{\mathfrak{p}}^*,\emptyset}(K,V_f(k/2)\otimes \chi^{-1}) \xrightarrow{\loc_{\mathfrak{p}}} \mathcal{L}_{\mathfrak{p}}^*\\ &\longrightarrow H^1_{\emptyset, 0}(K,V_f(k/2) \otimes \chi)^{\vee} \longrightarrow H^1_{\mathcal{L}_{\mathfrak{p}},0}(K,V_f(k/2)\otimes \chi)^{\vee} \longrightarrow 0,
	\end{split}
	\end{displaymath}
	where $\emptyset$ indicates that we are imposing no condition.
	Then
	$H^1_{\emptyset,0}\bigl(K,V_f(k/2) \otimes \chi\bigr) = 0$.
Since, by \cite[(6.2)]{CH}, there is an equality
	\begin{displaymath}
	H^1_f\bigl(K_v,V_{f}(k/2) \otimes \chi\bigr) = \begin{cases}
	H^1\bigl(K_v,V_{f}(k/2) \otimes \chi\bigr) &\text{if $v= \mathfrak{p}$}\\[2mm]
	\left\lbrace 0 \right\rbrace &\text{if $v =\overline{\mathfrak{p}}$},
	\end{cases}
	\end{displaymath}
we conclude that $\Sel\bigl(K,V_f(k/2) \otimes \chi\bigr) = H^1_{\emptyset,0}\bigl(K,V_f(k/2) \otimes \chi\bigr) = 0$.
\end{proof}

Now we construct another anticyclotomic Euler system associated with our generalized Heegner cycles. In the remainder of this subsection, we assume that
\begin{itemize}
	\item $\chi$ has infinity type $(j,-j)$ with $-k/2 < j < k/2$.
\end{itemize}
Notice that this time we do not need to assume that $f$ is ordinary at $p$.
Recall the cohomology classes $z_{n,\chi^{-1}} \in H^1(K_n, T \otimes \chi^{-1})$ defined in (\ref{chicomp}). 
For $n \in \mathcal{S}$, set
\begin{displaymath}
c'_n := z_{cn,\chi^-1} \in H^1(K_{nc}, T \otimes \chi^{-1}),
\end{displaymath}
and define
\begin{equation}
\mathbf{c}':= \left\lbrace c'_n \right\rbrace_{n \in \mathcal{S}} = \left\lbrace z_{cn,\chi^-1}\right\rbrace _{n \in \mathcal{S}}.
\end{equation}
Denote by $\mathcal{L}_{BK}$ the direct sum of the Bloch--Kato finite subspaces 
\[
\mathcal{L}_{BK} := H^1_f\bigl(K_{\mathfrak{p}},V_f(k/2) \otimes \chi^{-1}\bigr) \oplus H^1_f\bigl(K_{\overline{\mathfrak{p}}},V_f(k/2) \otimes \chi^{-1}\bigr).
\]
We would like to prove that this collection of cohomology classes is an anticyclotomic Euler system with local condition $\mathcal{L}_{BK}$.

\begin{prop} \label{c'eulsys}
	The collection $\mathbf{c}':= \left\lbrace c'_n \right\rbrace_{n \in \mathcal{S}}$ is an anticyclotomic Euler system for $T$ and $\chi$ with local condition $\mathcal{L}_{BK}$. Moreover, $\mathbf{c'}_K = z_{\chi^{-1}}$.
\end{prop}

\begin{proof}
	If $\mathbf{c}'$ is an anticyclotomic Euler system, then \eqref{z.chi} implies that
	\[
	\mathbf{c}'_K = \Cor_{K_c/K}(c'_1) = \Cor_{K_c/K}(z_{c,\chi^{-1}}) = z_{\chi^{-1}}.
	\]
By Propositions \ref{4.4}, \ref{4.6}, \ref{4.7}, we know that
	\begin{enumerate}[noitemsep]
		\item $\Cor_{K_{nc}/K_{mc}}(z_{nc,\chi^{-1}}) = a_{\ell}(f) \cdot z_{mc,\chi^{-1}}$;
		\item $\loc_{\ell}(z_{nc,\chi^{-1}}) = \Res_{K_{mc,\lambda}/K_{nc,\lambda}}(\loc_{\ell}(z_{mc,\chi^{-1}})^{\frob_{\ell}})$;
		\item $z_{nc,\chi^{-1}}^{\tau} = w_f \cdot \psi(\sigma) \cdot z_{nc,\chi}^{\sigma}$ for some $\sigma \in \Gal(K_{nc}/K)$.
	\end{enumerate}
It follows that $\mathbf{c}'$ is an anticyclotomic Euler system.
	
The last thing we need to show is that $\mathbf{c}'$ has local condition $\mathcal{L}_{BK}$. In analogy to what was remarked at the beginning of \S \ref{algL}, the results in \cite{Nek00} ensure that 
\[
z_{nc,\chi^{-1}} \in \Sel(K_{nc},T \otimes \chi^{-1}),\quad z_{\chi^{-1}} \in \Sel(K,T \otimes \chi^{-1}) = \Sel_{\mathcal{L}_{BK}}(K,T \otimes \chi^{-1}).
\]
Since $\tau$ induces an isomorphism 
\[
H^1_f(K_{\mathfrak{p}},T \otimes \chi^{-1}) \oplus H^1_f(K_{\overline{\mathfrak{p}}},T \otimes \chi^{-1}) \cong H^1_f(K_{\mathfrak{p}},T \otimes \chi) \oplus H^1_f(K_{\overline{\mathfrak{p}}},T \otimes \chi),
\]
we deduce that also ${\mathbf{c}'_K}^{\tau} = z_{\chi^{-1}}^{\tau} \in \Sel_{\mathcal{L}_{BK}^*}(K,T \otimes \chi)$. Furthermore, one has
\[
\loc_v\Bigl(\Res_{K_c/K_{nc}}\bigl(\mathcal{D}_M(n)\bigr)\Bigr) = \loc_v\Bigl(p^{3M'} \Red_M\bigl(D_n z_{nc,\chi^{-1}}\bigr)\Bigr) \in H^1_f(K_{nc,v}, T_M \otimes \chi^{-1})
\]
for $v = \mathfrak{p}, \overline{\mathfrak{p}}$. In light of \cite[Lemma 7.5]{CH}, it follows that 
\[
\loc_v\bigl(\mathcal{D}_M(n)\bigr)\in H^1_f(K_{c,v},T_M \otimes \chi^{-1}),
\]
which implies that $\loc_p\bigl(P_{M,\chi^{-1}}(n)\bigr) \in \mathcal{L}_{BK,M}$.	
\end{proof}

\subsection{Results on Selmer groups} \label{Sel}

In this final section, we use the anticyclotomic Euler system method to deduce results on the Selmer group of $V_{f,\chi}$. First of all, we recall notation and assumptions. 

As usual, $f \in S^{\text{new}}_{k}(\Gamma_0(N))$ is our newform of weight $k=2r+2 \geq 4$ and level $N$, $F$ is a finite
extension of $\mathbb{Q}_p$ containing the Fourier coefficients of $f$, $\chi : \Gal(K_{c_0p^{\infty}}/K) \rightarrow \mathcal{O}_F^{\times}$ is a locally algebraic anticyclotomic character of infinity type $(j, -j)$ and conductor $c_0p^s\mathcal{O}_K$ with $(c_0,pN)=1$, $V_{f,\chi} := V_f(k/2)_{{|}_{G_K}} \otimes \chi$ is the twist of $V_{f}(k/2)$ by $\chi$, $L(f,\chi,s)$ is the associated Rankin $L$-series and $\Sel(K,V_{f,\chi})$ is the Block--Kato Selmer group of $V_{f,\chi}$ over $K$. Assume that:
\begin{itemize}[noitemsep]
	\item[1.] $p \nmid 2N \phi(N^+)$;
	\item[2.] $c_0$ is prime to $N$;
	\item[3.] either $D_K > 3$ is odd or $8 \mid D_K$; 
	\item[4.] $p = \mathfrak{p} \overline{\mathfrak{p}}$ splits in $K$;
	\item[5.] $N = N^+N^-$ where $N^+$ is a product of primes that split in $K$ and $N^-$ is a square free product of an \emph{even} number of primes that are inert in $K$.
\end{itemize}
As before, the last condition can be expressed by saying that $K$ satisfies a \emph{generalized Heegner hypothesis} relative to $N$.

Recall now the definition of the Bloch--Kato Selmer group.
If $v$ is a place of $K$ such that $v \nmid p$, we consider the inertia group $I_{K_v} \subseteq G_{K_v}$. The unramified subgroup of $H^1(K_v,V_{f,\chi})$ is defined by
\[
H^1_{\text{ur}}(K_v,V_{f,\chi}) := \ker \Bigl( H^1(K_v,V_{f,\chi}) \longrightarrow H^1(I_{K_v}, V_{f,\chi})\Bigr).
\]
Set $H^1_f(K_v,V_{f,\chi}):= H^1_{\text{ur}}(K_v,V_{f,\chi})$.
If $v$ is a place of $K$ such that $v\,|\,p$, then we set
\[
H^1_{f}(K_v,V_{f,\chi}) := \ker \Bigl( H^1(K_v,V_{f,\chi}) \longrightarrow H^1(K_v, V_{f,\chi} \otimes_{\mathbb{Q}_p}\mathbf{B}_{\cris})\Bigr).
\]
The {\bf global Bloch--Kato Selmer group} $\Sel(K,V_{f,\chi})$ of $V_{f,\chi}$ over $K$ is the subspace of $H^1(K, V_{f,\chi})$ given by
	\begin{displaymath}
	\Sel(K, V_{f,\chi}) := 
	\Large \left\lbrace x \in H^1(K, V_{f,\chi} ) \;\,\bigg|\;\,
	\normalsize
	\begin{aligned}
	&\loc_v(x) \in H^1_{\text{ur}}(K_v,V_{f,\chi}) && \text{if}\ v \nmid p\\
	&\loc_v(x) \in  H^1_{f}(K_v,V_{f,\chi}) && \text{if}\ v \mid p
	\end{aligned}
	\Large \right\rbrace.
	\end{displaymath}
Now we can prove our theorems on Selmer groups, the first of which is a vanishing result.
 
\begin{teor}
\label{selmer-1}
	Suppose that $f$ is $p$-ordinary. If $L(f,\chi,k/2)\neq0$, then
	$
	\Sel(K,V_{f,\chi}) = 0.
	$
\end{teor}

\begin{proof}
Let $\epsilon(V_{f,\chi})\in\{\pm1\}$ be the sign of the functional equation for $L(f,\chi,s)$. Then
\[
\epsilon(V_{f,\chi}) = -1\; \Longleftrightarrow\; -k/2 < j < k/2.
\]
Indeed, $\epsilon(V_{f,\chi})$ is a product 
\begin{displaymath}
\epsilon(V_{f,\chi}) = \prod_v \epsilon(1/2, \pi_{K_v} \otimes \chi_v)
\end{displaymath}
of local signs, where $\pi_K$ is the base change to $K$ of the automorphic representation of $\GL_2(\mathbb{A}_{\mathbb{Q}})$ associated with $f$, $\pi_{K_v}$ are the local factors of $\pi_K$ and the local $\epsilon$-factors $\epsilon(1/2, \pi_{K_v} \otimes \chi_v)$ are defined as follows.
If $F$ is a finite extension of $\mathbb{Q}_{\ell}$ and $\pi'$ is an irreducible representation of $GL_n(F)$, then
\begin{displaymath}
\epsilon(s,\pi') := \epsilon(s,\pi', \psi_F),
\end{displaymath}
where $\psi_F= \psi \circ \text{Tr}_{F/\mathbb{Q}_{\ell}}$
is the standard additive character.
(see, e.g., \cite{Sch02} for the definition). It follows that 
\[\epsilon(V_{f,\chi}) = \prod_v \epsilon(1/2, \pi_{K_v} \otimes \chi_v, \psi_{K_v}).\]
Each of these local factors is equal either to $1$ or to $-1$. Because $N^-$ is an even product of inert primes, the product of the local factors at the finite places is equal to $1$. Thus, the global $\epsilon$-factor depends only on the infinite part.
Furthermore, by \cite[(3.2.5)]{Tate}, one has
\[
\epsilon(\frac{1}{2},\pi_{\infty} \otimes \chi_{\infty}, \psi_{K_{\infty}}) = \epsilon(\frac{1}{2},\mu^{(\frac{k}{2}-\frac{1}{2}+j)},\psi_{K_{\infty}})\epsilon(\frac{1}{2},\mu^{(-\frac{k}{2}+\frac{1}{2}+j)},\psi_{K_{\infty}}) = i^{\mid k-1+2j\mid + \mid 1-k+2j \mid},
\]
where $\mu : z \in \mathbb{C}^{\times} \mapsto \frac{z}{\overline{z}} \in \mathbb{C}^{\times}$.
Hence, one can check that
\begin{displaymath}
\begin{split}
\epsilon(V_{f,\chi}) = -1 &\;\Longleftrightarrow\; -k/2 < j < k/2,\\[1mm]
\epsilon(V_{f,\chi}) = +1 &\;\Longleftrightarrow\; j \leq -k/2 \ \text{or}\ j \geq k/2.
\end{split}
\end{displaymath}
Since $L(f,\chi,k/2) \neq 0$, we know that $\epsilon(V_{f,\chi}) = +1$, therefore either $ j \leq -k/2$ or $j \geq k/2$.
As before, let $\tau$ be complex conjugation; set $\chi^{\tau}(g) := \chi(\tau g\tau)$. There is an equality of $L$-functions $L(f,\chi,k/2) = L(f,\chi^{\tau},k/2)$ and the action of $\tau$ induces an isomorphism $\Sel(K,V_{f,\chi}) \cong \Sel(K,V_{f,\chi^{\tau}})$. This shows that we can assume $j \geq k/2$.

Finally, since $L(f,\chi,k/2) \neq 0$, by Theorem \ref{rec} we know that $\loc_{\mathfrak{p}}\bigl(\boldsymbol{z}_f^{\chi^{-1}}\bigr) \neq 0$, and then Proposition \ref{loc} gives $\Sel(K,V_{f,\chi})=0$.
\end{proof}

Our second theorem on Selmer groups gives a one-dimensionality result.

\begin{teor}
\label{selmer-2}
	If $\epsilon(V_{f,\chi})=-1$ and $z_{\chi} \neq 0$, then
	$
	\Sel(K,V_{f,\chi}) = F z_{\chi}.
	$
\end{teor}

\begin{proof}
Since $z_{\chi} \neq 0$, also $z_{\chi^{-1}} \neq 0$. Because $\epsilon(V_{f,\chi})=-1$, we know that $-k/2 < j < k/2$. Then, by Proposition \ref{c'eulsys}, the collection $\mathbf{c}'$ is an anticyclotomic Euler system for $T$ and $\chi$ with local condition $\mathcal{L}_{BK}$. Since $\mathbf{c}'_K=z_{\chi^{-1}}$, applying Theorem \ref{keyth}, we obtain that $\Sel(K,V_{f,\chi}) = \Sel_{\mathcal{L}_{BK}^*}(K,V_{f,\chi}) = F z_{\chi}$.
\end{proof}

\end{document}